\documentclass[a4paper,11pt]{article}
\usepackage[utf8]{inputenc}
\usepackage[T1]{fontenc}
\usepackage{amsthm}

\usepackage[ english]{babel}
\usepackage{amsmath}
\usepackage{amsfonts}
\usepackage{enumerate}
\usepackage{amssymb}
\usepackage{xcolor}

\newtheorem{theo}{Theorem}[section]
\newtheorem{lemma}[theo]{Lemma}
\newtheorem{coro}[theo]{Corollary}
\newtheorem{prop}[theo]{Proposition}
\newtheorem{rmk}[theo]{Remark}
\newtheorem{defi}[theo]{Definition}

\newcommand{\R}{\mathbb{R}}
\numberwithin{equation}{section}

\newcommand{\curl}{\mathbf{curl}}
\usepackage{graphicx}
\usepackage{graphics}
\usepackage{epsfig}
\usepackage{amsmath}
\usepackage{ mathrsfs }
\usepackage{amsfonts}
\usepackage{psfrag}
\usepackage{subfigure}
\usepackage[a4paper]{geometry}

\title{
 Semi-group theory for the Stokes operator with Navier-type boundary conditions on $L^{p}$-spaces}
\author{Hind Al Baba\ \and Ch\'erif Amrouche\
\and Miguel Escobedo}

\begin{document}

\maketitle

\begin{abstract} In this article we consider the Stokes problem with Navier-type boundary conditions on a domain $\Omega$, not necessarily simply connected. Since under these conditions the Stokes problem has a non trivial kernel, we also study the solutions lying in the orthogonal of that kernel. We prove the analyticity of several  semigroups generated by the Stokes operator  considered in different functional spaces. We  obtain strong, weak and very weak solutions for the time dependent Stokes problem with the Navier-type boundary condition under different hypothesis on the initial data $\boldsymbol{u}_0$ and external force $\boldsymbol{f}$. Then, we study the fractional and pure imaginary powers of several operators related with our Stokes operators. Using the fractional powers,  we prove maximal  regularity results for the homogeneous Stokes problem. On the other hand, using the boundedness of the pure imaginary powers we deduce  maximal $L^{p}-L^{q}$ regularity for the inhomogeneous Stokes problem.
\end{abstract}

\renewcommand{\thefootnote}{}
\footnotetext{\hspace*{-.51cm}
}
\tableofcontents
\section{Introduction}
We consider in a bounded cylindrical domain, $\Omega\times(0,T)$  the linearised evolution Navier-Stokes problem  
\begin{equation}\label{lens}
 \left\{
\begin{array}{cccc}
\frac{\partial\boldsymbol{u}}{\partial t} - \Delta \boldsymbol{u
}\,+\,\nabla\pi=\boldsymbol{f},& 
\mathrm{div}\,\boldsymbol{u}= 0 &\qquad \textrm{in} &\Omega\times (0,T), \\
&\boldsymbol{u}(0)=\boldsymbol{u}_{0}& \qquad \textrm{in} &
\Omega,
\end{array}
\right.
\end{equation}
where $\Omega $ is a bounded domain of  ${\mathbb{R}}^3$, not necessarily simply connected, whose  boundary $\Gamma $ is of class $C^{2,1}$.
Problem \eqref{lens} describes the motion of a viscous incompressible fluid in $\Omega$. The velocity of motion is denoted by $\boldsymbol{u}$ and the associated pressure by $\pi$.
 Given data are the external force $\boldsymbol{f}$ and   the initial velocity $\boldsymbol{u}_{0}$.
 
 Stokes and Navier-Stokes equations are often studied with Dirichlet boundary conditions $$\boldsymbol{u}=\boldsymbol{0}\qquad \textrm{on}\,\,\, \Gamma$$ when the boundary $\Gamma$ represents a fixed wall. This condition was formulated by G. Stokes \cite{Stokes} in 1845, but as stated in \cite{Serrin} this condition is not always realistic since it doesn't reflect necessarily the behaviour of the fluid on or near the boundary.   
 
Even before, H. Navier \cite{Navier} suggested in 1827 alternative boundary conditions more precisely a type of slip boundary conditions with friction on the wall based on a proportionality between the tangential components of the normal dynamic tensor and the velocity
\begin{equation}\label{Navierboundcond}
\boldsymbol{u}\cdot\boldsymbol{n}=0,\qquad
2\,\nu\left[ \mathbb{D}\boldsymbol{u}\cdot\boldsymbol{n}\right]_{\boldsymbol{\tau}}+\alpha\,\boldsymbol{u}_{\boldsymbol{\tau}}=0\qquad\textrm{on}\,\,\,\Gamma\times(0,T),
\end{equation}
where $\nu$ is the viscosity and $\alpha\geq 0$ is the coefficient of friction and $\mathbb{D}\boldsymbol{u}=\frac{1}{2}(\nabla\boldsymbol{u}+\nabla\boldsymbol{u}^{T})$ denotes the deformation tensor associated to the velocity field $\boldsymbol{u}$.  These Navier boundary conditions allows the fluid to slip and measure the friction on the wall. Observe that, formally, if $\alpha$ tends to infinity, the tangential component of the velocity will vanish and we recover the non slip boundary condition $\boldsymbol{u}=\boldsymbol{0}$ on $\Gamma$. 

An interesting particular arises  when the coefficient of friction $\alpha$ is zero.  This corresponds to a Navier-slip boundary condition without friction. This condition has been considered in particular in the mathematical literature on flows near rough walls \cite{Ami, fei, fei2, Bul, Jag, Jag2}.  We also mention that in the case of flat boundary and when $\alpha=0$ the second condition in \eqref{Navierboundcond} can be replaced by another boundary condition involving the vorticity 
\begin{equation}\label{nbc}
\boldsymbol{u}\cdot\boldsymbol{n}=0,\qquad
\boldsymbol{\mathrm{curl}}\,\boldsymbol{u}\times \boldsymbol{n} = \boldsymbol{0}\,\,\, \textrm{on} \qquad \Gamma\times (0,T).
\end{equation}
We call them Navier-type boundary conditions. For  a discussion on the No-Slip boundary condition in the physics literature we refer to \cite{Lau} and the references therein.

The relation between Navier conditions on rough boundary and the Dirichlet boundary condition boundary is studied by Casado in \cite{casado, casado2}.

In this paper we study the Stokes operator with the Navier-type boundary conditions (\ref{nbc}).
Our goal is to obtain a  semi-group theory  for the Stokes operator with Navier-type boundary conditions  as it already exists for other boundary conditions like Dirichlet and Robin. For instance K. Abe \& Y. Giga \cite{{GiGa0}}, W. Borchers \& T. Miyakawa \cite{Bor, Bor2}, R. Farwig \& H. Sohr \cite{FS},  Y. Giga \cite{GiGa1, GiGa2},  Y. Giga \& H. Sohr \cite{GiGa3, GiGa4}, J.  Saal \cite{Saal}, Y. Shibata \& R. Shimada \cite{Shibata1}, V. A. Solonnikov \cite{Solonnikov, Solonnikov2, Solonnikov3}). 

In what follows, if we do not state otherwise, $\Omega$ will be considered as an open bounded domain of $\mathbb{R}^{3}$ of class $C^{2,1}$. In some situation we suppose that $\Omega$ is of class $C^{1,1}$ in the case where the regularity is sufficient for the proof. Then a unit normal vector to the boundary can be defined almost everywhere it will be denoted by $\boldsymbol{n}$. The generic point in $\Omega$ is denoted by
$\boldsymbol{x}\,=\,(x_{1},\,x_{2},\,x_{3})$. 

We do not assume that $\Omega$ is simply-connected but we suppose
that it satisfies the following condition:

\textbf {  Condition H}: there exist $J$ connected open surfaces $\Sigma_{j}$, $1 \leq
j \leq J$, called ``cuts'', contained in $\Omega$, such that each
surface $\Sigma_{j}$ is an open subset of a smooth manifold, the
boundary of $\Sigma_{j}$ is contained in $\Gamma$. The
intersection $\overline{\Sigma}_{i}\cap\overline{\Sigma}_{j}$ is
empty for $i\neq j$ and finally the open set
$\Omega^{\circ}=\Omega\backslash\cup_{j=1}^{J}\Sigma_{j}$ is
simply connected and pseudo-$C^{1,1}$.

 For this Condition H see \cite{Am2} for instance. We denote by $\left[\cdot\right]_{j} $ the jump of a function over $\Sigma_{j}$, \textit{i.e.} the difference of the traces for $1\leq j \leq J$. (See figure below).

\begin{center}
\begin{figure}
\includegraphics[width=0.75\linewidth, height=5cm]{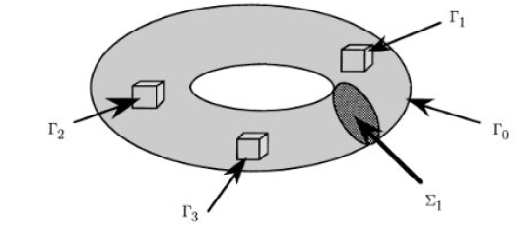}
\caption{The domain $\Omega$}
\label{fig.1}
\end{figure}
\end{center}
\subsection{Stokes problem with flux.}
When $\Omega $ is not simply-connected, the Stokes  operator with boundary condition \eqref{nbc} has a non trivial kernel $\boldsymbol{K}_{\tau}(\Omega)$ contained in all the $\boldsymbol{L}^{r}$ spaces for $r\in (1, \infty)$. This kernel is independent of $r$, it has been proved to be of finite dimension $J\geq 1$ (cf. \cite{Am2}, for $p=2$ and \cite{Am4} for $p\in (1, \infty)$) and it is spanned by the function $\widetilde{\boldsymbol{\mathrm{grad}}}\,q^{\tau}_{j}\,$, $1\leq j\leq J,\,$  where the function $\widetilde{\boldsymbol{\mathrm{grad}}}\,q^{\tau}_{j},$ $1\leq j\leq J$ is the extension by continuity of $\boldsymbol{\mathrm{grad}}\,q^{\tau}_{j}$ to $\Omega$ and  $q_{j}^{\tau}$ is the unique solution up to an additive
 constant to Problem \eqref{gradqj} below. On the other hand, it was proved in  \cite{Am4}, see also \cite{Albaba1}, that when $\Omega $ satisfies Condition H,  then,  for any function $\boldsymbol{u}\in \boldsymbol{L}^p(\Omega )$, divergent free and such that $\boldsymbol{u}\cdot \boldsymbol{n}=0$ on $\Gamma $,  to satisfy
\begin{equation}
\forall\,\,\boldsymbol{v}\in\boldsymbol{K}_{\tau}(\Omega),\quad\int_{\Omega}\boldsymbol{f}\cdot\overline{\boldsymbol{v}}\,\textrm{d}\,x=0,
\end{equation} 
 is equivalent to the condition
\begin{equation}
\langle\boldsymbol{u}(t)\cdot\boldsymbol{n}\,,\,1\rangle_{\Sigma_{j}}\,=\,0,\quad 1\leq
j\leq J,\quad0\leq t\leq\infty,\label{condition2}
\end{equation}
with $\langle\cdot\,.\,\cdot\rangle_{\Sigma_{j}}$ the duality product between $\boldsymbol{W}^{-\frac{1}{p},p}(\Sigma_{j})\,$ and $\,\boldsymbol{W}^{1-\frac{1}{p'},p'}(\Sigma_{j})$. 

\medskip

We will refer to the  problem \eqref{lens}, \eqref{nbc}, \eqref{condition2} as Stokes problem with flux condition. 
By the equivalence mentioned above, the addition of the  extra boundary condition   \eqref{condition2} makes  the Stokes operator invertible with bounded and compact inverse on the space of $L^p$ functions that are divergent free and satisfy $\boldsymbol{u}\cdot \boldsymbol{n}=0$ on $\Gamma $.

\subsection{Three types of solutions: strong, weak and very weak.}
All along this paper we are  interested in three different types of solutions for each of the two problems  \eqref{lens}, \eqref{nbc} and  \eqref{lens}, \eqref{nbc}, \eqref{condition2} defined above. The first, that we call strong solutions, are solutions $\boldsymbol{u}$ that belong to $L^p(0, T, \boldsymbol{L}^q(\Omega))$ type spaces. The second, called weak solutions, are solutions (in a suitable sense)  $\boldsymbol{u}(t)$ that may be writen for a.e. $t>0$, as $\boldsymbol{u}(t)=\boldsymbol{v}(t)+\nabla w(t)$ where 
$\boldsymbol{v}(t)\in L^p(0, T;  \boldsymbol{L}^{q}(\Omega)))$ and $w\in L^p(0, T; L^{q}(\Omega) )$. The third and last, called very weak, are  solutions $\boldsymbol{u}(t)$ that may be decomposed as before but where now $w\in L^p(0, T; W^{-1,\, q}(\Omega) )$.

The concept of very weak solutions was introduced by Lions and Magenes in \cite{LM}. Later on, Amann considered this type of solutions in a series of articles \cite{Amann1, Amann2} in the setting of Besov spaces. More recently this concept was modified by R. Farwig, G.P. Galdi and H. Sohr in \cite{FGS1, FGS, FGS2}, R. Farwig and H. Kozono in \cite{FKS1}, R. Farwig and H. Sohr  in \cite{ FS1} and G.P. Galdi and CHR. Simader in \cite{Galdi} to a setting in classical $L^{p}$-spaces. This concept has also been generalized by K. Schumacher \cite{Sch} to a setting in a weighted Lebesgue and Bessel potential spaces using arbitrary Muckenhoupt weights. The concept of very weak solutions is strongly based on duality arguments for strong solutions. Therefore the boundary regularity required in this theory is the same as for strong solutions.

\subsection{Analytic semigroups.}
In that general setting,  we study first the existence of analytic semigroups generated by  the  Stokes operators, defined on different functional spaces both for the  problem \eqref{lens}, \eqref{nbc} and for  \eqref{lens}, \eqref{nbc}, \eqref{condition2}. 

On the one hand, we consider  Stokes operators defined on the three different spaces $\boldsymbol{L}^p _{ \sigma , \tau  }(\Omega )$, $[\boldsymbol{H}^{p'}_{0}(\mathrm{div},\Omega)]'_{\sigma,\tau}$ and $[\boldsymbol{T}^{p'}(\Omega)]'_{\sigma,\tau}$ (cf. Section \ref{Functional framework} for precise definitions of these spaces).  They  lead respectively to  some strong, weak and very weak solutions of \eqref{lens}, \eqref{nbc}.  Similarly, we consider three Stokes operators with flux, defined respectively on  $\boldsymbol{X}_{p}$, $\boldsymbol{Y}_{p}$ and $\boldsymbol{Z}_{p}$ (cf. \eqref{Xp},  \eqref{Yp}, \eqref{Zp}) in Section  \ref{Stokes operator with flux boundary  conditions}), that   lead  to several solutions of \eqref{lens}, \eqref{nbc}, \eqref{condition2}. 

In the first  main result of this work,  we prove that each of these six operators generates an analytic semigroup on the corresponding functional space. More precisely:
\begin{theo}\label{theo1}\hfill \break
(i) The Stokes operators with Navier-type boundary conditions, $A_p$, $B_p$ and $C_p$,  generate a bounded analytic semi-group on  $\boldsymbol{L}^{p}_{\sigma,\tau}(\Omega)$, $[\boldsymbol{H}^{p'}_{0}(\mathrm{div},\Omega)]'_{\sigma,\tau}$ and $[\boldsymbol{T}^{p'}(\Omega)]'_{\sigma,\tau}$ respectively  for all $1<p<\infty$. \\
(ii) The Stokes operator with Navier-type boundary conditions and flux condition $A'_p$, $B'_p$ and $C'_p$ generate a bounded analytic semi-group on
$\boldsymbol{X}_{p}$, $\boldsymbol{Y}_{p}$ and $\boldsymbol{Z}_{p}$ respectively, for all $1<p<\infty$.
\end{theo}

The proof of  Theorem \ref{theo1} uses a classical approach and starts with the study of  the resolvent of the Stokes operator and Stokes operator with flux conditions, both with boundary conditions (\ref{nbc}). A key observation is that the Stokes operator with Navier-type boundary conditions, with and without flux conditions are equal to the Laplace operator with Navier-type boundary conditions. 

For this reason the study of the Stokes operator is reduced  to that  of the three operators denoted $A_p$, $B_p$ and $C_p$, defined on the  spaces $\boldsymbol{L}^p _{ \sigma , \tau  }(\Omega )$, $[\boldsymbol{H}^{p'}_{0}(\mathrm{div},\Omega)]'_{\sigma,\tau}$ and $[\boldsymbol{T}^{p'}(\Omega)]'_{\sigma,\tau}$ and whose resolvent sets are  given by the solutions  of the system
\begin{equation}\label{*}
\left\{
\begin{array}{r@{~}c@{~}l}
\lambda \boldsymbol{u} - \Delta \boldsymbol{u}\,=\, \boldsymbol{f}, &&\mathrm{div}\,\boldsymbol{u} = 0 \,\,\, \qquad\qquad \mathrm{in} \,\,\, \Omega, \\
\boldsymbol{u}\cdot \boldsymbol{n} = 0, && \boldsymbol{\mathrm{curl}}\,\boldsymbol{u}\times\boldsymbol{n}=\boldsymbol{0}\qquad
\mathrm{on}\,\,\, \Gamma,
\end{array}
\right.
\end{equation}
where $\lambda\in\mathbb{C}^{\ast}$ such that $\mathrm{Re}\,\lambda\geq 0$ and $f$ belonging respectively  to  $\boldsymbol{L}^p _{ \sigma , \tau  }(\Omega )$, $[\boldsymbol{H}^{p'}_{0}(\mathrm{div},\Omega)]'_{\sigma,\tau}$ and $[\boldsymbol{T}^{p'}(\Omega)]'_{\sigma,\tau}$.
Similarly, the problem for the Stokes operator with flux conditions is reduced  to the study of the three operators denoted $A'_p$, $B'_p$ and $C'_p$, defined respectively on  $\boldsymbol{X}_{p}$, $\boldsymbol{Y}_{p}$ and $\boldsymbol{Z}_{p}$ and whose resolvent sets are  given by the solutions  of the problem:
\begin{equation}\label{*BIS}
\left\{
\begin{array}{r@{~}c@{~}l}
\lambda \boldsymbol{u} - \Delta \boldsymbol{u}\,=\, \boldsymbol{f}, &&\mathrm{div}\,\boldsymbol{u} = 0 \,\,\, \qquad\qquad \mathrm{in} \,\,\, \Omega, \\
\boldsymbol{u}\cdot \boldsymbol{n} = 0, && \boldsymbol{\mathrm{curl}}\,\boldsymbol{u}\times\boldsymbol{n}=\boldsymbol{0}\qquad
\mathrm{on}\,\,\, \Gamma,\\
\langle\boldsymbol{u}\cdot\boldsymbol{n}, 1\rangle_{\Sigma_{j}}\,=\,0,&& 1\leq
j\leq J.
\end{array}
\right.
\end{equation}
 where $\lambda\in\mathbb{C}^{\ast}$ such that $\mathrm{Re}\,\lambda\geq 0$ .

 We prove the existence of strong solutions of \eqref{*} satisfying the resolvent estimate
\begin{equation}\label{**}
\Vert\boldsymbol{u}\Vert_{\boldsymbol{L}^{p}(\Omega)}\,\leq\,
\frac{C(\Omega,p)}{\vert\lambda\vert} \,\Vert\boldsymbol{f}\Vert_{\boldsymbol{L}^{p}(\Omega)}.
\end{equation}
For $p=2$ one has estimate (\ref{**}) in a sector $\lambda\in\Sigma_{\varepsilon}$ for a fixed $\varepsilon\in\left] 0,\pi\right[ $ where:
\begin{equation*}
  \Sigma_{\varepsilon}=\{\lambda\in\mathbb{C}^{\ast};\,\,\vert\arg\lambda\vert\leq\pi-\varepsilon\}.
  \end{equation*}

We also show the existence of weak  and very weak solutions  and prove  estimates like  (\ref{**}) for the norms of $[\boldsymbol{H}^{p'}_{0}(\mathrm{div},\Omega)]'_{\sigma,\tau}$ and $[\boldsymbol{T}^{p'}(\Omega)]'_{\sigma,\tau}$.   We obtain similar results for the operators $A'_p$, $B'_p$ and $C'_p$.

There exists several results in the literature,  on the \textit{analyticity of the Stokes semi-group with Dirichlet boundary condition in $L^{p}$-spaces}. This question was already studied by  V. A. Solonnikov in \cite{Solonnikov}. In that work, the author proves the resolvent estimate (\ref{**}) for $\vert\arg\lambda\vert\leq\delta+\pi/2$ where $\delta\geq0$ is small. To derive this estimate \cite{Solonnikov} follows an idea of Sobolevskii \cite{Sobolevskii2} (see the proof \cite[Theorem 5.2]{Solonnikov}). New proofs and extension of the result of \cite{Solonnikov} have been proved by Giga \cite{GiGa1}, Sohr and Farwig \cite{FS} and others. 

In bounded domains the resolvent of the Stokes operator with Dirichlet boundary condition has been studied by  Giga  in \cite{GiGa1}. Using the theory of pseudo-differential operators, the results in \cite{GiGa1} extends those in  \cite{Solonnikov} in two directions. First,  the resolvent estimate \eqref{**} is proved for larger set of values of $\lambda$. More precisely  the  estimate \eqref{**} is  proved in \cite{GiGa1} for all $\lambda$ in the sector $\Sigma_{\varepsilon}$ for any $\varepsilon>0$. Second,  in \cite{GiGa1} the resolvent of the Stokes operator is obtained explicitly and this enables him to describe the domains of fractional powers of the Stokes operator with Dirichlet boundary condition. 

In exterior domains, Giga and Sohr \cite{GiGa3} approximate the resolvent of the Stokes operator with Dirichlet boundary condition with the resolvent of the Stokes operator in the entire space to prove this analyticity. 

 Later on, Farwig and Sohr \cite{FS} investigated the resolvent of the Stokes operator with Dirichlet boundary conditions  when $\mathrm{div}\,\boldsymbol{u}\neq 0$ in $\Omega$.  Their results include bounded and unbounded domains, for the whole and the half space the proof rests on multiplier technique. The problem is also investigated for bended half spaces and for cones by using perturbation criterion and referring to the half space problem.

 More recently, the analyticity of the Stokes semi-group with Dirichlet boundary condition is studied in spaces of bounded functions by Abe and Giga \cite{GiGa0} using a different approach. One of the keys to prove their result is the estimate:
 \begin{equation*}
\Vert N(\boldsymbol{u},\pi)\Vert_{\boldsymbol{L}^{\infty}(\Omega\times\left] 0,T_{0}\right[ )}\leq C\,\Vert\boldsymbol{u}_{0}\Vert_{\boldsymbol{L}^{\infty}(\Omega)}
\end{equation*}
where:
\begin{equation*}
N(\boldsymbol{u},\pi)(x,t)= \vert\boldsymbol{u}(x,t)\vert\,+\,t^{1/2}\,\vert\nabla\boldsymbol{u}(x,t)\vert\,+\,t\,\vert\nabla^{2}\boldsymbol{u}(x,t)\vert\,+\,t\,\vert\partial_{t}\boldsymbol{u}(x,t)\vert\,+\,\vert\nabla\pi(x,t)\vert.
\end{equation*} 

This estimate is obtained by means of a blow-up argument, often used in the study of non linear elliptic and parabolic equations.

The resolvent of the Stokes operator is also studied with \textit{Robin boundary conditions} by Saal \cite{Saal}, Shibata and Shimada  \cite{Shibata1}. In \cite{Saal}, Saal proved that the Stokes operator with Robin boundary conditions is sectorial and admits an $\boldsymbol{H}^{\infty}$-calculus on $\boldsymbol{L}^{p}_{\sigma,\tau}(\mathbb{R}^{3}_{+})$. The strategy for proving these results is firstly to construct an explicit solution for the associated Stokes resolvent problem. Next, the required resolvent estimates to conclude that such an operator is sectorial are obtained by using the rotation invariance in $(n-1)$-dimensions  of large parts of the constructed solution formula, followed by using the known bounded $\boldsymbol{H}^{\infty}$-calculus for the Poisson operator $(-\Delta_{\mathbb{R}^{2}})^{1/2}$ on $\boldsymbol{L}^{p}(\mathbb{R}^{2})$ and performing further computations. Shibata and Shimada proved in \cite{Shibata1} a generalized resolvent estimate for the Stokes equ
 ations with non-homogeneous Robin boundary conditions and divergence condition in $\boldsymbol{L}^{p}$-framework in a bounded or exterior domain by extending the argument of Farwig and Shor \cite{FS}. So that, their approaches in \cite{Shibata1} is different  from Saal \cite{Saal} and rather close to that in \cite{FS}.   

Concerning the \textit{Navier-type boundary conditions}, Mitrea and Monniaux \cite{Mi1} have studied the resolvent of the Stokes operator with Navier-type boundary conditions in Lipschitz domains and proved estimate (\ref{**}) using differential forms on Lipschitz sub-domains of a smooth compact Riemannian manifold. In addition,  when the boundary of $\Omega$ is sufficiently smooth, estimates of type (\ref{**}) are proved using that the boundary conditions (\ref{nbc}) are regular elliptic (e.g. \cite{tay}) and the so called  ``Agmon trick'' (e.g. \cite{Ag}). Moreover, when the domain $\Omega$ is of class $C^{\infty}$, \cite{Miya} shows that the Laplacian with the Navier-type boundary conditions (\ref{nbc}) on $\boldsymbol{L}^{p}(\Omega)$ leaves the space $\boldsymbol{L}^{p}_{\sigma,\tau}(\Omega)$ invariant and hence generates a holomorphic semi-group on $\boldsymbol{L}^{p}_{\sigma,\tau}(\Omega)$. In \cite{Geissert} the authors proved that the Stokes operator with Navier-type 
 boundary
  conditions admits a bounded $\mathcal{H}^{\infty}$-calculus in the case where the domain $\Omega$ is simply connected and this has many consequences in the associated parabolic problem. In \cite{Albaba} the authors proved the analyticity of the semi-group generated by the Stokes operator with these boundary conditions on $\boldsymbol{L}^{p}_{\sigma,\tau}(\Omega)$. For this reason they established estimate (\ref{**}) using a formula involving the boundary conditions (\ref{nbc}) and  that, for every $p\geq2$ and for every $\boldsymbol{u}\in\boldsymbol{W}^{1,p}(\Omega)$ such that
$\Delta\boldsymbol{u}\in\boldsymbol{L}^{p}(\Omega)$ one has
\begin{multline*}\label{calcul1}
-\,\int_{\Omega}|\boldsymbol{u}|^{p-2}\Delta\boldsymbol{u}\cdot\overline{\boldsymbol{u}}\,\mathrm{d}\,x\,=\,\int_{\Omega}|\boldsymbol{u}|^{p-2}|\nabla\boldsymbol{u}|^{2}\,\mathrm{d}\,x\,+\,4\,\frac{p-2}{p^{2}}\,\int_{\Omega}\Big|\nabla|\boldsymbol{u}|^{p/2}\Big|^{2}\,\mathrm{d}\,x\\
\,+\,(p-2)\,i\sum_{k=1}^{3}\,\int_{\Omega}|\boldsymbol{u}|^{p-4}\,\mathrm{Re}\,\Big(\frac{\partial\,\boldsymbol{u}}{\partial
x_{k}}\cdot\overline{\boldsymbol{u}}\Big)\mathrm{Im}\,\Big(\frac{\partial\,\boldsymbol{u}}{\partial
x_{k}}\cdot\overline{\boldsymbol{u}}\Big)\,\mathrm{d}\,x\,-\,\Big\langle\frac{\partial\,\boldsymbol{u}}{\partial\boldsymbol{n}},\,|\boldsymbol{u}|^{p-2}\boldsymbol{u}\Big\rangle_{\Gamma},
\end{multline*}
where $\langle.\,,\,.\rangle_{\Gamma}$ is the anti-duality between
$\boldsymbol{W}^{-1/p,p}(\Gamma)$ and
$\boldsymbol{W}^{1/p,p'}(\Gamma)$.

In this paper, we prove this resolvent estimate for the norms of $[\boldsymbol{H}^{p'}_{0}(\mathrm{div},\Omega)]'_{\sigma,\tau}$ and $[\boldsymbol{T}^{p'}(\Omega)]'_{\sigma,\tau}$ using a duality argument. The next step after establishing the analyticity of the semi-group is to solve the time dependent Stokes Problem \eqref{lens} with the Navier-type boundary conditions (\ref{nbc}).

Existence and uniqueness of solutions for the Stokes system with Navier-type boundary conditions has been proved by Yudovich \cite{Yu} in a two dimensional, simply connected bounded domain. These two-dimensional results are based on the fact that the vorticity is scalar and satisfies the maximum principle. However this technique can not be extended to the three-dimensional case since the standard maximum principle for the vorticity fails. On the other hand Mitrea and Monniaux \cite{Mi2} have employed the Fujita-Kato approach and  proved the existence of a local mild solution to Problem (\ref{lens}) and (\ref{nbc}).
 
 \subsection{Existence, uniqueness and maximal regularity of solutions.}
 
With the analyticity of the different semi-groups in hand we can solve the time dependent Stokes Problem \eqref{lens},  \eqref{nbc} and Stokes Problem with flux condition \eqref{lens}, \eqref{nbc}, \eqref{condition2}. We first deduce of course existence and uniqueness of several types of solutions, using the classical semi group theory. But one of our main goals is also  to  obtain  maximal regularity results in each of these three cases. To this end, following classical arguments (cf. in particular   \cite{GiGa4}), we are led to study the fractional and pure imaginary powers of the operators $I+L$ and $L'$ where $L= A_p, B_p, C_p$ and $L'= A'_p, B'_p, C'_p$.
\subsubsection{The non homogeneous problem.}

Consider first the non-homogeneous problems. When the external force $\boldsymbol{f}$ belongs to $L^{q}(0,T;\,\boldsymbol{L}^{p}_{\sigma,\tau}(\Omega))$ it is known  that the unique solution $\boldsymbol{u}$ of Problem \eqref{lens}, \eqref{nbc} satisfies $\boldsymbol{u}\in C(\left[ 0,\,T\right];\,\boldsymbol{L}^{p}_{\sigma,\tau}(\Omega) )$ for $T<\infty$ (cf. \cite{Pa}). For such $\boldsymbol{f}$ the analyticity of the semi-group is not sufficient to obtain a solution $\boldsymbol{u}$ satisfying what is called the maximal $L^{p}-L^{q}$ regularity property, \textit{i.e.}
\begin{equation*}
\boldsymbol{u}\in L^{q}(0,T;\,\boldsymbol{W}^{2,p}(\Omega)),\qquad \frac{\partial\boldsymbol{u}}{\partial t}\in L^{q}(0,T;\,\boldsymbol{L}^{p}_{\sigma,\tau}(\Omega)).
\end{equation*}
In order to have that property, one possibility is to impose further regularity on $\boldsymbol{f}$, such that local H$\mathrm{\ddot{o}}$lder continuity (see \cite{Pa}). The \textit{maximal $L^{p}$-regularity for the Stokes system with Dirichlet boundary conditions} was first studied by Solonnikov \cite{Solonnikov} when $0<T<\infty$. Solonnikov \cite{Solonnikov} constructed a solution $(\boldsymbol{u},\pi)$ of \eqref{lens} in $\Omega\times\left[ 0,\,T\right)$ satisfying the $L^{p}$ estimate
\begin{multline*}
\int_{0}^{T}\Big\Vert\frac{\partial\boldsymbol{u}}{\partial t}\Big\Vert^{p}_{\boldsymbol{L}^{p}(\Omega)}\,\mathrm{d}\,t\,+\,\int_{0}^{T}\Vert\nabla^{2}\boldsymbol{u}(t)\Vert^{p}_{\boldsymbol{L}^{p}(\Omega)}\,\mathrm{d}\,t\,+\,\int_{0}^{T}\Vert\nabla\pi(t)\Vert^{p}_{\boldsymbol{L}^{p}(\Omega)}\mathrm{d}\,t\\
\leq\,C(T,\Omega,p)\,\int_{0}^{T}\Vert\boldsymbol{f}(t)\Vert^{p}_{\boldsymbol{L}^{p}(\Omega)}\,\mathrm{d}\,t,
\end{multline*}
where the matrix $\nabla^{2}\boldsymbol{u}=(\partial_{i}\partial_{j}\boldsymbol{u})_{i,j=1,2,3}$ is the matrix of the second order derivatives of $\boldsymbol{u}$. When $\Omega$ is not bounded Solonnikov's estimate is not global in time because $C(T,\Omega,p)$ may tend to infinity as $T\rightarrow\infty$ . His approach is based on methods in the theory of potentials. Later on, Giga and Sohr  \cite{GiGa4} strengthened Solonnikov's result in two directions. First their estimate is global in time, \textit{i.e.} the above constant is independent of $T$. Second, the integral norms that they used may have different exponent $p,q$ in space and time. To derive such global $L^{p}-L^{q}$ estimate for the Stokes system with Dirichlet boundary conditions \cite{GiGa4} use the boundedness of the pure imaginary power of the Stokes operator. More precisely they use and extend an abstract perturbation result developed by Dore and Venni \cite{DV}.

Following the same strategy as in \cite{GiGa4} we prove maximal regularity for the inhomogeneous  Stokes problems by studying the pure imaginary powers of  $I+L$ and $L'$ for $L=A_p, B_p, C_p$. Among the earliest works on the boundedness of complex and pure imaginary powers of elliptic operators we refer to the work of R. Seeley \cite{Seeley}. In this work Seeley proved that an elliptic operator $A_{B}$ whose domain is defined by well posed boundary conditions has  bounded complex and imaginary powers in $L^{p}$ satisfying the estimate 
\begin{equation*}
\forall\,\, x\leq 0,\quad\forall\,\,y\in\mathbb{R},\qquad\Vert(A_{B})^{x+iy}\Vert_{\mathcal{L}(L^{p}(\Omega))}\leq\,C_{p}e^{\gamma\vert y\vert},
\end{equation*}
for some constant $C_{p}$ and $\gamma$.

Maximal $L^{p}-L^{q}$ regularity for the Stokes problem with homogeneous \textit{Robin boundary conditions} in $\mathbb{R}^{3}_{+}$ is obtained in
Saal \cite{Saal}  from the boundedness of the pure imaginary powers. However, the same  approach is not applicable to the non-homogeneous boundary condition case and for this reason Shimada \cite{Shimada} didn't follow Saal's arguments. Shimada \cite{Shimada} derive the maximal $L^{q}-L^{p}$ regularity for the Stokes problem with non-homogeneous Robin boundary conditions by applying Weis's operator-valued Fourier multiplier theorem to the concrete representation formulas of solutions to the Stokes problem.

Estimates  of the imaginary powers of the Stokes operator with Dirichlet boundary condition  have been proved in \cite{GiGa2, GiGa3, GiGa4}. That result is proved  in \cite{GiGa2} using the theory of pseudo-differential operators. When $\Omega=\mathbb{R}^{3}$,  this boundedness is proved in  \cite{GiGa3} using  Fourier transform and multiplier theorem. Furthermore, in the case of an exterior domain  the desired estimate is obtained in \cite{GiGa3} by comparing the pure imaginary powers of the Stokes operator with the corresponding powers of the Stokes operator in $\mathbb{R}^{3}$. Finally, in the half space such theorem for the Stokes operator with Dirichlet conditions  is obtained in \cite{GiGa4}  using the results in  \cite{Bor}. 

In our case,  the boundedness of the imaginary powers of $I+L$ and of $I+L'$ with $L=A_p, B_p, C_p$  essentially  follows from previous results in \cite{Geissert}.  The boundedness of the imaginary powers of $A'_p, B'_p, C'_p$  is then obtained using a scaling argument and passage to the limit following  \cite[Theorem A1]{GiGa4}. 

Using these properties  it is then possible to prove the second set of main results of this article, about the existence, uniqueness and maximal  regularity of strong, weak and very weak solutions of the non homogeneous Stokes problem \eqref{lens}, \eqref{nbc} and  Stokes problem with flux \eqref{lens}, \eqref{nbc}, \eqref{condition2}. We only  state in this Introduction  the result for the strong solutions of \eqref{lens}, \eqref{nbc} (cf. Theorem \ref{Existinhsphdiv} and Theorem \ref{Existinhsptp} for the weak and very weak solutions of the Stokes problem and  Theorem \ref{Theorem10},  Remark \ref{rmk10} for the results on the Stokes problem with flux):

\begin{theo}[Strong Solutions for the inhomogeneous Stokes Problem]\label{Exisinhnsplp}
Let $T\in (0, \infty]$, $1<p,q<\infty$, 
$\boldsymbol{f}\in\boldsymbol{L}^{q}(0,T;\boldsymbol{L}^{p}(\Omega))$ and $\boldsymbol{u}_{0}=\boldsymbol{0}$. The Problem (\ref{lens}) with (\ref{nbc}) has a unique solution $(\boldsymbol{u},\pi)$ such that
\begin{equation}\label{reglplqstokes1}
\boldsymbol{u}\in L^{q}(0,T_{0};\,\boldsymbol{W}^{2,p}(\Omega)),\,\,\,\,  T_{0}\leq T\,\,\,\textrm{if}\,\,\,T<\infty\,\,\,\, \mathrm{and }\,\,\,\,T_{0}<T\,\,\,\textrm{if}\,\,\,T=\infty,
\end{equation}
\begin{equation}\label{reglplqstokes2}
\pi\in L^{q}(0,T;\,W^{1,p}(\Omega)/\mathbb{R}),\qquad \frac{\partial\boldsymbol{u}}{\partial t}\in L^{q}(0,T;\,\boldsymbol{L}^{p}(\Omega))
\end{equation}
and
\begin{multline}\label{estlplqstokes}
\int_{0}^{T}\Big\Vert\frac{\partial\boldsymbol{u}}{\partial t}\Big\Vert^{q}_{\boldsymbol{L}^{p}(\Omega)}\,\mathrm{d}\,t\,+\,\int_{0}^{T}\Vert\Delta\boldsymbol{u}(t)\Vert^{q}_{\boldsymbol{L}^{p}(\Omega)}\,\mathrm{d}\,t\,+\,\int_{0}^{T}\Vert\pi(t)\Vert^{q}_{W^{1,p}(\Omega)/\mathbb{R}}\,\mathrm{d}\,t\\
\leq\,C(p,q,\Omega)\,\int_{0}^{T}\Vert\boldsymbol{f}(t)\Vert^{q}_{\boldsymbol{L}^{p}(\Omega)}\,\mathrm{d}\,t.
\end{multline}
\end{theo}

\subsubsection{The homogeneous problem.}

In the homogeneous case, the Stokes Problem \eqref{lens},  \eqref{nbc}  is equivalent to the problem  \eqref{henp}. With an initial data $\boldsymbol{L}^{p} _{ \sigma , \tau  }(\Omega)$, the analyticity of the semi-group  generated by $A_p$ gives a unique solution $u$ of \eqref{henp} satisfying $\boldsymbol{u}\in C^{k}(\left] 0\,,\,\infty\right[,\,\mathbf{D}(A^{\ell}_{p}) )$, for all $k\in\mathbb{N},$ for all $\ell\in\mathbb{N}^{*}$ (see Theorem \ref{exishenp} below).  
This function  is a weak solution of 
  the Stokes problem  \eqref{lens},  \eqref{nbc} and satisfies (cf. Corollary \ref{corolrlpw1p} ) 
  $$\forall\,1\leq q<2,\,\,\forall\,\,T<\infty,\quad\boldsymbol{u}\in L^{q}(0,T;\,\boldsymbol{W}^{1,p}(\Omega))\,\,\,\textrm{and}\quad\frac{\partial\boldsymbol{u}}{\partial t}\in L^{q}(0,T;\,[\boldsymbol{H}^{p'}_{0}(\mathrm{div},\Omega)]').$$

\noindent When the domain $\Omega$ is of class $C^{2,1}$ and the initial data $\boldsymbol{u}_{0}\in\boldsymbol{X}^{p}_{\sigma,\tau}(\Omega)$ given by \eqref{vpt},  the homogeneous Stokes problem  \eqref{lens},  \eqref{nbc} has a unique solution $\boldsymbol{u}(t)$ satisfying
$$\forall\,1\leq q<2,\,\,\,\forall\,\, T<\infty,\quad\boldsymbol{u}\in L^{q}(0,T;\,\boldsymbol{W}^{2,p}(\Omega))\,\,\,\,\textrm{and}\quad\frac{\partial\boldsymbol{u}}{\partial t}\in L^{q}(0,T;\,\boldsymbol{L}^{p}_{\sigma,\tau}(\Omega)).$$ 

\noindent  (see Proposition \ref{StrongSolution}). We also  prove the existence of very weak solution for the homogeneous Stokes Problem \eqref{lens},  \eqref{nbc} when the initial data is less regular and belongs to the dual space $[\boldsymbol{H}^{p'}_{0}(\mathrm{div},\Omega)]'_{\sigma,\tau}$. In this case the solution $\boldsymbol{u}$ satisfy (see Theorem \ref{veryweak})
$$\forall\,1\leq q<2,\,\,\,\forall\,\,T<\infty,\quad\boldsymbol{u}\in L^{q}(0,T;\,\boldsymbol{L}^{p}(\Omega))\,\,\,\,\textrm{and}\quad\frac{\partial\boldsymbol{u}}{\partial t}\in L^{q}(0,T;\,[\boldsymbol{T}^{p'}(\Omega)]'_{\sigma,\tau}).$$

In order to obtain the $L^p-L^q$ estimates for the solution to the homogeneous Stokes problem  \eqref{lens}, \eqref{nbc} with an initial data $\boldsymbol{u}_{0}\in\boldsymbol{L}^{p}_{\sigma,\tau}(\Omega)$ we study the fractional powers of $A'_p$.  We characterize the domain $\mathbf{D}((A'_{p})^{\frac{1}{2}})$ and we prove that $\mathbf{D}((A'_{p})^{\frac{1}{2}})\,=\,\boldsymbol{V}^{p}_{\sigma,\tau}(\Omega)$, where $\boldsymbol{V}^{p}_{\sigma,\tau}(\Omega)$ is given by \eqref{vptflux}. This yields an equivalence of the two norms $\Vert (A'_{p})^{\frac{1}{2}}\boldsymbol{u}\Vert_{\boldsymbol{L}^{p}(\Omega)}$ 
and $\Vert\boldsymbol{\mathrm{curl}}\,\boldsymbol{u}\Vert_{\boldsymbol{L}^{p}(\Omega)}$. We also prove an embedding of Sobolev type for the domain of the fractional powers of the Stokes operator $\mathbf{D}((A'_{p})^{\alpha})$, $\alpha\in\mathbb{R}^{\ast}_{+}$ such that $0<\alpha<3/2p$.

This  is similar to previous results by Borchers and Miyakawa in \cite{Bor, Bor2}, Giga and Sohr \cite{GiGa2, GiGa3, GiGa4} where the fractional powers of the Stokes operator  with Dirichlet boundary conditions $A$  is studied. They have proved that $\mathbf{D}(A^{1/2})\,=\,\boldsymbol{W}^{1,p}_{0}(\Omega)\cap\boldsymbol{L}^{p}_{\sigma}(\Omega)$ and the equivalence of the two norms $\Vert A^{1/2}\boldsymbol{u}\Vert_{\boldsymbol{L}^{p}(\Omega)}$ and $\Vert \nabla\boldsymbol{u}\Vert_{\boldsymbol{L}^{p}(\Omega)}$ for every $\boldsymbol{u}\in\mathbf{D}(A^{1/2})$.They also proved the Sobolev embedding  of the domain $D(A^\alpha )$ into $L^q(\Omega )$  for the Stokes operator with Dirichlet boundary conditions.

Using the fractional powers of $A'_p$ we prove our  third main result:
\begin{theo}\label{theo3}
 Let $1<p\leq q<\infty,$ $\boldsymbol{u}_{0}\in\boldsymbol{L}^{p}_{\sigma,\tau}(\Omega)$ and
\begin{equation}\label{widetildeu0}
\widetilde{\boldsymbol{u}}_{0}=\boldsymbol{u}_{0}\,-\,\boldsymbol{w}_{0},
\end{equation}
\begin{equation}\label{W0introduction}
\boldsymbol{w}_{0}=\sum_{j=1}^{J}\langle\boldsymbol{u}_{0}\cdot\boldsymbol{n}\,,\,1\rangle_{\Sigma_{j}}\widetilde{\boldsymbol{\mathrm{grad}}}\,q_{j}^{\tau}.
\end{equation}
Then the homogeneous Problem
 \eqref{lens},  \eqref{nbc} has a unique solution $\boldsymbol{u}$ satisfying
\begin{equation}\label{Reghenp1A'BIS}
\boldsymbol{u}\in
C([0,\,+\infty[,\,\boldsymbol{L}^{p} _{ \sigma , \tau  }(\Omega ))\cap
C(]0,\,+\infty[,\,\mathbf{D}(A_{p}))\cap
C^{1}(]0,\,+\infty[,\,\boldsymbol{L}^{p} _{ \sigma , \tau  }(\Omega )),
\end{equation}
\begin{equation}\label{Reghenp2A'BIS}
\boldsymbol{u}\in C^{k}(]0,\,+\infty[,\,\mathbf{D}(A^{\ell}_{p})),\qquad
\forall\,k,\,\ell\in\mathbb{N}.
\end{equation}
Moreover, for all $q\in [p, \infty)$, and for all  integers $m,n\in\mathbb{N}$, such that  $m+n> 0$,  there exists constants $M>0$ and $\mu>0$, such that the solution $\boldsymbol{u}$ satisfies the estimates:
\begin{equation}\label{estlplqutxpBIS}
 \Vert\boldsymbol{u}(t)-\boldsymbol{w}_{0}\Vert_{\boldsymbol{L}^{q}(\Omega)}\,\leq\,C\,e^{-\mu \,t}\,t^{-3/2(1/p-1/q)}\Vert \widetilde{\boldsymbol{u}}_{0}\Vert_{\boldsymbol{L}^{p}(\Omega)},
 \end{equation}
 \begin{equation}\label{estlplqcurlutxpBIS}
  \Vert\boldsymbol{\mathrm{curl}}\,\boldsymbol{u}(t)\Vert_{\boldsymbol{L}^{q}(\Omega)}\,\leq\,M\,e^{-\mu t}\,t^{-3/2(1/p-1/q)-1/2}\Vert \widetilde{\boldsymbol{u}}_{0}\Vert_{\boldsymbol{L}^{p}(\Omega)}
 \end{equation}
 and
 \begin{equation}\label{estlplqlaputxpBIS}
 \Big\Vert\frac{\partial^{m}}{\partial t^{m}}\Delta^{n}\boldsymbol{u}(t)\Big\Vert_{\boldsymbol{L}^{q}(\Omega)}\,\leq\,M\,e^{-\mu t}\,t^{-(m+n)-3/2(1/p-1/q)}\Vert \widetilde{\boldsymbol{u}}_{0}\Vert_{\boldsymbol{L}^{p}(\Omega)}.
 \end{equation}  
\end{theo}

The result  for the Stokes problem with flux condition \eqref{lens}, \eqref{nbc},  \eqref{condition2} with an initial data in $\boldsymbol{X}_{p}$ is given in Theorem \ref{exishenpA'}. In order to keep a reasonable   size for this paper, we have not included  the study of  the fractional powers of the operators $B_p, C_p, B'_p$ and $C'_p$. Therefore,  there are  no  regularity results for the weak and very weak solutions for the homogeneous problem. Nevertheless, the general theory of analytic semigroups applied to the semigroups generated by $B_p, C_p, B'_p$ and $C'_p$ provide existence and uniqueness results of solutions to problems  \eqref{lens}, \eqref{nbc} and  \eqref{lens}, \eqref{nbc},  \eqref{condition2} for initial data $\boldsymbol{u}_0$ with less regularity (cf. Theorem \ref{semigroupBC} for the Stokes problem and Theorem \ref{brpimecprime} for the Stokes problem with flux conditions).

 When $\Omega=\mathbb{R}^{3}$, Kato \cite{Ka3} shows that estimate \eqref{estlplqutxpBIS} follows directly from the corresponding estimates for the heat semi-group. In the half space, Borchers and Miyakawa \cite{Bor, Bor2} deduced estimate \eqref{estlplqutxpBIS} for the Stokes semi-group with Dirichlet boundary condition from Ukai's formula \cite{Ukai}. In a bounded domain  Giga \cite{GiGa5} derives this estimate for the Stokes semi-group with Dirichlet boundary conditions from the inequality 
 \begin{equation}
 \label{ineq2}
 \Vert\boldsymbol{u}\Vert_{\boldsymbol{L}^{q}(\Omega)}\leq C\Vert A^{\alpha/2}\boldsymbol{u}\Vert_{\boldsymbol{L}^{p}(\Omega)},\,\,\hbox{with}\,\,\alpha=3(1/p-1/q)
 \end{equation}
 which can be obtained directly from the usual Sobolev inequality for the Laplacian and from the fact that in the case of bounded domains 
 \begin{equation}
 \label{ineq}
 \Vert \Delta^{\alpha/2}\boldsymbol{u}\Vert_{\boldsymbol{L}^{p}(\Omega)}\leq C\Vert A^{\alpha/2}\boldsymbol{u}\Vert_{\boldsymbol{L}^{p}(\Omega)}
 \end{equation}
  for every regular function $\boldsymbol{u}$, for every $\alpha>0$ and for every $1<p<\infty$ (see \cite{GiGa2}). In the case of exterior domains Giga and Sohr follow in \cite{GiGa3} the same procedure as in the case of bounded domains but with limitations with respect to the values of $p$ and $q$, because in this case the inequality 
 \eqref{ineq} still hold true but for limited values $p$ and $q$. We note also that in exterior domain Borchers and Miyakawa \cite{Bor2} prove the same result as \cite{GiGa3} but using \eqref{ineq2}. More recently Coulhon and Lamberton \cite{Coulhon} proved the estimate \eqref{estlplqutxpBIS} by showing that some properties of the Stokes semi-group with Dirichlet boundary condition can be obtained by a simple transfer of the properties of the heat semi-group. 

\subsection{Plan of the paper.}
This paper is organized as follows. In Section \ref{Notations} we give the functional framework and some preliminary results at the basis of our proofs. In Section \ref{operators} we  
define the three different Stokes operators with Navier-type boundary conditions, and prove some of their properties. In Section \ref{semigroups}  we prove that the operators introduced in Section  \ref{operators} generate bounded analytic semi-groups. Section \ref{Stokes operator with flux boundary  conditions} is devoted to Stokes operators with Navier-type boundary conditions and flux conditions.  We introduce three   operators of that kind   and prove that they generate analytic semigroups. We prove in Section \ref{powers} several results on the pure imaginary and fractional powers of several  operators. Then, in Section \ref{time}, we solve the Stokes problem  and the Stokes problem with flux under different assumptions on the initial data $\boldsymbol{u}_0$ and the function $\boldsymbol{f}$.
\section{Notations and preliminary results}  
\label{Notations}
\subsection{Functional framework}
\label{Functional framework}
In this subsection we review some basic notations, definitions and functional framework which are essential in our work.

We do not assume
that the boundary $\Gamma$ is connected and we denote by
$\Gamma_{i},$ $0\leq i\leq I,$ the connected component of
$\Gamma,$ $\Gamma_{0}$ being the boundary of the only unbounded
connected component of
$\mathbb{R}^{3}\backslash\overline{\Omega}$. We also fix a smooth
open set $\vartheta$ with a connected boundary (a ball, for
instance), such that $\overline{\Omega}$ 
 is contained in $\vartheta$, and we denote by $\Omega_{i}$, $0\leq i \leq I$, the connected component of
 $\vartheta\backslash\overline{\Omega}$
 with boundary $\Gamma_{i}$ ($\Gamma_{0}\cup\partial\vartheta$ for $i=0$).
 
We do not assume that $\Omega$ is simply-connected but we suppose
that there exist $J$ connected open surfaces $\Sigma_{j}$, $1 \leq
j \leq J$, called ``cuts'', contained in $\Omega$, such that each
surface $\Sigma_{j}$ is an open subset of a smooth manifold, the
boundary of $\Sigma_{j}$ is contained in $\Gamma$. The
intersection $\overline{\Sigma}_{i}\cap\overline{\Sigma}_{j}$ is
empty for $i\neq j$ and finally the open set
$\Omega^{\circ}=\Omega\backslash\cup_{j=1}^{J}\Sigma_{j}$ is
simply connected and pseudo-$C^{1,1}$ (see \cite{Am2} for instance).

We denote by $\left[\cdot\right]_{j} $ the jump of a function over $\Sigma_{j}$, \textit{i.e.} the difference of the traces for $1\leq j \leq J$. In addition, for any function $q$ in $W^{1,p}(\Omega^{\circ})$, $\boldsymbol{\mathrm{grad}}\,q$ is the gradient of $q$ in the sense of distribution in $\boldsymbol{\mathcal{D}}'(\Omega^{\circ})$, it belongs to $\boldsymbol{L}^{p}(\Omega^{\circ})$ and therefore can be extended to $\boldsymbol{L}^{p}(\Omega)$. In order to distinguish this extension from the gradient of $q$ in $\boldsymbol{\mathcal{D}}'(\Omega^{\circ})$ we denote it by $\widetilde{\boldsymbol{\mathrm{grad}}}\,q$.
 
Finally, vector fields,
matrix fields and their corresponding spaces defined on $\Omega$
will be denoted by bold character. The functions treated here are
complex valued functions. We will use also the symbol $\sigma$ to
represent a set of divergence free functions. In other words If
$\boldsymbol{E}$ is Banach space, then
$$\boldsymbol{E}_{\sigma}\,=\,\big\{\boldsymbol{v}\in\boldsymbol{E};\,\,\mathrm{div}\,\boldsymbol{v}\,=\,0\,\,\,\textrm{in}\,\,\Omega\big\}.$$

Now, we introduce some functional spaces.
Let $\boldsymbol{L}^{p}(\Omega)$ denotes the usual vector valued
$\boldsymbol{L}^{p}$-space over $\Omega$. Let us define the
spaces:
\begin{equation*}
\boldsymbol{H}^{p}(\textrm{\textbf{curl}},\Omega)\,=\,\big\{\boldsymbol{v}\in\boldsymbol{L}^{p}(\Omega);\,\,\,\textrm{\textbf{curl}}\,\boldsymbol{v}\in\boldsymbol{L}^{p}(\Omega)\big\},
\end{equation*}
\begin{equation*}
\boldsymbol{H}^{p}(\textrm{div},\Omega)\,=\,\big\{\boldsymbol{v}\in\boldsymbol{L}^{p}(\Omega);\,\,\,\textrm{div}\,\boldsymbol{v}\in\boldsymbol{L}^{p}(\Omega)\big\},
\end{equation*}
\begin{equation*}
\boldsymbol{X}^{p}(\Omega)\,=\,\boldsymbol{H}^{p}(\textrm{\textbf{curl}},\Omega)\cap\boldsymbol{H}^{p}(\textrm{div},\Omega),
\end{equation*}
equipped with the graph norm.
Thanks to \cite{Am4} we know
that $\boldsymbol{\mathcal{D}}(\overline{\Omega})$ is dense in
$\boldsymbol{H}^{p}(\textrm{\textbf{curl}},\Omega),$
$\boldsymbol{H}^{p}(\textrm{div},\Omega)$ and
$\boldsymbol{X}^{p}(\Omega)$.\\
We also define the subspaces:
\begin{equation*}
\boldsymbol{H}^{p}_{0}(\textrm{\textbf{curl}},\Omega)\,=\,\big\{\boldsymbol{v}\in\boldsymbol{H}^{p}(\textrm{\textbf{curl}},\Omega);\,\,\,\boldsymbol{v}\times\boldsymbol{n}\,=\,\boldsymbol{0}\,\,\textrm{on}\,\,\Gamma\big\},
\end{equation*}
\begin{equation*}
\boldsymbol{H}^{p}_{0}(\textrm{div},\Omega)\,=\,\big\{\boldsymbol{v}\in\boldsymbol{H}^{p}(\textrm{div},\Omega);\,\,\,\boldsymbol{v}\cdot\boldsymbol{n}\,=\,0\,\,\textrm{on}\,\,\Gamma\big\},
\end{equation*}
\begin{equation*}
\boldsymbol{X}^{p}_{N}(\Omega)\,=\,\big\{\boldsymbol{v}\in\boldsymbol{X}^{p}(\Omega);\,\,\,\boldsymbol{v}\times\boldsymbol{n}\,=\,\boldsymbol{0}\,\,\textrm{on}\,\,\Gamma\big\},
\end{equation*}
\begin{equation*}
\boldsymbol{X}^{p}_{\tau}(\Omega)\,=\,\big\{\boldsymbol{v}\in\boldsymbol{X}^{p}(\Omega);\,\,\,\boldsymbol{v}\cdot\boldsymbol{n}\,=\,0\,\,\textrm{on}\,\,\Gamma\big\}
\end{equation*}
and
\begin{equation*}
\boldsymbol{X}^{p}_{0}(\Omega)\,=\,\boldsymbol{X}^{p}_{N}(\Omega)\cap\boldsymbol{X}^{p}_{\tau}(\Omega).
\end{equation*}
We have denoted by $\boldsymbol{v}\times\boldsymbol{n}$
(respectively by $\boldsymbol{v}\cdot\boldsymbol{n}$) the
tangential (respectively normal) boundary value of
$\boldsymbol{v}$ defined in $\boldsymbol{W}^{-1/p,\,p}(\Gamma)$
(respectively in $W^{-1/p,\,p}(\Gamma)$) as soon as
$\boldsymbol{v}$ belongs to
$\boldsymbol{H}^{p}(\textrm{\textbf{curl}},\Omega)$ (respectively
to $\boldsymbol{H}^{p}(\textrm{div},\Omega)$). More precisely, any
function $\boldsymbol{v}$ in
$\boldsymbol{H}^{p}(\textrm{\textbf{curl}},\Omega)$ (respectively
in $\boldsymbol{H}^{p}(\textrm{div},\Omega)$) has a tangential
(respectively normal) trace $\boldsymbol{v}\times\boldsymbol{n}$
 (respectively $\boldsymbol{v}\cdot\boldsymbol{n}$) in
 $\boldsymbol{W}^{-1/p,\,p}(\Gamma)$ (respectively in $W^{-1/p,\,p}(\Gamma)$) defined by:
 \begin{equation}\label{tt}
 \forall\,\boldsymbol{\varphi}\in\boldsymbol{W}^{1,\,p'}(\Omega),\,\,\,\langle\boldsymbol{v}\times\boldsymbol{n},\,\boldsymbol{\varphi}\rangle_{\Gamma}\,=\,\int_{\Omega}\textrm{\textbf{curl}}\,\boldsymbol{v}\cdot\overline{\boldsymbol{\varphi}}\,\textrm{d}\,x\,-\,\int_{\Omega}\boldsymbol{v}\cdot\textrm{\textbf{curl}}\,\overline{\boldsymbol{\varphi}}\,\textrm{d}\,x
 \end{equation}
 and
 \begin{equation}\label{nt}
 \forall\,\varphi\in
 W^{1,\,p'}(\Omega),\,\,\,\langle\boldsymbol{v}\cdot\boldsymbol{n},\,\varphi\rangle_{\Gamma}\,=\,\int_{\Omega}\boldsymbol{v}\cdot\textrm{\textbf{grad}}\,\overline{\varphi}\,\textrm{d}\,x\,+\,\int_{\Omega}\textrm{div}\,\boldsymbol{v}\,\overline{\varphi}\,\textrm{d}\,x,
 \end{equation}
 where $\langle.,.\rangle_{\Gamma}$ is the anti-duality between
 $\boldsymbol{W}^{-1/p,\,p}(\Gamma)$ and
 $\boldsymbol{W}^{1/p,\,p'}(\Gamma)$ in (\ref{tt}) and between
 $W^{-1/p,\,p}(\Gamma)$ and $W^{1/p,\,p'}(\Gamma)$ in
 (\ref{nt}).
 Thanks to \cite{Am4} we know that
 $\boldsymbol{\mathcal{D}}(\Omega)$ is dense in
 $\boldsymbol{H}^{p}_{0}(\textrm{\textbf{curl}},\Omega)$ and in
 $\boldsymbol{H}^{p}_{0}(\textrm{div},\Omega)$.

  We denote by
 $[\boldsymbol{H}^{p}_{0}(\textrm{\textbf{curl}},\Omega)]'$ and
 $[\boldsymbol{H}^{p}_{0}(\textrm{div},\Omega)]'$ the dual spaces
 of $\boldsymbol{H}^{p}_{0}(\textrm{\textbf{curl}},\Omega)$ and
 $\boldsymbol{H}^{p}_{0}(\textrm{div},\Omega)$ respectively.
 
  Notice that we can characterize these dual spaces as follows:
  A distribution $\boldsymbol{f}$ belongs to
 $[\boldsymbol{H}^{p}_{0}(\textrm{\textbf{curl}},\Omega)]'$ if and
 only if there exist functions functions
 $\boldsymbol{\psi}\in\boldsymbol{L}^{p'}(\Omega)$ and
 $\boldsymbol{\xi}\in\boldsymbol{L}^{p'}(\Omega)$, such that
 $\boldsymbol{f}\,=\,\boldsymbol{\psi}\,+\,\boldsymbol{\mathrm{curl}}\,\boldsymbol{\xi}$. Moreover one has
 \begin{equation*}
\|\boldsymbol{f}\|_{[\boldsymbol{H}^{p}_{0}(\textrm{\textbf{curl}},\Omega)]'}\,=\,\inf_{\boldsymbol{f}\,=\,\boldsymbol{\psi}\,+\,\boldsymbol{\mathrm{curl}}\,\boldsymbol{\xi}}\max\,(\|\boldsymbol{\psi}\|_{\boldsymbol{L}^{p'}(\Omega)},\,\|\boldsymbol{\xi}\|_{\boldsymbol{L}^{p'}(\Omega)}).
 \end{equation*}
 Similarly, a distribution $\boldsymbol{f}$ belongs to
$[\boldsymbol{H}^{p}_{0}(\textrm{div},\Omega)]'$
 if and only if there exist
 $\boldsymbol{\psi}\in\boldsymbol{L}^{p'}(\Omega)$ and $\chi\in
 L^{p'}(\Omega)$ such that
 $\boldsymbol{f}\,=\,\boldsymbol{\psi}\,+\,\boldsymbol{\mathrm{grad}}\,\chi$
 and
 \begin{equation*}
\|\boldsymbol{f}\|_{[\boldsymbol{H}^{p}_{0}(\textrm{div},\Omega)]'}\,=\,\inf_{\boldsymbol{f}\,=\,\boldsymbol{\psi}\,+\,\boldsymbol{\mathrm{grad}}\,\chi}\max\,(\|\boldsymbol{\psi}\|_{\boldsymbol{L}^{p'}(\Omega)}\,,\,\|\chi\|_{ L^{p'}(\Omega)}).
 \end{equation*}
 
 Finally we consider the space
\begin{equation}\label{tp}
\boldsymbol{T}^{p}(\Omega)=\big\{\boldsymbol{v}\in\boldsymbol{H}^{p}_{0}(\mathrm{div},\Omega);\,\,\,\mathrm{div}\,\boldsymbol{v}\in
{W}^{1,p}_{0}(\Omega)\big\},
\end{equation}
equipped with the graph norm.
 Thanks to \cite[Lemma 4.11, Lemma 4.12]{Am3} we
know that $\boldsymbol{\mathcal{D}}(\Omega)$ is dense in
$\boldsymbol{T}^{p}(\Omega)$ and a distribution $\boldsymbol{f}\in(\boldsymbol{T}^{p}(\Omega))'$ if and only if there exists a function $\boldsymbol{\psi}\in\boldsymbol{L}^{p'}(\Omega)$ and a function $\chi\in W^{-1,p'}(\Omega)$ such that $\boldsymbol{f}=\boldsymbol{\psi}\,+\,\nabla\chi$. 
 
\subsection{Preliminary results} 
In this subsection, we review some known results which are essential in our work. First, we recall that the vector-valued Laplace operator of a vector field $\textbf{\textit{v}}=(v_1,v_2,v_3)$ is equivalently defined by  
  \begin{equation*}\label{laplacebis}
\Delta\,\textbf{\textit{v}}= \mathbf{grad}\,(\mathrm{div} \, \textbf{\textit{v}})-\mathbf{curl}\, \mathbf{curl}\,\textbf{\textit{v}}.
\end{equation*}

Next, we review some Sobolev embeddings (see \cite{Am4}):
\begin{lemma}\label{con1} The spaces
$\boldsymbol{X}^{p}_{N}(\Omega)$ and
$\boldsymbol{X}^{p}_{\tau}(\Omega)$ defined above are continuously embedded in
$\boldsymbol{W}^{1,p}(\Omega)$.
\end{lemma}
Consider now the spaces
\begin{equation}
\boldsymbol{X}^{2,p}(\Omega)\,=\,\big\{\boldsymbol{v}\in\boldsymbol{L}^{p}(\Omega);\,\,\mathrm{div}\,\boldsymbol{v}\in W^{1,p}(\Omega),\,\,\boldsymbol{\mathrm{curl}}\,\boldsymbol{u}\in\boldsymbol{W}^{1,p}(\Omega)
\,\,\textrm{and}\,\,\boldsymbol{v}\cdot\boldsymbol{n}\in
W^{1-1/p,p}(\Gamma)\big\}
\end{equation}
and
\begin{equation*}
\boldsymbol{Y}^{2,p}(\Omega)\,=\,\big\{\boldsymbol{v}\in\boldsymbol{L}^{p}(\Omega);\,\textrm{div}\,\boldsymbol{v}\in
W^{1,p}(\Omega),\,\boldsymbol{\mathrm{curl}}\,\boldsymbol{v}\in\boldsymbol{W}^{1,p}(\Omega)\,\,\textrm{and}\,\,\boldsymbol{v}\times\boldsymbol{n}\in\boldsymbol{W}^{1-1/p,p}(\Gamma)\big\}.
\end{equation*}
\begin{lemma}\label{con2}
The spaces
$\boldsymbol{X}^{2,p}(\Omega)$ and $\boldsymbol{Y}^{2,p}(\Omega)$
are continuously embedded in $\boldsymbol{W}^{2,p}(\Omega)$.
\end{lemma}
Consider now the space
\begin{equation*}
\textbf{\textit{E}}^{\,p}(\Omega)=\{\textbf{\textit{v}}\in \textbf{\textit{W}}^{\,1,p}(\Omega); \, \Delta \textbf{\textit{v}}\in [\textbf{\textit{H}}^{\,p'}_0(\mathrm{div},\Omega)]'\},
\end{equation*}
which is a Banach space for the norm:
$$\parallel\textbf{\textit{v}}\parallel_{\textbf{\textit{E}}^{\,p}(\Omega)}=\parallel\textbf{\textit{v}}\parallel_{\textbf{\textit{W}}^{\,1,p}(\Omega)}+\parallel\Delta\textbf{\textit{v}}\parallel_{[\textbf{\textit{H}}^{\,p'}_0(\mathrm{div},\Omega)]'}.$$

Thanks to \cite[Lemma 4.1]{Am3} we know that $\boldsymbol{\mathcal{D}}(\overline{\Omega})$ is dense in $\textbf{\textit{E}}^{\,p}(\Omega)$.
  Moreover we have the following Lemma (see \cite[Corollary 4.2]{Am3}):
\begin{lemma}\label{fg1}  
The linear mapping 
$\gamma : \textbf{\textit{v}}\xrightarrow{\hspace*{0.5cm}}\mathbf{curl} \,\textbf{\textit{v}}\times \textbf{\textit{n}}$ defined on $\boldsymbol{\mathcal{D}}(\overline{\Omega})$ can be extended  to a linear and continuous mapping 
 \begin{equation*} \gamma : \textbf{\textit{E}}^{\,p}(\Omega)\xrightarrow{\hspace*{0.8cm}} \textbf{\textit{W}}^{\,-\frac{1}{p},p}(\Gamma).  \end{equation*}
  Moreover, we have the Green formula: for any $\textbf{\textit{v}}\in \textbf{\textit{E}}^{\,p}(\Omega)$ and $\boldsymbol{\varphi} \in \textbf{\textit{X}}^{\,p'}_{\tau}(\Omega)$ such that $\mathrm{div}\,\boldsymbol{\varphi}=0$ in $\Omega$.
  \begin{equation*}\label{formule de Green bis}
  -\left\langle \Delta \textbf{\textit{v}},\boldsymbol{\varphi}\right\rangle_\Omega=\int_{\Omega}\curl \,\textbf{\textit{v}}\cdot\curl\,\boldsymbol{\varphi}\,\mathrm{d}\textbf{\textit{x}}-\langle \curl \textbf{\textit{v}}\times \textbf{\textit{n}}, \boldsymbol{\varphi}\rangle_{\Gamma}.
  \end{equation*}
where $\langle.,.\rangle_{\Gamma}$ denotes the anti-duality between $\textbf{\textit{W}}^{\,-\frac{1}{p},p}(\Gamma)$ and $\textbf{\textit{W}}^{\,\frac{1}{p},p'}(\Gamma)$ and $\langle., .\rangle_\Omega$ denotes the anti-duality between $[\textbf{\textit{H}}^{\,p'}_0(\mathrm{div},\Omega)]'$ and $ \textbf{\textit{H}}^{\,p'}_0(\mathrm{div},\Omega).$
\end{lemma}

Next we consider the space
\begin{equation*}
\boldsymbol{H}^{p}(\Delta,\Omega)=\big\{\boldsymbol{v}\in\boldsymbol{L}^{p}(\Omega);\,\,\,\Delta\boldsymbol{v}\in(\boldsymbol{T}^{p'}(\Omega))'\big\},
\end{equation*}
which is a Banach space for the graph norm.
Thanks to \cite[Lemma 4.13, Lemma 4.14]{Am3} we know that
\begin{prop} The space
$\boldsymbol{\mathcal{D}}(\overline{\Omega})$ is dense in
$\boldsymbol{H}^{p}(\Delta,\Omega)$. Moreover for every
$\boldsymbol{v}$ in $\boldsymbol{H}^{p}(\Delta,\Omega)$ the trace
$\boldsymbol{\mathrm{curl}}\,\boldsymbol{v}\times\boldsymbol{n}$
exists and belongs to $\boldsymbol{W}^{-1-1/p,p}(\Gamma)$.
In addition we have the Green formula: for all
$\boldsymbol{v}\in\boldsymbol{H}^{p}(\Delta,\Omega)$ and for all $\boldsymbol{\varphi}\in\boldsymbol{W}^{2,p}(\Omega)$ such that $\mathrm{div}\,\boldsymbol{\varphi}=\boldsymbol{\varphi}\cdot\boldsymbol{n}=0$ on $\Gamma$ and $\boldsymbol{\mathrm{curl}}\,\boldsymbol{\varphi}\times\boldsymbol{n}=\boldsymbol{0}$ on $\Gamma$:
\begin{equation}\label{greenfrtp2}
\langle\Delta\boldsymbol{v}\,,\,\boldsymbol{\varphi}\rangle_{(\boldsymbol{T}^{p'}(\Omega))'\times\boldsymbol{T}^{p'}(\Omega)}=\int_{\Omega}\boldsymbol{v}\cdot\Delta\overline{\boldsymbol{\varphi}}\,\textrm{d}\,x\,+\,\langle\boldsymbol{\mathrm{curl}}\,\boldsymbol{v}\times\boldsymbol{n}\,,\,\boldsymbol{\varphi}\rangle_{\Gamma},
\end{equation}
where $\langle.\,,\,.\rangle_{\Gamma}=\langle.\,,\,.\rangle_{\boldsymbol{W}^{-1-1/p,p}(\Gamma)\times\boldsymbol{W}^{1+1/p,p'}(\Gamma)}$.
\end{prop}

Next we consider the problem:
\begin{equation}\label{wn.1}
\mathrm{div}\,(\boldsymbol{\mathrm{grad}}\,\pi\,-\,\boldsymbol{f})=0\qquad
\mathrm{in}\,\,\Omega,\qquad
(\boldsymbol{\mathrm{grad}}\,\pi\,-\,\boldsymbol{f})\cdot\boldsymbol{n}=0\qquad\mathrm{on}\,\,\Gamma.
\end{equation} 
We recall the following lemma  concerning the weak Neumann problem without giving the proof.
\begin{lemma}\label{wn1}
\textbf{(i)}  Let $\boldsymbol{f}\in\boldsymbol{L}^{p}(\Omega)$ (see \cite{Si} for instance), the Problem (\ref{wn.1})
has a unique solution $\pi\in W^{1,p}(\Omega)/\mathbb{R}$ satisfying the estimate
\begin{equation*}
\|\boldsymbol{\mathrm{grad}}\,\pi\|_{\boldsymbol{L}^{p}(\Omega)}\,\leq\,C_{1}(\Omega)\,\|\boldsymbol{f}\|_{\boldsymbol{L}^{p}(\Omega)},
\end{equation*}
for some constant $C_{1}(\Omega)>0$.

\textbf{(ii)}  Let $\boldsymbol{f}\in[\boldsymbol{H}^{p'}_{0}(\mathrm{div},\Omega)]'$, the Problem (\ref{wn.1})
 has a unique solution $\pi\in L^{p}(\Omega)/\mathbb{R}$ satisfying the estimate
 \begin{equation*}
 \Vert\pi\Vert_{L^{p}(\Omega)/\mathbb{R}}\,\leq\,C_{2}(\Omega,p)\Vert\boldsymbol{f}\Vert_{[\boldsymbol{H}^{p'}_{0}(\mathrm{div},\Omega)]'}.
 \end{equation*}
 
$\textbf{(iii)}$ Let $\boldsymbol{f}\in(\boldsymbol{T}^{p'}(\Omega))'$, where $\boldsymbol{T}^{p}(\Omega)$ is given by (\ref{tp}). The Problem (\ref{wn.1}) has a unique solution $\pi\in W^{-1,p}(\Omega)/\mathbb{R}$ satisfying the estimate
\begin{equation*}
\Vert\pi\Vert_{W^{-1,p}(\Omega)}\,\leq\,C(\Omega,p)\,\Vert\boldsymbol{f}\Vert_{(\boldsymbol{T}^{p'}(\Omega))'}.
\end{equation*}
\end{lemma}
\subsection{Some Properties of sectorial and non-negative operators}
This subsection is devoted to the definitions and some relevant properties of sectorial and non-negative operators very useful in our work. In all this subsection $X$ denotes a Banach space and $\mathcal{A}: D(\mathcal{A})\subset X\mapsto X$ is a closed linear operator. $D(\mathcal{A})$ is the domain of $\mathcal{A}$, it is equipped with the graph norm and form with this norm a Banach space.

\medskip

Let $0\leq\theta<\pi/2$ and let $\Sigma_{\theta}$ be the sector
\begin{equation*}
\Sigma_{\theta}\,=\,\Big\{\lambda\in\mathbb{C}^{\ast};\,\,\,|\arg\lambda|<\pi-\theta\Big\}.
\end{equation*}
Thanks to \cite[Chapter 2, page 96]{En}, we know that a linear densely defined
operator $\mathcal{A}$ is
sectorial if there exists a constant $M>0$ and $0\leq\theta<\pi/2$ such that
\begin{equation}\label{ç}
\forall\,\lambda\in\Sigma_{\theta},\qquad\|R(\lambda,\,\mathcal{A})\|_{\mathcal{L}(X)}\,\leq\,\frac{M}{|\lambda|},
\end{equation}
where $R(\lambda,\,\mathcal{A})\,=\,(\lambda\,I\,-\,\mathcal{A})^{-1}$. This means that the resolvent of a sectorial operator contains a sector $\Sigma_{\theta}$ for some $0\leq\theta<\pi/2$ and for every $\lambda\in\Sigma_{\theta}$ one has estimate (\ref{ç}). 
 
 \medskip
 
Moreover, the authors give in \cite{En} a necessary and sufficient condition for an operator $\mathcal{A}$ to generates a bounded analytic semi-group. In fact, according to \cite[Chapter 2, Theorem 4.6, page 101]{En}, an operator $\mathcal{A}$ generates a bounded analytic semi-group if and only
if it is sectorial in the sense of \eqref{ç}.  

\medskip

Nevertheless, it is not always easy to prove that an operator $\mathcal{A}$ is sectorial in the sense  \eqref{ç}. Although, Yosida proved in \cite{Yo} that it is suffices to prove (\ref{ç}) in the half plane  $\{\lambda\in\mathbb{C}^{\ast};\,\mathrm{Re}\,\lambda\geq 0\}.\,$ This result is stated in \cite[Chapter 1, Theorem 3.2, page 30]{Ba} and proved by K. Yosida.
\begin{prop}\label{pr2}
Let $\mathcal{A}\,:\,\mathrm{D}(\mathcal{A})\,\subseteq\,X\,\longmapsto\,X$ be a
linear densely defined operator and $M>0$ such that
\begin{equation*}
\forall\,\lambda\in\mathbb{C}^{\ast},\,\,\,\mathrm{Re}\,\lambda\geq 0,\qquad\|R(\lambda,\,\mathcal{A})\|_{\mathcal{L}(X)}\,\leq\,\frac{M}{|\lambda|}.
\end{equation*}
 Then $\mathcal{A}$ is sectorial in the sense of \eqref{ç}.
\end{prop}
\begin{proof} 
Thanks to Yosida \cite[Chapter VIII, Theorem 1, page 211]{Yo} we know that $\rho(\mathcal{A})$ is an open subset of
$\mathbb{C}$ and for all $\lambda_{0}\in\rho(\mathcal{A})$, the disc of center $\lambda_{0}$ and radius $\vert\lambda_{0}\vert/M$ is contained in $\rho(\mathcal{A})$. In particular, for every $r>0,\,$ the open disks with center $\pm\,
i\,r$ and radius $|r|/M$ is contained in $\rho(\mathcal{A})$. The union of
such  disks and of the half plane
$\{\lambda\in\mathbb{C};\,\mathrm{Re}\,\lambda\geq 0\}$ contains
the sector
\begin{equation*}
\Big\{\lambda\in\mathbb{C};\,\, \lambda\neq
0,\,|\arg\lambda|<\pi\,-\,\arctan(M)\Big\},
\end{equation*}
hence it contains the sector
\begin{equation*}
S\,=\,\Big\{\lambda\in\mathbb{C};\,\, \lambda\neq
0,\,|\arg\lambda|<\pi\,-\,\arctan(2\,M)\Big\}.
\end{equation*}
If $\lambda\in S$ and $\textrm{Re}\,\lambda <0,$ we write
$\lambda$ in the form
$\lambda\,=\,\pm\,i\,r\,-\,(\theta\,r)/(2\,M)$ for some
$\theta\in(0,1).$ Thanks to \cite[Chapter 4, formula 1.2, page
239]{En} we know that
$$R(\lambda,\,\mathcal{A})\,=\,R(\pm\,i\,r,\,\mathcal{A})\big[I\,+\,(\lambda\,\mp\,i\,r)R(\pm\,i\,r,\,\mathcal{A})\big]^{-1}.$$
We can easily verify that
$\big\|\big[I\,+\,(\lambda\,\mp\,i\,r)R(\pm\,i\,r,\,\mathcal{A})\big]^{-1}\big\|_{\mathcal{L}(X)}\,\leq\,2.$

Next, observe that
$|\lambda|\,=\,\sqrt{r^{2}\,+\,\frac{\theta^{2}\,r^{2}}{4\,M^{2}}}\,=\,r\,\frac{\sqrt{4\,M^{2}\,+\,\theta^{2}}}{2\,M}$.
Then
\begin{equation*}
\|R(\lambda,\,\mathcal{A})\|\,\leq\,\frac{2\,M}{r}\,\leq\,\frac{2\,M\,\frac{\sqrt{4\,M^{2}\,+\,\theta^{2}}}{2\,M}}{r\,\frac{\sqrt{4\,M^{2}\,+\,\theta^{2}}}{2\,M}}\,\leq\,\frac{\sqrt{4\,M^{2}\,+\,1}}{|\lambda|}.
\end{equation*}
Now if $\lambda\in S$ such that $\textrm{Re}\,\lambda\,\geq\,0$
then thanks to our assumption one has
\begin{equation}\label{esttéta}
\|R(\lambda,\,\mathcal{A})\|_{\mathcal{L}(X)}\,\leq\,\frac{M}{|\lambda|}
\end{equation}
which ends the proof.
\end{proof}
\begin{rmk}\label{rmktétareso}
\rm{Proposition \ref{pr2} means that there exists an angle $0<\theta_{0}<\pi/2$ such that 
the resolvent set of the operator $\mathcal{A}$ contains the sector
\begin{equation*}
\Sigma_{\theta_{0}}=\big\{\lambda\in\mathbb{C};\,\,\,\vert\arg\lambda\vert\leq\pi-\theta_{0}\big\}
\end{equation*}
where estimate (\ref{esttéta}) is satisfied.}
\end{rmk}

Next we recall some definitions and properties concerning the fractional powers of a non-negative operator. We start by the following definition. 
\begin{defi}
An operator $\mathcal{A}$ is said to be a non-negative operator if its resolvent set contains all negative real numbers and
\begin{equation*}
\sup _{t>0}\, t\,\Vert(t\,I+\mathcal{A})^{-1}\Vert_{\mathcal{L}(X)}< \infty.
\end{equation*}
\end{defi}
For a non-negative operator $\mathcal{A}$ it is possible to define its complex power $\mathcal{A}^{z}$ for every $z\in\mathbb{C}$ as a densely defined closed linear operator in the closed subspace $X_{\mathcal{A}}=\overline{D(\mathcal{A})}\cap\overline{R(\mathcal{A})}$ in $X$. Here $D(\mathcal{A})$ and $R(\mathcal{A})$ denote, respectively, the domain and the range of $\mathcal{A}$. Observe that, if both $D(\mathcal{A})$ and $R(\mathcal{A})$ are dense in $X$, then $X_{\mathcal{A}}=X$.  We refer to \cite{Ko, Tri} for the definition and some relevant properties of the complex power of a non-negative operator.

For a non-negative bounded operator whose  inverse $\mathcal{A}^{-1}$ exists and it is bounded (\textit{i.e.} $0\in\rho(\mathcal{A})$), the complex power $\mathcal{A}^{z}$ can be defined for all $z\in\mathbb{C}$ by the means of the Dunford integral (\cite{Yo}):
\begin{equation}\label{forintimpur}
 \mathcal{A}^{z}\,\textit{f}\,=\,\frac{1}{2\,\pi\,i}\,\int_{\Gamma_{\theta}}(-\lambda)^{z}\,(\lambda\,I+\mathcal{A})^{-1}\,\textit{f}\,\mathrm{d}\,\lambda,
 \end{equation}
where $\Gamma_{\theta}$ runs in the resolvent set of $-\mathcal{A}$ from $\infty\,e^{i(\theta-\pi)}$ to zero and from zero to $\infty\,e^{i(\pi-\theta)}$, $0<\theta<\pi/2$ in $\mathbb{C}$ avoiding the non negative real axis. The branch of $(-\lambda)^{z}$ is taken so that $\mathrm{Re}((-\lambda)^{z})>0$ for $\lambda<0$.  It is proved by Triebel \cite{Tri}  that when the operator $\mathcal{A}$ is of bounded inverse, the complex powers $\mathcal{A}^{z}$ for $\mathrm{Re}\,z>0$ are isomorphisms from $D(\mathcal{A}^{z})$ to $X_{\mathcal{A}}$.

The following property plays an important role in the study of the abstract inhomogeneous Cauchy-Problem and give us more regularity for the solutions (see \cite{GiGa4}).
\begin{defi}\label{EthetaK}
Let $\theta\geq 0$ and $K\geq 1$. A non-negative operator $\mathcal{A}$ belongs to $E^{\theta}_{K}(X)$ if $\mathcal{A}^{is}\in\mathcal{L}(X_{\mathcal{A}})$ for every $s\in\mathbb{R}$ and its norm in $\mathcal{L}(X_{\mathcal{A}})$ satisfies the estimate
\begin{equation}\label{estimpur}
\Vert\mathcal{A}^{is}\Vert_{\mathcal{L}(X_{\mathcal{A}})}\leq\,K\,e^{\theta\,\vert s\vert}.
\end{equation}
If in addition $D(\mathcal{A})$ and $R(\mathcal{A})$ are dense in $X$, we say that $\mathcal{A}\in\mathcal{E}^{\theta}_{K}(X)$.
\end{defi}
\noindent We note that, these spaces $E^{\theta}_{K}(X)$ and $\mathcal{E}^{\theta}_{K}(X)$ were introduced by Dore and Venni \cite{DV}, Giga and Sohr \cite{GiGa4} in the  abstract perturbation theory.

When $-\mathcal{A}$ is the infinitesimal generator of a bounded analytic semi-group $(T(t))_{t\geq0}$, the following proposition is proved by Komatsu (see \cite[Theorem 12.1]{Ko} for instance)
\begin{prop}\label{t-alpha}
Let $-\mathcal{A}$ be the infinitesimal generator of a bounded analytic semi-group $(T(t))_{t\geq0}$. For any complex number $\alpha$ such that $\mathrm{Re}\,\alpha>0$ one has
\begin{equation}\label{talpha}
\forall t>0,\qquad\Vert\mathcal{A}^{\alpha}T(t)\Vert_{\mathcal{L}(X)}\leq C\,t^{-\mathrm{Re}\,\alpha}.
\end{equation}
\end{prop}
The following lemma is proved by Komatsu (see \cite{Ko}) and plays an important role in the study of the domains of fractional powers of the Stokes operator with Navier-type boundary conditions.
\begin{lemma}
Let $\mathcal{A}$ be a non-negative closed linear operator. If $\mathrm{Re}\alpha>0$ the domain $D((\nu\,I\,+\,\mathcal{A})^{\alpha})$ doesn't depend on $\nu\geq0$ and coincides with $D((\mu\,I\,+\,\mathcal{A})^{\alpha})$ for $\mu\geq0$. In other words 
\begin{equation*}
\forall\,\mu,\,\nu>0,\qquad D(\mathcal{A}^{\alpha})\,=\,D((\mu\,I\,+\,\mathcal{A})^{\alpha})\,=\,D((\nu\,I\,+\,\mathcal{A})^{\alpha}).
\end{equation*}
\end{lemma}
Finally, let $\mathcal{A}$ be a non-negative operator such that $0\in\rho(\mathcal{A})$. The boundedness of $\mathcal{A}^{is}$, $s\in\mathbb{R}$ allows us to determine the domain of definition of $D(\mathcal{A}^{\alpha})$, for complex number $\alpha$ satisfying $\mathrm{Re}\,\alpha>0$ using complex interpolation. The following result is due to \cite{Tri}
\begin{theo}\label{domfracpower}
Let $\mathcal{A}$ be a non-negative operator with bounded inverse. We suppose that there exist two positive numbers $\varepsilon$ and $C$ such that $\mathcal{A}^{is}$ is bounded for $-\varepsilon\leq s\leq\varepsilon$ and $\Vert\mathcal{A}^{is}\Vert_{\mathcal{L}(X_{A})}\,\leq\,C$. If $\alpha$ is a complex number such that $0<\mathrm{Re}\,\alpha\,<\infty$ and $0<\theta<1$ then
\begin{equation*}
\left[X\,,\,D(\mathcal{A}^{\alpha}) \right] _{\theta}\,=\,D(\mathcal{\mathcal{A}^{\alpha\theta}}).
\end{equation*} 
\end{theo}

\subsection{Some auxiliary results on  $\zeta $-convexity.}

In order to prove maximal $L^p-L^q$ regularity properties for the solutions of the inhomogeneous Stokes problem, we use the property of $\zeta$-convexity of  Banach spaces. This property has already proved to be useful in the same context (cf. \cite{GiGa4}). For further readings on $\zeta$-convex Banach spaces we refer to \cite{Bur3,Rub}.

The $\zeta $ convex property  may be defined as follows:
\begin{defi}\label{zetaconv}
A Banach space $X$ is $\zeta$-convex if there is a symmetric biconvex function $\zeta$ on $X\times X$ such that $\zeta(0,0)>0$ and
\begin{equation}
\forall\,x,y\in X,\,\,\,\Vert x\Vert_{X}\geq 1, \qquad\zeta(x,y)\leq\Vert x+y\Vert_{X}.
\end{equation}
\end{defi}
For this and equivalent definitions see   Theorem 1 and Theorem 2 in  \cite{Rub}.

The  $\zeta$-convexity  property is stronger than  uniform convexity or reflexivity. 
It has been  proved in  Proposition 3 of  \cite{Rub} that for any $\Omega$  open domain of $\mathbb{R}^{3}$ the space $L^{p}(\Omega)$ is $\zeta$-convex if and only if $1<p<\infty$.

The following property of $\zeta$-convex spaces is needed in the following. Since its proof is elementary we shall skip it.
\begin{prop}\label{zetaconvexsubsp}
Every  closed subspace of a  $\zeta$-convex space is $\zeta$-convex.
\end{prop}

On the other hand, the following characterization of $\zeta $-convex spaces in terms of the Hilbert transform is proved in \cite{Bur3} (cf. Theorem 3.3 in Section 3 and Section 2). See also  \cite{Rub} (Theorem 1 and Theorem 2):
\begin{theo}\label{Hilberttransf}
 A Banach space $X$ is $\zeta$-convex if and only if, for some $s\in(1, \infty)$,  the truncated Hilbert transform 
\begin{equation*}
(H_{\varepsilon}\textit{f})(t)\,=\,\frac{1}{\pi}\,\int_{\vert\tau\vert>\varepsilon}\frac{\textit{f}(t-\tau)}{\tau}\,\mathrm{d}\tau\end{equation*}
converges as $\varepsilon\rightarrow 0$, for almost all $t\in\mathbb{R}$, for all $\textit{f}\in L^{s}(\mathbb{R};\,X) $,  and there is a constant $C=C(s,X)$ independent of $\textit{f}$ such that 
\begin{equation*}
\Vert H\textit{f}\,\Vert_{L^{s}(\mathbb{R},\,X)}\,\leq\,C\,\Vert\textit{f}\,\Vert_{L^{s}(\mathbb{R};\,X)},
\end{equation*}
where $(H\textit{f})(t)=\lim _{\varepsilon\rightarrow 0}(H_{\varepsilon}\textit{f})(t)$.
\end{theo} 

\medskip

Using Theorem \ref{Hilberttransf} we prove the  following proposition and show the $\zeta$-convexity of the dual spaces  $[\boldsymbol{H}^{p'}_{0}(\mathrm{div},\Omega)]'$ and $[\boldsymbol{T}^{p'}(\Omega)]'$. 
\begin{prop}\label{Hpdivtp'zetaconx}
Let $1<p<\infty$, the dual spaces $[\boldsymbol{H}^{p'}_{0}(\mathrm{div},\Omega)]'$ and $[\boldsymbol{T}^{p'}(\Omega)]'$ are $\zeta$-convex Banach spaces.
\end{prop}
\begin{proof}
We will only write the proof of the $\zeta$-convexity of $[\boldsymbol{H}^{p'}_{0}(\mathrm{div},\Omega)]'$ because the  proof of the $\zeta$-convexity of $[\boldsymbol{T}^{p'}(\Omega)]'$ is similar. Let $\boldsymbol{f}\in L^{s}(\mathbb{R};\,[\boldsymbol{H}^{p'}_{0}(\mathrm{div},\Omega)]')$, then for almost all $t\in\mathbb{R}$, there exists $\boldsymbol{\psi}(t)\in\boldsymbol{L}^{p}(\Omega)$ and $\chi(t)\in L^{p}(\Omega)$ such that 
\begin{equation*}
\boldsymbol{f}(t)\,=\,\boldsymbol{\psi}(t)\,+\,\nabla\chi(t),\qquad\Vert\boldsymbol{f}(t)\Vert_{[\boldsymbol{H}^{p'}_{0}(\mathrm{div},\Omega)]'}\,=\,\max (\Vert\boldsymbol{\psi}(t)\Vert_{\boldsymbol{L}^{p}(\Omega)},\,\Vert\chi(t)\Vert_{L^{p}(\Omega)}).
\end{equation*}
Since $\boldsymbol{f}\in L^{s}(\mathbb{R};\,[\boldsymbol{H}^{p'}_{0}(\mathrm{div},\Omega)]')$, it is clear that $\boldsymbol{\psi}\in L^{s}(\mathbb{R};\,\boldsymbol{L}^{p}(\Omega))$ and $\chi\in L^{s}(\mathbb{R};L^{p}(\Omega))$. On the other hand we can easily verify that 
\begin{equation*}
(H_{\varepsilon}\boldsymbol{f})(t)\,=\,(H_{\varepsilon}\boldsymbol{\psi})(t)\,+\,\nabla(H_{\varepsilon}\chi)(t).
\end{equation*}
Next, since $\boldsymbol{L}^{p}(\Omega)$ (respectively $L^{p}(\Omega)$) is $\zeta$-convex then $(H_{\varepsilon}\boldsymbol{\psi})(t)$ (respectively $(H_{\varepsilon}\chi)(t)$) converges as $\varepsilon\rightarrow 0$ to $H\boldsymbol{\psi}(t)$ (respectively to $H\chi(t)$). Moreover we have the estimate
\begin{equation*}
\Vert H\boldsymbol{\psi}(t)\Vert_{L^{s}(\mathbb{R};\,\boldsymbol{L}^{p}(\Omega))}\,\leq\,C(s,\Omega,p)\,\Vert\boldsymbol{\psi}\Vert_{L^{s}(\mathbb{R};\,\boldsymbol{L}^{p}(\Omega))}
\end{equation*}
and
\begin{equation*}
\Vert H\chi(t)\Vert_{L^{s}(\mathbb{R};\,L^{p}(\Omega))}\,\leq\,C(s,\Omega,p)\, \Vert\psi\Vert_{L^{s}(\mathbb{R};\,L^{p}(\Omega))}
\end{equation*}
 This means that $(H_{\varepsilon}\boldsymbol{f})(t)$ converges as $\varepsilon\rightarrow0$ to $H\boldsymbol{f}(t)\,=\, H\boldsymbol{\psi}(t)\,+\,\nabla\,H\chi(t)$. Moreover we have the estimate
 \begin{equation*}
\Vert H\boldsymbol{f}(t)\Vert_{L^{s}(\mathbb{R};\,[\boldsymbol{H}^{p'}_{0}(\mathrm{div},\Omega)]')}\,\leq\,C(s,\Omega,p)\,\Vert\boldsymbol{f}\Vert_{L^{s}(\mathbb{R};\,[\boldsymbol{H}^{p'}_{0}(\mathrm{div},\Omega)]')},
\end{equation*}
which ends the proof.
\end{proof}

\section{The Stokes operator}\label{Stokesoperatorsection}
\label{operators}
The main object of this section is to introduce the different Stokes operators with Navier-type boundary conditions that we need in order to solve the Stokes problem for the different types of initial data $\boldsymbol{u}_0$ and external forces $\boldsymbol{f}$ that we want to consider. For the sake of comparison we also recall the definition of
the Stokes operator with Dirichlet boundary conditions.

\subsection{The Stokes operator with
Dirichlet boundary conditions} 
We consider the space
\begin{equation*}
\boldsymbol{L}^{p}_{\sigma,\tau}(\Omega)\,=\,\Big\{\boldsymbol{f}\in\boldsymbol{L}^{p}(\Omega);\,\,\mathrm{div}\,\boldsymbol{f}=0\,\,\,\textrm{in}\,\,\,\Omega,\,\,\,\boldsymbol{f}\cdot\boldsymbol{n}=0\,\,\,\textrm{on}\,\,\,\Gamma\Big\}.
\end{equation*}
Endowed the $L^p(\Omega )$ norm, it is a Banach space. We also define
\begin{equation*}
\boldsymbol{V}^{p}_{0}(\Omega)\,=\,\big\{\boldsymbol{v}\in\boldsymbol{W}^{1,p}_{0}(\Omega);\,\mathrm{div}\,\boldsymbol{v}=0\,\,\,\textrm{in}\,\,\Omega\big\}
\end{equation*}
which is a Banach space for the norm of $\boldsymbol{W}^{1,p}(\Omega)$. For every $\boldsymbol{u}\in\boldsymbol{V}^{p}_{0}(\Omega)$ we define the Stokes operator with Dirichlet boundary condition by
\begin{equation*}
\forall\,\boldsymbol{v}\in\boldsymbol{V}^{p'}_{0}(\Omega),\qquad\langle A\boldsymbol{u}\,,\,\boldsymbol{v}\rangle_{(\boldsymbol{V}^{p'}_{0}(\Omega))'\times\boldsymbol{V}^{p'}_{0}(\Omega)}\,=\,\int_{\Omega}\nabla\boldsymbol{u}:\nabla\overline{\boldsymbol{v}}\,\textrm{d}\,x.
\end{equation*}
Notice that, we can also define the Stokes operator with Dirichlet boundary
condition by $$A\,:\,\mathbf{D}(A)\subset\boldsymbol{L}^{p}_{\sigma,\tau}(\Omega)\longmapsto\boldsymbol{L}^{p}_{\sigma,\tau}(\Omega),$$
where $\mathbf{D}(A)\,=\,
\boldsymbol{W}^{2,p}(\Omega)\cap\boldsymbol{W}^{1,p}_{0}(\Omega)\cap\boldsymbol{L}^{p}_{\sigma}(\Omega)$ and
$A\,=\,-\,P\Delta$. We recall that
\begin{equation}\label{helmholtzproj1}
P:\boldsymbol{L}^{p}(\Omega)\longmapsto\boldsymbol{L}^{p}_{\sigma,\tau}(\Omega)
\end{equation} is the Helmholtz projection
defined by,
\begin{equation}\label{helmholtzproj2}
\forall\quad \boldsymbol{f}\in\boldsymbol{L}^{p}(\Omega),\qquad
P\boldsymbol{f}\,=\,\boldsymbol{f}\,-\,\boldsymbol{\mathrm{grad}}\,\pi,
\end{equation}
where $\pi$ is the unique solution of Problem (\ref{wn.1}). This means that, the Stokes operator is defined by : $$\boldsymbol{u}\in\mathbf{D}(A),\qquad
A\boldsymbol{u}\,=\,-P\Delta\boldsymbol{u}\,=\,-\Delta\boldsymbol{u}\,+\boldsymbol{\mathrm{grad}}\,\pi,$$
where $\pi$ is the unique solution up to an additive constant of
the problem
\begin{equation}\label{stokesdirichlet}
\mathrm{div}(\boldsymbol{\mathrm{grad}}\,\pi\,-\,\Delta\boldsymbol{u})=0\qquad\mathrm{in}\,\,\Omega,\qquad
(\boldsymbol{\mathrm{grad}}\,\pi\,-\,\Delta\boldsymbol{u})\cdot\boldsymbol{n}=0\qquad\textrm{on}\,\,\,\Gamma.
\end{equation}
\subsection{The Stokes operator with Navier-type boundary conditions}\label{stokesoperatorwithnbcsubsection}
In this Section we consider three different Stokes operators with Navier type boundary conditions.

When $\Omega $ is not simply-connected, the Stokes  operator with boundary condition \eqref{nbc} has a non trivial kernel included in all the $L^p$ spaces for $p\in (1, \infty)$. 
It may be caracterised as follows:
\begin{equation}\label{noinailb}
 \boldsymbol{K}_{\tau}(\Omega)\,=\,\big\{\boldsymbol{v}\in\boldsymbol{X}^{p}_{\tau}(\Omega);\,\,\mathrm{div}\,\boldsymbol{v}=0,\,\,\boldsymbol{\mathrm{curl}}\,\boldsymbol{v}=\boldsymbol{0}\,\,\textrm{in}\,\,\Omega\big\}.
 \end{equation}
It has been proved that his kernel  is actually independent of $p$ (cf.  \cite{Am2}, for $p=2$ and \cite{Am4} for $p\in (1, \infty)$), is of finite dimension $J\ge 1$ and spanned by the  functions
 $\widetilde{\boldsymbol{\mathrm{grad}}}\,q^{\tau}_{j}$, $1\leq j\leq
 J$, (see \cite[proposition 3.14]{Am2}). For all $1\leq j\leq J$, the function $\widetilde{\boldsymbol{\mathrm{grad}}}\,q^{\tau}_{j}$ is the extension by continuity of $\boldsymbol{\mathrm{grad}}\,q^{\tau}_{j}$ to $\Omega$,  with $q_{j}^{\tau}$ is the unique solution up to an additive
 constant of the problem:
 \begin{equation}\label{gradqj}
\left\{
\begin{array}{r@{~}c@{~}l}
- \Delta q_{j}^{\tau} &=& 0 \qquad \textrm{in} \,\,\, \Omega^{\circ}, \\
\partial_{n}q_{j}^{\tau}&=&0\qquad\textrm{on}\,\,\,\Gamma,\\
\left[ q_{j}^{\tau}\right] _{k}&=& \mathrm{constant},\,\,\,1\leq k \leq J,\\
\left[ \partial_{n}q_{j}^{\tau}\right] _{k}&=&0;\,\,\,1\leq k\leq J,\\
\langle\partial_{n}q_{j}^{\tau}\,,\,1\rangle_{\Sigma_{k}}&=&\delta_{jk},\,\,\,1\leq
k\leq J.
\end{array}
\right.
\end{equation}
We recall that, for all $1\leq j\leq J,\,$ the product $\langle\cdot\,.\,\cdot\rangle_{\Sigma_{j}}$ is the duality product between $\boldsymbol{W}^{-\frac{1}{p},p}(\Sigma_{j})\,$ and $\,\boldsymbol{W}^{1-\frac{1}{p'},p'}(\Sigma_{j})$.

\subsubsection{The Stokes operator with Navier-type  conditions on $\boldsymbol{L}^p _{ \sigma , \tau  }(\Omega )$}
Consider the space
\begin{equation}\label{vpt}
\boldsymbol{X}^{p}_{\sigma,\tau}(\Omega)\,=\,\Big\{\boldsymbol{v}\in\boldsymbol{X}^{p}_{\tau}(\Omega);\,\,\textrm{div}\,\boldsymbol{v}=0\,\,\,\textrm{in}\,\,\Omega\Big\},
\end{equation}
which is a Banach space for the norm $\boldsymbol{X}^{p}(\Omega)$. We recall that $\boldsymbol{X}^{p}_{\sigma,\tau}(\Omega)$ is a closed subspace
of $\boldsymbol{X}^{p}_{\tau}(\Omega)$ and on $\boldsymbol{X}^{p}_{\sigma,\tau}(\Omega)$ the norm of $\boldsymbol{X}^{p}_{\tau}(\Omega)$ is equivalent to the norm of $\boldsymbol{W}^{1,p}(\Omega)$.

Let $\boldsymbol{u}\in\boldsymbol{L}^{p} _{ \sigma , \tau  }(\Omega)$ be fixed and consider the mapping
\begin{eqnarray*}
{A}_{p}\boldsymbol{u}&:&\quad \boldsymbol{W}\longrightarrow \mathbb{C}\\
& &\quad \boldsymbol{v} \longrightarrow -\int_{\Omega}\boldsymbol{u}\cdot\Delta\overline{\boldsymbol{v}}\,\mathrm{d}\,x,
\end{eqnarray*}
where
\begin{equation*}
\boldsymbol{W}=\boldsymbol{X}^{p'}_{\sigma,\tau}(\Omega)\cap\boldsymbol{W}^{2,p'}(\Omega).
\end{equation*}
It is clear that ${A}_{p}\in\mathcal{L}(\boldsymbol{L}^{p} _{ \sigma , \tau  }(\Omega),\,\boldsymbol{W}')$ and thanks to de Rham's Lemma there exists $\pi\in W^{-1,p}(\Omega)$ such that $${A}_{p}\boldsymbol{u}+\Delta\boldsymbol{u}=\nabla\pi\qquad\mathrm{in}\quad\Omega.$$
Now suppose that $\boldsymbol{u}\in\boldsymbol{L}^{p} _{ \sigma , \tau  }(\Omega)$ and ${A}_{p}\boldsymbol{u}\in\boldsymbol{L}^{p} _{ \sigma , \tau  }(\Omega)$. Since $\Delta\boldsymbol{u}=-{A}_{p}\boldsymbol{u}+\nabla\pi,\,$ then thanks to \cite[Lemma 4.4]{Am3} $\,\mathbf{curl}\,\boldsymbol{u}\times\boldsymbol{n}\in\boldsymbol{W}^{-1-1/p,p}(\Gamma).$ Moreover if we suppose that $\mathbf{curl}\,\boldsymbol{u}\times\boldsymbol{n}=\boldsymbol{0}\,$ on $\,\Gamma$ then $(\boldsymbol{u},\,\pi)\in\boldsymbol{L}^{p}_{\sigma,\tau}(\Omega)\times W^{-1,p}(\Omega)$ is a solution of the problem
 \begin{equation*}
\left\{
\begin{array}{r@{~}c@{~}l}
 - \Delta \boldsymbol{u}\,+\,\nabla\pi ={A}_{p}\boldsymbol{u}, &&\quad\mathrm{div}\,\boldsymbol{u} = 0 \,\,\, \qquad\qquad \mathrm{in} \,\,\, \Omega, \\
\boldsymbol{u}\cdot \boldsymbol{n} = 0, &&\quad \mathbf{curl}\,\boldsymbol{u}\times\boldsymbol{n}=\boldsymbol{0}\qquad
\,\,\mathrm{on}\,\,\, \Gamma.
\end{array}
\right.
\end{equation*}
As a result using the regularity of the Stokes Problem \cite[Theorem 4.8]{Am3} one has $(\boldsymbol{u},\,\pi)\in\boldsymbol{W}^{2,p}(\Omega)\times W^{1,p}(\Omega)$ .
 
The operator ${A}_{p}:\mathbf{D}({A}_{p})\subset\boldsymbol{L}^{p} _{ \sigma , \tau  }(\Omega)\longmapsto\boldsymbol{L}^{p} _{ \sigma , \tau  }(\Omega)$ is a linear operator with
\begin{equation}
\label{c2dpa}
\mathbf{D}({A}_{p})=\Big\{ \boldsymbol{u} \in
\boldsymbol{W}^{2,p}(\Omega);\,\,\mathrm{div}\,\boldsymbol{u}=
0\,\, \mathrm{in}\,\,\Omega,
\,\,\boldsymbol{u}\cdot\boldsymbol{n}=0,\,\,\,\boldsymbol{\mathrm{curl}}\,\boldsymbol{u}\times\boldsymbol{n}=\boldsymbol{0}
\,\,\mathrm{on}\,\,\Gamma \Big\},
\end{equation} 
provided that $\Omega$ is of class $C^{2.1}$ (cf. \cite{Albaba}). Moreover 
\begin{equation}\label{Aptilde1}
\forall\,\,\boldsymbol{u}\in\mathbf{D}({A}_{p}),\qquad
{A}_{p}\boldsymbol{u}\,=\,-\Delta\boldsymbol{u}\,+\boldsymbol{\mathrm{grad}}\,\pi,
\end{equation}
where $\pi$ is the unique solution up to an additive constant of
the problem
\begin{equation*}
\mathrm{div}(\boldsymbol{\mathrm{grad}}\,\pi\,-\,\Delta\boldsymbol{u})=0\qquad\mathrm{in}\,\,\Omega,\qquad
(\boldsymbol{\mathrm{grad}}\,\pi\,-\,\Delta\boldsymbol{u})\cdot\boldsymbol{n}=0\qquad\textrm{on}\,\,\,\Gamma.
\end{equation*}
Observe that for all $\boldsymbol{u}\in\mathbf{D}({A}_{p})$ and for all $\boldsymbol{v}\in\boldsymbol{X}^{p'}_{\sigma,\tau}(\Omega)$ one has \begin{equation*}
\int_{\Omega}{A}_{p}\,\boldsymbol{u}\cdot\overline{\boldsymbol{v}}\,\mathrm{d}\,x\,=\,\int_{\Omega}\boldsymbol{\mathrm{curl}}\,\boldsymbol{u}\cdot\boldsymbol{\mathrm{curl}}\,\overline{\boldsymbol{v}}\,\textrm{d}\,x.
\end{equation*}
It easily follows that $(A_{p})^{\ast}=A_{p'}$. 

Notice also that for all for all $1<p,q<\infty$ and $\boldsymbol{u}\in\mathbf{D}(A_{p})\cap\mathbf{D}(A_{q})$, $A_{p}\boldsymbol{u}=A_{q}\boldsymbol{u}$.
We also recall the following propositions, see \cite[Proposition 3.1]{Albaba} for the proof.
\begin{prop}\label{sl}
For all $\boldsymbol{u}\in\mathbf{D}(A_{p})$, $A_{p}\boldsymbol{u}=-\Delta\boldsymbol{u}$. 
\end{prop}
\begin{rmk}
\rm{Unlike the Stokes operator with Dirichlet boundary condition,
we observe that here the pressure is constant, while with
Dirichlet boundary condition the pressure cannot be a constant
since it is the solution of the Problem (\ref{stokesdirichlet}).}
\end{rmk}
In the rest of this paper we will consider the Stokes operator with Navier-type boundary conditions \eqref{nbc}. We end this section by the following propositions (see \cite[Proposition 3.2, Proposition 3.3]{Albaba} for the proof):
\begin{prop}\label{dd1}
The space $\mathbf{D}(A_{p})$ is dense in
$\boldsymbol{L}^{p}_{\sigma,\tau}(\Omega)$.
\end{prop}

\begin{rmk}\label{rmkequivnorm}
\rm{(i) Notice that, thanks to Lemmas \ref{con1} and \ref{con2}, since $\Omega$ is of class $C^{2,1}$ we have
\begin{equation*} 
\forall\,\boldsymbol{u}\in\mathbf{D}(A_{p}),\qquad
\Vert\boldsymbol{u}\Vert_{\boldsymbol{W}^{2,p}(\Omega)}\simeq \Vert\boldsymbol{u}\Vert_{\boldsymbol{L}^{p}(\Omega)}+\Vert\Delta\boldsymbol{u}\Vert_{\boldsymbol{L}^{p}(\Omega)}.
\end{equation*}

\noindent(ii) We recall that, thanks to \cite[Proposition 4.7]{Am3}, since $\Omega$ is of class $C^{2,1}$, for all $\boldsymbol{u}\in\mathbf{D}(A_{p})$ such that $\langle\boldsymbol{u}\cdot\boldsymbol{n}\,,\,1\rangle_{\Sigma_{j}}=0$, $1\leq j \leq J$ we have
\begin{equation*}
\Vert\boldsymbol{u}\Vert_{\boldsymbol{W}^{2,p}(\Omega)}\simeq\,\Vert\Delta\boldsymbol{u}\Vert_{\boldsymbol{L}^{p}(\Omega)}.
\end{equation*}
}
\end{rmk}
\begin{prop}\label{densitédurang}
Suppose that $\Omega$ is not simply connected. The range $R(A_{p})$ of the Stokes operator is not dense in $\boldsymbol{L}^{p}_{\sigma,\tau}(\Omega)$.
\end{prop}
\begin{proof}
Since the domain $\Omega$ is not simply connected, the dimension of the kernel $\boldsymbol{K}_{\tau}(\Omega)$ of the Stokes operator $A_{p'}$ on $\boldsymbol{L}^{p'}_{\sigma,\tau}(\Omega)$ is finite and greater than or equal to 1. Suppose then that the range $R(A_{p})$ is dense in $\boldsymbol{L}^{p}_{\sigma,\tau}(\Omega)$.  Using the fact  that $(A_{p})^{\ast}=A_{p'}$ and that
\begin{equation*}
\boldsymbol{L}^{p}_{\sigma,\tau}(\Omega)=\overline{R(A_{p})}= [Ker(A_{p'})]^{\perp}
\end{equation*}
we obtain that 
\begin{equation*}
Ker(A_{p'})=\boldsymbol{K}_{\tau}(\Omega)=\{0\}
\end{equation*}
which is a contradiction.
\end{proof}
\subsubsection{The Stokes operator with Navier-type  conditions on $[\boldsymbol{H}^{p'}_{0}(\mathrm{div},\Omega)]'_{\sigma,\tau}$}
Consider now the space:
\begin{equation*}
\boldsymbol{E}=\{\boldsymbol{f}\in[\boldsymbol{H}^{p'}_{0}(\mathrm{div},\Omega)]';\,\,\mathrm{div}\,\boldsymbol{f}\in L^{p}(\Omega)\},
\end{equation*}
which is a Banach space with the norm 
\begin{equation}
\Vert\boldsymbol{f}\Vert_{\boldsymbol{E}}=\Vert\boldsymbol{f}\Vert_{[\boldsymbol{H}^{p'}_{0}(\mathrm{div},\Omega)]'}+\Vert\mathrm{div}\,\boldsymbol{f}\Vert_{L^{p}(\Omega)}.
\end{equation}
\begin{lemma}\label{density1}
The space $\boldsymbol{\mathcal{D}}(\overline{\Omega})$ is dense in $\boldsymbol{E}$.
\end{lemma}
\begin{proof}
Let $\boldsymbol{\ell}\in\boldsymbol{E}'$ such that $\langle\boldsymbol{\ell}\,,\,\boldsymbol{v}\rangle_{\boldsymbol{E}'\times\boldsymbol{E}}=0$ for all $\boldsymbol{v}\in\boldsymbol{\mathcal{D}}(\overline{\Omega})$ and let us show that $\boldsymbol{\ell}$ is null in $\boldsymbol{E}$. We know that there exists a function $\boldsymbol{u}$ in $\boldsymbol{H}^{p'}_{0}(\mathrm{div},\Omega)$ and a function $\chi$ in $L^{p'}(\Omega)$ such that for all $\boldsymbol{f}$ in $\boldsymbol{E}$ one has:
\begin{equation}
\langle\boldsymbol{\ell}\,,\,\boldsymbol{f}\rangle_{\boldsymbol{E}'\times\boldsymbol{E}}\,=\,\langle\boldsymbol{f}\,,\,\boldsymbol{u}\rangle_{[\boldsymbol{H}^{p'}_{0}(\mathrm{div},\Omega)]'\times\boldsymbol{H}^{p'}_{0}(\mathrm{div},\Omega)}\,+\,\int_{\Omega}\mathrm{div}\,\boldsymbol{f}\,\overline{\chi}\,\textrm{d}\,x.
\end{equation}
We denote by $\widetilde{\boldsymbol{u}}$ and $\widetilde{\chi}$ the extension of $\boldsymbol{u}$ and $\chi$ by zero to $\mathbb{R}^{3}$. As a result for every $\boldsymbol{f}\in\boldsymbol{\mathcal{D}}(\mathbb{R}^{3})$ one has
\begin{equation*}
\langle\boldsymbol{f}\,,\,\widetilde{\boldsymbol{u}}\rangle_{[\boldsymbol{H}^{p'}_{0}(\mathrm{div},\mathbb{R}^{3})]'\times\boldsymbol{H}^{p'}_{0}(\mathrm{div},\mathbb{R}^{3})}\,+\,\int_{\mathbb{R}^{3}}\mathrm{div}\,\boldsymbol{f}\,\overline{\widetilde{\chi}}\textrm{d}\,x\,=\,0.
\end{equation*}
Then $\widetilde{\boldsymbol{u}}=\nabla\,\widetilde{\chi}$ and $\boldsymbol{u}=\nabla\,\chi$. This means that $\widetilde{\chi}\in L^{p'}(\mathbb{R}^{3})$ and $\nabla\,\widetilde{\chi}\in\boldsymbol{H}^{p'}_{0}(\mathrm{div},\,\mathbb{R}^{3})$. Then $\widetilde{\chi}\in W^{2,p'}(\mathbb{R}^{3})$ and $\chi\in W^{2,p'}_{0}(\Omega)$. Now since $\mathcal{D}(\Omega)$ dense in $W^{2,p'}_{0}(\Omega)$ there exists a sequence $(\chi_{k})_{k}$ in $\mathcal{D}(\Omega)$ that converges to $\chi$ in $W^{2,p'}(\Omega)$. Finally  for all $\boldsymbol{f}\in\boldsymbol{E}$ one has:
\begin{eqnarray*}
\langle\boldsymbol{\ell}\,,\,\boldsymbol{f}\rangle_{\boldsymbol{E}'\times\boldsymbol{E}}&=&\langle\boldsymbol{f}\,,\,\boldsymbol{u}\rangle_{[\boldsymbol{H}^{p'}_{0}(\mathrm{div},\Omega)]'\times\boldsymbol{H}^{p'}_{0}(\mathrm{div},\Omega)}\,+\,\int_{\Omega}\mathrm{div}\,\boldsymbol{f}\,\overline{\chi}\,\textrm{d}\,x.\\
&=&\lim _{k\rightarrow +\infty}\Big[\langle\boldsymbol{f}\,,\,\nabla\,\chi_{k}\rangle_{[\boldsymbol{H}^{p'}_{0}(\mathrm{div},\Omega)]'\times\boldsymbol{H}^{p'}_{0}(\mathrm{div},\Omega)}\,+\,\int_{\Omega}\mathrm{div}\,\boldsymbol{f}\,\overline{\chi_{k}}\,\textrm{d}\,x.\Big] \\
&=&0.
\end{eqnarray*}
\end{proof}
The following Corollary gives us the normal trace of a function $\boldsymbol{f}$ in $\boldsymbol{E}$.
\begin{coro}\label{tracehpdiv}
The linear mapping $\gamma\,:\,\boldsymbol{f}\longmapsto\boldsymbol{f}\cdot\boldsymbol{n}$ defined on $\boldsymbol{\mathcal{D}}(\overline{\Omega})$ can be extended to a linear continuous mapping still denoted by $\gamma\,:\,\boldsymbol{E}\longmapsto W^{-1-1/p,p}(\Gamma)$. Moreover we have the following Green formula: for all $\boldsymbol{f}\in\boldsymbol{E}$ and for all $\chi\in W^{2,p'}(\Omega)$ such that $\frac{\partial\,\chi}{\partial\,\boldsymbol{n}}=0$ on $\Gamma$,
\begin{equation}\label{greenfrhdiv}
\int_{\Omega}(\mathrm{div}\,\boldsymbol{f})\,\overline{\chi}\,\textrm{d}\,x\,=\,-\,\langle\boldsymbol{f}\,,\,\nabla\,\chi\rangle_{\Omega}\,+\,\langle\boldsymbol{f}\cdot\boldsymbol{n}\,,\,\chi\rangle_{\Gamma},
\end{equation}
where $\langle\cdot\,,\,\cdot\rangle_{\Omega}=\langle\cdot\,,\,\cdot\rangle_{[\boldsymbol{H}^{p'}_{0}(\mathrm{div},\Omega)]'\times\boldsymbol{H}^{p'}_{0}(\mathrm{div},\Omega)}$ and  $\langle \cdot\,,\,\cdot \rangle_{\Gamma}\,=\,\langle \cdot\,,\, \cdot
\rangle_{W^{-1-1/p,p}(\Gamma)\times W^{1+1/p,p'}(\Gamma)}$.
 \end{coro}
 
 Now we consider the space
 \begin{equation}\label{hpdivsigmatau}
 [\boldsymbol{H}^{p'}_{0}(\mathrm{div},\Omega)]'_{\sigma,\tau}\,=\,\big\{\boldsymbol{f}\in[\boldsymbol{H}^{p'}_{0}(\mathrm{div},\Omega)]';\,\,\mathrm{div}\,\boldsymbol{f}=0\,\,\textrm{in}\,\,\Omega,\,\,\boldsymbol{f}\cdot\boldsymbol{n}=0\,\,\textrm{on}\,\,\Gamma\big\}.
 \end{equation}
  We define the operator  
 \begin{equation*}\label{b1}
 B_{p}\,:\,\mathbf{D}(B_{p})\subset[\boldsymbol{H}^{p'}_{0}(\mathrm{div},\Omega)]'_{\sigma,\tau}\longmapsto[\boldsymbol{H}^{p'}_{0}(\mathrm{div},\Omega)]'_{\sigma,\tau},
 \end{equation*}
 by
 \begin{equation}\label{b2}
 \forall\,\boldsymbol{u}\in\mathbf{D}(B_{p}),\qquad B_{p}\,\boldsymbol{u}=-\Delta\boldsymbol{u}\qquad\textrm{in}\,\,\Omega.
 \end{equation}
 The domain of $B_{p}$ is given by
 \begin{multline}\label{Dpb}
 \mathbf{D}(B_{p})\,=\big\{\boldsymbol{u}\in\boldsymbol{W}^{1,p}(\Omega);\,\,\Delta\boldsymbol{u}\in[\boldsymbol{H}^{p'}_{0}(\mathrm{div},\Omega)]'\,\,\mathrm{div}\,\boldsymbol{u}=0\,\,\textrm{in}\,\,\Omega,\\
 \boldsymbol{u}\cdot\boldsymbol{n}\,=\,0,\,\,\,\boldsymbol{\mathrm{curl}}\,\boldsymbol{u}\times\boldsymbol{n}\,=\,\boldsymbol{0}\,\,\textrm{on}\,\,\Gamma\big\}.
 \end{multline}
 \begin{rmk}
 \rm{The operator $B_{p}$ is the extension of the Stokes operator to $[\boldsymbol{H}^{p'}_{0}(\mathrm{div},\Omega)]'_{\sigma,\tau}$.
} 
 \end{rmk}
 \begin{prop}\label{denshpdiv}
 The space $\boldsymbol{\mathcal{D}}_{\sigma}(\Omega)$ is dense in $[\boldsymbol{H}^{p'}_{0}(\mathrm{div},\Omega)]'_{\sigma,\tau}$.
 \end{prop}
 \begin{proof}
 Let $\boldsymbol{\ell}$  be a linear form on $[\boldsymbol{H}^{p'}_{0}(\mathrm{div},\Omega)]'_{\sigma,\tau}$ such that $\boldsymbol{\ell}$ vanishes on $\boldsymbol{\mathcal{D}}_{\sigma}(\Omega)$ and let us show that $\boldsymbol{\ell}$ is null on $[\boldsymbol{H}^{p'}_{0}(\mathrm{div},\Omega)]'_{\sigma,\tau}$. Thanks to the Hahn-Banach theorem, $\boldsymbol{\ell}$ can be extended to a linear continuous form on $[\boldsymbol{H}^{p'}_{0}(\mathrm{div},\Omega)]'$ denoted by $\widetilde{\boldsymbol{\ell}}$. Moreover
 \begin{equation*}
 \forall\,\boldsymbol{f}\in[\boldsymbol{H}^{p'}_{0}(\mathrm{div},\Omega)]'_{\sigma,\tau},\qquad\boldsymbol{\ell}(\boldsymbol{f})\,=\,\langle\widetilde{\boldsymbol{\ell}}\,,\,\boldsymbol{f}\rangle_{\boldsymbol{H}^{p'}_{0}(\mathrm{div},\Omega)\times[\boldsymbol{H}^{p'}_{0}(\mathrm{div},\Omega)]'}.
\end{equation*}   
Since $\boldsymbol{\ell}$ vanishes on $\boldsymbol{\mathcal{D}}_{\sigma}(\Omega)$ then thanks to De-Rham lemma there exists a function $\pi\in W^{2,p'}(\Omega)$ such that $\frac{\partial\,\pi}{\partial\,\boldsymbol{n}}=0$ on $\Gamma$ and $\widetilde{\boldsymbol{\ell}}=\nabla\pi$ in $\Omega$. Now let $\boldsymbol{f}\in[\boldsymbol{H}^{p'}_{0}(\mathrm{div},\Omega)]'_{\sigma,\tau}$ then by Corollary \ref{tracehpdiv} we have
 \begin{eqnarray*}
 \boldsymbol{\ell}(\boldsymbol{f})&=&\langle\boldsymbol{f}\,,\,\nabla\pi\rangle_{[\boldsymbol{H}^{p'}_{0}(\mathrm{div},\Omega)]'\times\boldsymbol{H}^{p'}_{0}(\mathrm{div},\Omega)}\\
 &=&-\int_{\Omega}(\mathrm{div}\,\boldsymbol{f})\,\overline{\pi}\,\textrm{d}\,x\,+\,\langle\boldsymbol{f}\cdot\boldsymbol{n}\,,\,\pi\rangle_{\Gamma}\\
 &=&0.
 \end{eqnarray*}
 \end{proof}
 As a result of Proposition \ref{denshpdiv} we deduce the density of the domain of the operator $B_{p}$.
 \begin{coro}
 The operator $B_{p}$ is a densely defined operator.
 \end{coro}

\subsubsection{The Stokes operator with Navier-type  conditions on $[\boldsymbol{T}^{p'}(\Omega)]'_{\sigma,\tau}$}
Consider the space
 \begin{equation*}
 \boldsymbol{G}\,=\,\big\{\boldsymbol{f}\in(\boldsymbol{T}^{p'}(\Omega))';\,\,\mathrm{div}\,\boldsymbol{f}\in L^{p}(\Omega)\big\},
 \end{equation*}
 equipped with the graph norm.
We skip the proof of the following lemma because it is similar to the proof of Lemma \ref{density1}:
 \begin{lemma}
The space $\boldsymbol{\mathcal{D}}(\overline{\Omega})$ is dense in $\boldsymbol{G}$.
\end{lemma}
As in the previous Subsection The following Corollary gives the normal trace of functions in $\boldsymbol{G}$.
\begin{coro}
The linear mapping $\gamma\,:\,\boldsymbol{f}\longmapsto\boldsymbol{f}\cdot\boldsymbol{n}$ defined on $\boldsymbol{\mathcal{D}}(\overline{\Omega})$ can be extended to a linear continuous mapping still denoted by $\gamma\,:\,\boldsymbol{G}\longmapsto W^{-2-1/p,p}(\Gamma)$. Moreover we have the following Green formula: for all $\boldsymbol{f}\in\boldsymbol{G}$ and for all $\chi\in W^{3,p'}(\Omega)$ such that $\frac{\partial\,\chi}{\partial\,\boldsymbol{n}}=0$ on $\Gamma$  and $\Delta\chi=0$ on $\Gamma$,
\begin{equation}\label{greenfrtp}
\int_{\Omega}(\mathrm{div}\,\boldsymbol{f})\,\overline{\chi}\,\textrm{d}\,x\,=\,-\,\langle\boldsymbol{f}\,,\,\nabla\,\chi\rangle_{(\boldsymbol{T}^{p'}(\Omega))'\times\boldsymbol{T}^{p'}(\Omega)}\,+\,\langle\boldsymbol{f}\cdot\boldsymbol{n}\,,\,\chi\rangle_{\Gamma}.
\end{equation}
We recall that $\langle .\,,\,. \rangle_{\Gamma}\,=\,\langle .\,,\, .
\rangle_{W^{-2-1/p,p}(\Gamma)\times W^{2+1/p,p'}(\Gamma)}$.
 \end{coro}

Now we consider the space
\begin{equation}\label{tp'sigmatau}
[\boldsymbol{T}^{p'}(\Omega)]'_{\sigma,\tau}\,=\,\big\{\boldsymbol{f}\in(\boldsymbol{T}^{p'}(\Omega))';\,\,\mathrm{div}\,\boldsymbol{f}=0\,\,\textrm{in}\,\,\Omega,\,\,\boldsymbol{f}\cdot\boldsymbol{n}=0\,\,\textrm{on}\,\,\Gamma\big\} .
\end{equation}
Next, we consider the operator:
\begin{equation*}
C_{p}\,:\,\mathbf{D}(C_{p})\subset[\boldsymbol{T}^{p'}(\Omega)]'_{\sigma,\tau}\longmapsto[\boldsymbol{T}^{p'}(\Omega)]'_{\sigma,\tau},
\end{equation*}
defined by
\begin{equation}\label{C2}
\forall\,\boldsymbol{u}\in\mathbf{D}(C_{p}),\qquad C_{p}\,\boldsymbol{u}=-\Delta\boldsymbol{u}\qquad \textrm{in} \Omega.
\end{equation}
The domain of $C_{p}$ is given by
 \begin{equation}\label{Dpc}
 \mathbf{D}(C_{p})\,=\big\{\boldsymbol{u}\in\boldsymbol{L}^{p}(\Omega);\,\,\Delta\boldsymbol{u}\in(\boldsymbol{T}^{p'}(\Omega))',\,\,\mathrm{div}\,\boldsymbol{u}=0\,\,\textrm{in}\,\,\Omega,\,\,\boldsymbol{u}\cdot\boldsymbol{n}\,=\,0,\,\,\,\boldsymbol{\mathrm{curl}}\,\boldsymbol{u}\times\boldsymbol{n}\,=\,\boldsymbol{0}\,\,\textrm{on}\,\,\Gamma\big\}.
 \end{equation}
 \begin{rmk}
 \rm{ The operator $C_{p}$ is the extension of the stokes operator to $[\boldsymbol{T}^{p'}(\Omega)]'_{\sigma,\tau}$.
}
 \end{rmk}
 We skip the proof of the following proposition because it is similar to the proof of Proposition \ref{denshpdiv}:
 \begin{prop}\label{densttp}
 The space $\boldsymbol{\mathcal{D}}_{\sigma}(\Omega)$ is dense in $[\boldsymbol{T}^{p'}(\Omega)]'_{\sigma,\tau}$.
 \end{prop}

\section{Analyticity results}
\label{semigroups}
In this section we will state our main result and  its proof. We will prove that the Stokes operator with Navier-type boundary conditions \eqref{nbc} generates a bounded analytic semi-group on $\boldsymbol{L}^{p}_{\sigma,\tau}(\Omega)$, $[\boldsymbol{H}^{p'}_{0}(\mathrm{div},\Omega)]'_{\sigma,\tau}$ and $[\boldsymbol{T}^{p'}(\Omega)]'_{\sigma,\tau}$ respectively for all $1<p<\infty$. 
\subsection{Analyticity on $\boldsymbol{L}^{p}_{\sigma,\tau}(\Omega)$}
In this subsection, we review the main results of \cite{Albaba} concerning the analyticity of the semi-group generated by the Stokes operator with Navier-type boundary conditions $A_p$ on $\boldsymbol{L}^{p}_{\sigma,\tau}(\Omega)$, (see \cite{Albaba} for the proof).
\subsubsection{The Hilbertian case}
 The results  of \cite{Albaba} on the resolvent of the Stokes operator are obtained considering the problem \eqref{*}, that we recall here:
 \begin{equation}\label{resolventexistence}
\left\{
\begin{array}{r@{~}c@{~}l}
\lambda \boldsymbol{u} - \Delta \boldsymbol{u} = \boldsymbol{f}, && \mathrm{div}\,\boldsymbol{u} = 0 \qquad\qquad\,\,\, \mathrm{in} \,\,\, \Omega, \\
\boldsymbol{u}\cdot \boldsymbol{n} = 0, &&\boldsymbol{\mathrm{curl}}\,\boldsymbol{u}\times\boldsymbol{n}=\boldsymbol{0} \qquad
\mathrm{on}\,\,\, \Gamma,
\end{array}
\right.
\end{equation}
where $\boldsymbol{f}\in\boldsymbol{L}^{2}_{\sigma,\tau}(\Omega)$ and $\lambda\in\Sigma_{\varepsilon}$.
\begin{rmk}
\rm{
Observe that, Problem (\ref{resolventexistence}) is equivalent to the problem
\begin{equation}\label{sansdiv0}
\left\{
\begin{array}{cccc}
\lambda \boldsymbol{u} - \Delta \boldsymbol{u} = \boldsymbol{f},&  &\qquad \mathrm{in}& \Omega, \\
\boldsymbol{u}\cdot \boldsymbol{n} = 0, &\boldsymbol{\mathrm{curl}}\,\boldsymbol{u}\times\boldsymbol{n}=\boldsymbol{0}& \qquad
\mathrm{on}& \Gamma.
\end{array}
\right.
\end{equation}
Let $\boldsymbol{u}\in\boldsymbol{H}^{1}(\Omega)$ be the unique solution of Problem (\ref{sansdiv0}) and set $\mathrm{div}\,\boldsymbol{u}=\chi$. It is clear that $\lambda\chi-\Delta\chi\,=\,0$ in $\Omega$. Moreover, since $\boldsymbol{f}\cdot\boldsymbol{n}=0$ and $\boldsymbol{u}\cdot\boldsymbol{n}=0$ on $\Gamma$ then $\Delta\boldsymbol{u}\cdot\boldsymbol{n}=0$ on $\Gamma$. Notice also that the condition $\boldsymbol{\mathrm{curl}}\boldsymbol{u}\times\boldsymbol{n}=\boldsymbol{0}$ on $\Gamma$
 implies that $\boldsymbol{\mathrm{curl}}\,\boldsymbol{\mathrm{curl}}\boldsymbol{u}\cdot\boldsymbol{n}=0$ on $\Gamma$. Finally since $\Delta\boldsymbol{u}=\boldsymbol{\mathrm{grad}}(\mathrm{div}\,\boldsymbol{u})-\boldsymbol{\mathrm{curl}}\,\boldsymbol{\mathrm{curl}}\boldsymbol{u}$ one gets $\frac{\partial\chi}{\partial\boldsymbol{n}}=0$ on $\Gamma$. Thus $\chi=0$ in $\Omega$ and the result is proved.}
\end{rmk}
 We have the following theorem, for the proof see \cite[Theorem 4.3]{Albaba}.
\begin{theo}\label{existencehilbert}
Let $\varepsilon\in\left] 0,\pi\right[ $ be fixed, $\boldsymbol{f}\in\boldsymbol{L}^{2}_{\sigma,\tau}(\Omega)$ and $\lambda\in\Sigma_{\varepsilon}$.\\
\textbf{(i)}  The
Problem (\ref{resolventexistence})
 has a unique solution
$\boldsymbol{u}\in\boldsymbol{H}^{1}(\Omega)$.

\noindent\textbf{(ii)} There exist a constant $C'_{\varepsilon}>0$ independent of $\boldsymbol{f}$ and $\lambda$ such that the solution $\boldsymbol{u}$ satisfies the
estimates
\begin{equation}\label{esthilb1}
\|\boldsymbol{u}\|_{\boldsymbol{L}^{2}(\Omega)}\leq\frac{C'_{\varepsilon}}{|\lambda|}\|\boldsymbol{f}\|_{\boldsymbol{L}^{2}(\Omega)}
\end{equation}
and
\begin{equation}\label{esthilb2}
\|\boldsymbol{\mathrm{curl}}\,\boldsymbol{u}\|_{\boldsymbol{L}^{2}(\Omega)}\leq\frac{C'_{\varepsilon}}{\sqrt{|\lambda|}}\|\boldsymbol{f}\|_{\boldsymbol{L}^{2}(\Omega)}.
\end{equation}

\noindent\textbf{(iii)} $\boldsymbol{u}\in\boldsymbol{H}^{2}(\Omega)$ and satisfies the estimate
\begin{equation}\label{esthilb3}
\|\boldsymbol{u}\|_{\boldsymbol{H}^{2}(\Omega)}\leq\,\frac{C(\Omega,\lambda,\varepsilon)}{\vert\lambda\vert}\,\|\boldsymbol{f}\|_{\boldsymbol{L}^{2}(\Omega)},
\end{equation}
where $C(\Omega,\lambda,\varepsilon)=C(\Omega)(C'_{\varepsilon}+1)(\vert\lambda\vert+1)$.
\end{theo}
\begin{rmk}\label{m-dissipative}
\rm{We note that for $\lambda>0$ the constant $C'_{\varepsilon}$ is equal to $1$ and we recover the m-acritiveness property of the stokes operator on $\boldsymbol{L}^{2}_{\sigma,\tau}(\Omega )$.}
\end{rmk}
\begin{rmk}
\rm{Consider the sesqui-linear form:
\begin{equation}\label{sesq1}
\forall\,\boldsymbol{u},\,\boldsymbol{v}\in\boldsymbol{X}^{2}_{\sigma,\tau}(\Omega),\qquad
a(\boldsymbol{u},\boldsymbol{v})\,=\,\int_{\Omega}\boldsymbol{\mathrm{curl}}\,\boldsymbol{u}\cdot\boldsymbol{\mathrm{curl}}\,\overline{\boldsymbol{v}}\,\textrm{d}\,x.
\end{equation}
If $\Omega$ is simply connected, we know that (see \cite[Corollary 3.16]{Am2}) for all
$\boldsymbol{v}\in\boldsymbol{X}^{2}_{\sigma,\tau}(\Omega)$ one has
\begin{equation}\label{equivnorm1}
\|\boldsymbol{v}\|_{\boldsymbol{X}^{2}(\Omega)}\leq
C\,\|\boldsymbol{\mathrm{curl}}\,\boldsymbol{v}\|_{\boldsymbol{L}^{2}(\Omega)}.
\end{equation}
As a result, the sesqui-linear form $a$ is coercive and we can apply
Lax-Milgram Lemma to find solution to the problem: find
$\boldsymbol{u}\in\boldsymbol{X}^{2}_{\sigma,\tau}(\Omega)$ such that for
all $\boldsymbol{v}\in\boldsymbol{X}^{2}_{\sigma,\tau}(\Omega)$
\begin{equation*}
a(\boldsymbol{u},\boldsymbol{v})\,=\,\int_{\Omega}\boldsymbol{f}\cdot\overline{\boldsymbol{v}}\,\textrm{d}\,x,
\end{equation*}
where $\boldsymbol{f}\in\boldsymbol{L}^{2}_{\sigma,\tau}(\Omega)$.
This means that the operator
$A_{2}\,:\,\mathbf{D}(A_{2})\subset\boldsymbol{L}^{2}_{\sigma,\tau}(\Omega)\longmapsto\boldsymbol{L}^{2}_{\sigma,\tau}(\Omega)$
 is bijective.

 Now, if $\Omega$ is multiply-connected, the inequality (\ref{equivnorm1})
 is false because the kernel $K _{ \tau  }(\Omega )$ of the Stokes operator with Navier-type boundary conditions is not trivial (cf.  \cite{Am2}).
It is also proved in \cite{Am2}, that for all
$\boldsymbol{v}\in\boldsymbol{X}^{2}_{\tau}(\Omega)$ we have instead the
following Poincar\'{e}-type inequality:
\begin{equation}\label{pc1}
\|\boldsymbol{v}\|_{\boldsymbol{X}^{2}_{\tau}(\Omega)}\leq
C_{2}(\Omega)(\|\boldsymbol{\mathrm{curl}}\,\boldsymbol{v}\|_{\boldsymbol{L}^{2}(\Omega)}\,+\,\|\mathrm{div}\,\boldsymbol{v}\|_{\boldsymbol{L}^{2}(\Omega)}\,+\,\sum_{j=1}^{J}|\langle\boldsymbol{v}\cdot\boldsymbol{n}\,,\,1\rangle_{\Sigma_{j}}|).
\end{equation}
 }
\end{rmk}
As a consequence of Theorem \ref{existencehilbert} we have the following theorem
\begin{theo}\label{analsemi1}
The operator $-A_{2}$ generates a bounded analytic semi-group on $\boldsymbol{L}^{2}_{\sigma,\tau}(\Omega)$.
\end{theo}
\begin{rmk}
\rm{ We recall that the
restriction of an analytic semi-group to the non negative real axis is a $C_{0}$
semi-group. Thanks to Remark \ref{m-dissipative} the restriction of our analytic
semi-group to the real axis gives a $C_{0}$ semi-group of
contraction.\rm}
\end{rmk}
The following proposition gives the eigenvalues of the Stokes operator. We will see later that the following proposition allows us to obtain an explicit form for the unique solution of the homogeneous Stokes Problem \eqref{hensp} as a linear combination of the eigenfunctions of the Stokes operator.
\begin{prop}\label{eva}
There exists a sequence of functions
$(\boldsymbol{z}_{k})_{k}\subset\mathbf{D}(A_{2})$ and an increasing
sequence of real numbers $(\lambda_{k})_{k}$ such that
$\lambda_{k}\geq 0$, $\lambda_{k}\rightarrow +\infty$ as
$k\rightarrow +\infty$ and
\begin{equation*}
\forall\,\boldsymbol{v}\in\boldsymbol{X}^{2}_{\tau}(\Omega),\qquad
\int_{\Omega}\boldsymbol{\mathrm{curl}}\,\boldsymbol{z}_{k}\cdot\boldsymbol{\mathrm{curl}}\,\overline{\boldsymbol{v}}\,\mathrm{d}\,x\,=\,\lambda_{k}\,\int_{\Omega}\boldsymbol{z}_{k}\cdot\overline{\boldsymbol{v}}\,\mathrm{d}\,x.
\end{equation*}
In other words, $(\lambda_{k})_{k}$ are the eigenvalues of the
Stokes operator and $(\boldsymbol{z}_{k})_{k}$ are the associated
eigenfunctions.
\end{prop}
\begin{proof}
Consider the operator
\begin{equation*}
\Lambda\,:\,\boldsymbol{L}^{2}_{\sigma,\tau}(\Omega)\longmapsto\mathbf{D}(A_{2})\longmapsto\boldsymbol{L}^{2}_{\sigma,\tau}(\Omega)
\end{equation*}
\begin{equation*}
\boldsymbol{f}\longmapsto\boldsymbol{u}\longmapsto\boldsymbol{u},
\end{equation*}
where $\boldsymbol{u}$ is the unique solution of the problem
\begin{equation*}
\left\{
\begin{array}{ccc}
\boldsymbol{u} + A_{2} \boldsymbol{u} = \boldsymbol{f}, & \mathrm{div}\,\boldsymbol{u} =0 & \qquad \textrm{in} \,\,\, \Omega, \\
\boldsymbol{u}\cdot \boldsymbol{n} = 0, & \boldsymbol{\mathrm{curl}}\,\boldsymbol{u}\times\boldsymbol{n}=\boldsymbol{0}&\qquad
\textrm{on}\,\,\, \Gamma.
\end{array}
\right.
\end{equation*}
Thanks to Theorem \ref{existencehilbert}, we know that $\Lambda$ is a bounded linear operator from
$\boldsymbol{L}^{2}_{\sigma,\tau}(\Omega)$ into itself. Moreover,
thanks to Lemma \ref{con1} and the compact embedding of $\boldsymbol{H}^{1}(\Omega)$ in $\boldsymbol{L}^{2}(\Omega)$, the canonical embedding
$\mathbf{D}(A_{2})\hookrightarrow\boldsymbol{L}^{2}_{\sigma,\tau}(\Omega)$ is compact. Equivalently, the operator $\Lambda$ is compact from
$\boldsymbol{L}^{2}_{\sigma,\tau}(\Omega)$ into itself. Moreover we can easily verify that this operator
is also a self adjoint operator. Thus
$\boldsymbol{L}^{2}_{\sigma,\tau}(\Omega)$ has a Hilbertian basis
formed from the eigenvectors of the operator $\Lambda$. Then, there exists a sequence of real numbers  $(\mu_{k})_{k\geqslant 0}$
and eigenfunctions $(\boldsymbol{z}_{k})_{k\geqslant 0}$ such that
$\Lambda\boldsymbol{z}_{k}\,=\,\mu_{k}\,\boldsymbol{z}_{k}$ and 
$\mu_{k}\longrightarrow 0$ as $k\rightarrow +\infty$. This means that
$-\mu_{k}\,\Delta\boldsymbol{z}_{k}\,+\,\mu_{k}\,\boldsymbol{z}_{k}\,=\,\boldsymbol{z}_{k}$. Note that $0<\mu_{k}\leq 1$.
As a result $A_{2}\,\boldsymbol{z}_{k}\,=\,\lambda_{k}\,\boldsymbol{z}_{k}$, where $\lambda_{k}\,=\,\frac{1}{\mu_{k}}-1$ and $\lambda_{k}\longrightarrow
+\infty$ as $k\rightarrow +\infty.$  In conclusion
$(\boldsymbol{z}_{k})_{k}$ is a sequence of eigenfunctions of the
Stokes operator associated to the eigenvalues $(\lambda_{k})_{k}$.
\end{proof}
\begin{rmk}\label{remarkeva}
\rm{
As a consequence of Proposition \ref{eva},
$\boldsymbol{L}^{2}_{\sigma,\tau}(\Omega)$ can be written in the form
\begin{equation*}
\boldsymbol{L}^{2}_{\sigma,\tau}(\Omega)\,=\,\boldsymbol{\mathrm{Ker}}A_{2}\,\bigoplus_{k=1}^{+\infty}\boldsymbol{\mathrm{Ker}}(\lambda_{k}\,I\,-\,A_{2}).
\end{equation*}
In other words, any vector
$\boldsymbol{v}\in\boldsymbol{L}^{2}_{\sigma,\tau}(\Omega)$ can be
written in the form
\begin{equation*}
\boldsymbol{v}\,=\,\sum_{k=1}^{J}\alpha_{k}\widetilde{\boldsymbol{\mathrm{grad}}}\,q^{\tau}_{k}\,+\,\sum_{k=1}^{+\infty}\beta_{k}\boldsymbol{z}_{k},
\end{equation*}
where $(\widetilde{\boldsymbol{\mathrm{grad}}}\,q^{\tau}_{k})_{1\leq
k\leq J}$ is a basis for
$\boldsymbol{\mathrm{ker}}\,A_{2}\,=\,\boldsymbol{K}^{2}_{\tau}(\Omega)$
and
$\forall\,k\in\mathbb{N},\,\,\,\boldsymbol{z}_{k}\in\boldsymbol{\mathrm{ker}}\,(\lambda_{k}\,I\,-\,A_{2}).$
We recall that $J$ is the dimension of
$\boldsymbol{\mathrm{ker}}\,A_{2}\,=\,\boldsymbol{K}^{2}_{\tau}(\Omega)$,
(see \cite{Am2}).

 As described above, when $\Omega$ is simply-connected,
 $\boldsymbol{K}^{2}_{\tau}(\Omega)=\{\boldsymbol{0}\}$, $\lambda_{0}=0$ is
 not an eigenvalue and the Stokes operator is bijective from $\mathbf{D}(A_{2})$ into $\boldsymbol{L}^{2}_{\sigma,\tau}(\Omega)$ with bounded and compact inverse.
 In this case,
 \begin{equation*}
\boldsymbol{L}^{2}_{\sigma,\tau}(\Omega)\,=\,\bigoplus_{k=1}^{+\infty}\boldsymbol{\mathrm{Ker}}(\lambda_{k}\,I\,-\,A_{2}),
\end{equation*}
where $(\lambda_{k})_{k\geq 1}$ are the eigenvalues of the Stokes
operator and $(\boldsymbol{z_{k}})_{k}$ are the eigenfunctions
associated to eigenvalues $(\lambda_{k})_{k\geq 1}$. Moreover, the
sequence $(\lambda_{k})_{k\geq 1}$ is an increasing sequence of
positive real numbers and the first eigenvalue $\lambda_{1}$ is
equal to $\frac{1}{C_{2}(\Omega)}$ where $C_{2}(\Omega)$ is the
constant that comes from the Poincar\'{e}-type inequality
(\ref{pc1}).
}
\end{rmk}
\subsubsection{$\boldsymbol{L}^{p}$-theory}
This subsection extends Theorem \ref{existencehilbert} to every $1<p<\infty$. Theorem \ref{existencelp} gives the well posedness of the resolvent Problem \eqref{resolventexistence} in $\boldsymbol{L}^{p}(\Omega)$, while Theorem \ref{estlpnavier} extends estimates (\ref{esthilb1}-\ref{esthilb3}) to all $1<p<\infty$ (see \cite[Theorem 4.8, Theorem 4.11]{Albaba} for the proof).
\begin{theo}\label{existencelp}
Let $\lambda\in\mathbb{C}\in\Sigma_{\varepsilon}$ and let
$\boldsymbol{f}\in\boldsymbol{L}^{p}_{\sigma,\tau}(\Omega)$. The
Problem (\ref{resolventexistence}) has a  solution
$\boldsymbol{u}\in\boldsymbol{W}^{2,p}(\Omega)$. Moreover, this solution is unique in 
$\boldsymbol{u}\in\boldsymbol{W}^{1,p}(\Omega)$.
\end{theo}
 \begin{theo}\label{estlpnavier}
Let $\lambda\in\mathbb{C}^{\ast} $ such that
$\mathrm{Re}\,\lambda\,\geq 0$, let $1<p<\infty$,
$\boldsymbol{f}\in\boldsymbol{L}^{p}_{\sigma,\tau}(\Omega)$ and let
$\boldsymbol{u}\in\boldsymbol{W}^{1,p}(\Omega)$ be the unique
solution of Problem (\ref{resolventexistence}). Then $\boldsymbol{u}$ satisfies the
estimates
\begin{eqnarray}\label{estimlpf}
\|\boldsymbol{u}\|_{\boldsymbol{L}^{p}(\Omega)}\,\leq\,\frac{\kappa_{1}(\Omega,p)}{|\lambda|}\|\boldsymbol{f}\|_{\boldsymbol{L}^{p}(\Omega)},\\
\label{curlestlp}
\Vert\boldsymbol{\mathrm{curl}}\,\boldsymbol{u}\Vert_{\boldsymbol{L}^{p}(\Omega)}\,\leq\,\frac{\kappa_{2}(\Omega,p)}{\sqrt{\vert\lambda\vert}}\,\Vert\boldsymbol{f}\Vert_{\boldsymbol{L}^{p}(\Omega)}\\
\label{estw2p}
\Vert\boldsymbol{u}\Vert_{\boldsymbol{W}^{2,p}(\Omega)}\,\leq\,\kappa_{3}(\Omega,p)\,\frac{1+\vert\lambda\vert}{\vert\lambda\vert}\Vert\boldsymbol{f}\Vert_{\boldsymbol{L}^{p}(\Omega)}.
\end{eqnarray}
where $\kappa_{i}(\Omega,p)$, $i=1, 2, 3$  are positive constants independent of $\lambda$ and $\boldsymbol{f}$.
\end{theo}

As a result we have the following theorem (see \cite[Theorem 4.12]{Albaba} for the proof)
 \begin{theo}\label{analsemi2}
The operator $-A_{p}$ generates a bounded analytic semi-group on
$\boldsymbol{L}^{p}_{\sigma,\tau}(\Omega)$  for all $1<p<\infty$.
\end{theo}
\begin{rmk}\label{Reposi}
\rm{
Notice that, unlike the Hilbertian case, we can not use the result of \cite[Chapter II, Theorem 4.6, page 101]{En} to prove the analyticity of the semi-group generated by the Stokes operator in the  $\boldsymbol{L}^{p}$-space where we have supposed that $\mathrm{Re}\,\lambda\geq 0$.
}
\end{rmk}
\begin{rmk}\label{normboundcond}
\rm{ Consider the two problems:
\begin{equation}\label{resprlp}
\left\{
\begin{array}{cccc}
\lambda \boldsymbol{u} - \Delta \boldsymbol{u} = \boldsymbol{f},&\mathrm{div}\,\boldsymbol{u}=0& \textrm{in}& \Omega, \\
\boldsymbol{u}\times \boldsymbol{n} = \boldsymbol{0}&& \textrm{on}& \Gamma
\end{array}
\right.
\end{equation}
and
\begin{equation}\label{na}
\left\{
\begin{array}{cccc}
\lambda \boldsymbol{u} - \Delta \boldsymbol{u}\,+\,\nabla\pi=\, \boldsymbol{f},& \mathrm{div}\,\boldsymbol{u} = 0& \mathrm{in}&\qquad \Omega, \\
\boldsymbol{u}\cdot \boldsymbol{n} = 0,& \left[ \mathbb{D}\boldsymbol{u}\cdot\boldsymbol{n}\right]_{\boldsymbol{\tau}}=0&
\mathrm{on}&\qquad \Gamma,
\end{array}
\right.
\end{equation}
where $\lambda\in\mathbb{C}^{\ast} $ is such that $\mathrm{Re}\,\lambda\geq 0$ and $\boldsymbol{f}\in\boldsymbol{L}^{p}_{\sigma}(\Omega)$ (respectively $\boldsymbol{f}\in\boldsymbol{L}^{p}_{\sigma,\tau}(\Omega)$  ). 

In two forthcoming  papers we  study the two Problems (\ref{resprlp}) and (\ref{na}). Proceeding in a similar way as in \cite{Albaba} we prove that these two Problems have a unique solution $\boldsymbol{u}\in\boldsymbol{W}^{1,p}(\Omega)$ (respectively $(\boldsymbol{u},\pi)\in\boldsymbol{W}^{1,p}(\Omega)\times W^{1,p}(\Omega)/\mathbb{R}$) that satisfy the estimate
\begin{equation*}
\Vert\boldsymbol{u}\Vert_{\boldsymbol{L}^{p}(\Omega)}\leq\,\frac{C(\Omega,p)}{\vert\lambda\vert}\,\Vert\boldsymbol{f}\Vert_{\boldsymbol{L}^{p}(\Omega)}.
\end{equation*}
Moreover when $\Omega$ is of class $C^{2,1}$, we have $\boldsymbol{u}\in\boldsymbol{W}^{2,p}(\Omega)$.
This means that the Laplacian operator with normal boundary conditions and the Stokes operator with Navier boundary condition generate a bounded analytic semi-group on $\boldsymbol{L}^{p}_{\sigma}(\Omega)$ and $\boldsymbol{L}^{p}_{\sigma,\tau}(\Omega)$ respectively .

This analyticity allows us to solve the time dependent Stokes Problem with normal boundary condition and pressure boundary condition:
\begin{equation}\label{stnobo}
\left\{
\begin{array}{cccc}
\frac{\partial\boldsymbol{u}}{\partial t} - \Delta \boldsymbol{u
}+\nabla\pi= \boldsymbol{f},& \mathrm{div}\,\boldsymbol{u}=0& \textrm{in}&
\Omega\times (0,T), \\
\boldsymbol{u}\times \boldsymbol{n} = \boldsymbol{0},&\pi=0& \textrm{on} & \Gamma\times (0,T), \\
&\boldsymbol{u}(0)= \boldsymbol{u}_{0} &\textrm{in} &\Omega,
\end{array}
\right.
\end{equation}
}
as well as the time dependent Stokes Problem (\ref{lens}) with Navier-boundary condition (\ref{Navierboundcond}) for a given $\boldsymbol{f}\in L^{q}(0,T;\,\boldsymbol{L}^{p}(\Omega))$ and $\boldsymbol{u}_{0}\in\boldsymbol{L}^{p}_{\sigma}(\Omega)$ (respectively $\boldsymbol{u}_{0}\in\boldsymbol{L}^{p}_{\sigma,\tau}(\Omega)$).
\end{rmk}
Remark \ref{normboundcond} allows us to conclude the following corollary
\begin{coro}
The operator $-A_{p}$ generates a bounded analytic semi-group on $\boldsymbol{X}^{p}_{\sigma,\tau}(\Omega)$, where $\boldsymbol{X}^{p}_{\sigma,\tau}(\Omega)$ is given by (\ref{vpt}).
\end{coro}
\begin{proof}
Consider the Problem (\ref{resolventexistence}) where $\lambda\in\mathbb{C}^{\ast}$ such that $\mathrm{Re}\lambda\geq0$ and $\boldsymbol{f}\in\boldsymbol{X}^{p}_{\sigma,\tau}(\Omega)$. Thanks to Theorem \ref{existencelp} and Theorem \ref{estlpnavier} we know that Problem (\ref{resolventexistence}) has a unique solution $\boldsymbol{u}\in\boldsymbol{W}^{1,p}(\Omega)$ satisfying estimate (\ref{estimlpf}). Since $\Omega$ is of class $C^{2,1}$ then $\boldsymbol{u}\in\boldsymbol{W}^{2,p}(\Omega)$. Next, set $\boldsymbol{z}=\boldsymbol{\mathrm{curl}}\,\boldsymbol{u}$, thanks to Remark \ref{normboundcond} it is clear that $\boldsymbol{z}$ verifies
\begin{equation*}
\left\{
\begin{array}{cccc}
\lambda \boldsymbol{z} - \Delta \boldsymbol{z} = \boldsymbol{\mathrm{curl}}\,\boldsymbol{f},&\mathrm{div}\,\boldsymbol{z}=0& \textrm{in}& \Omega, \\
\boldsymbol{z}\times \boldsymbol{n} = \boldsymbol{0}&& \textrm{on}& \Gamma
\end{array}
\right.
\end{equation*}
and satisfies the estimate
\begin{equation*}
\Vert\boldsymbol{z}\Vert_{\boldsymbol{L}^{p}(\Omega)}\leq\,\frac{C(\Omega,p)}{\vert\lambda\vert}\,\Vert\boldsymbol{\mathrm{curl}}\,\boldsymbol{f}\Vert_{\boldsymbol{L}^{p}(\Omega)}.
\end{equation*} 
Thus the solution $\boldsymbol{u}$ of Problem (\ref{resolventexistence}) satisfies the estimate
\begin{equation*}
\Vert\boldsymbol{u}\Vert_{\boldsymbol{X}^{p}_{\tau}(\Omega)}\leq\,\frac{C(\Omega,p)}{\vert\lambda\vert}\,\Vert\boldsymbol{f}\Vert_{\boldsymbol{X}^{p}_{\tau}(\Omega)}.
\end{equation*}
As a result, we recover the analyticity of the semi-group generated by the operator $-A_{p}$ on $\boldsymbol{X}^{p}_{\sigma,\tau}(\Omega)$. We recall that $\boldsymbol{X}^{p}_{\sigma,\tau}(\Omega)$ is a subspace of $\boldsymbol{X}^{p}_{\tau}(\Omega)$, where the norm of $\boldsymbol{X}^{p}_{\tau}(\Omega)$ is equivalent to the norm of $\boldsymbol{W}^{1,p}(\Omega)$.
\end{proof}
\subsection{Analyticity on $[\boldsymbol{H}^{p'}_{0}(\mathrm{div},\Omega)]'_{\sigma,\tau}$}
\label{semigroupBp}
This subsection is devoted to the analyticity of the semi-group generated by the Stokes operator on $[\boldsymbol{H}^{p'}_{0}(\mathrm{div},\Omega)]'_{\sigma,\tau}$. This analyticity allows us to obtain the weak solution to the Problem (\ref{lens}) with the boundary condition (\ref{nbc}). 

To this end we  consider the problem:
\begin{equation}\label{wsp1}
 \left\{
\begin{array}{cccc}
\lambda \boldsymbol{u} - \Delta \boldsymbol{u}\,+\,\nabla\pi = \boldsymbol{f},& \mathrm{div}\,\boldsymbol{u} = 0& \qquad \mathrm{in}&\Omega, \\
\boldsymbol{u}\cdot \boldsymbol{n} = 0,&\boldsymbol{\mathrm{curl}}\,\boldsymbol{u}\times\boldsymbol{n}=\boldsymbol{0} &\qquad\mathrm{on}&\Gamma,
\end{array}
\right.
\end{equation}
where $\lambda\in\mathbb{C}^{\ast}$ such that $\mathrm{Re}\,\lambda\geq 0$ and $\boldsymbol{f}\in[\boldsymbol{H}^{p'}_{0}(\mathrm{div},\Omega)]'$.
  The following theorem gives  the existence and uniqueness of solution to Problem \eqref{wsp1}:
 \begin{theo}\label{weakexistenceth}
 Let  $\lambda\in\mathbb{C}^{\ast}$ such that $\mathrm{Re}\,\lambda\geq 0$ and let $\boldsymbol{f}\in[\boldsymbol{H}^{p'}_{0}(\mathrm{div},\Omega)]'$. The Problem (\ref{wsp1})
 has a unique solution $(\boldsymbol{u},\pi)\in\boldsymbol{W}^{1,p}(\Omega)\times L^{p}(\Omega)/\mathbb{R}$ satisfying 
 \begin{equation}\label{estwsp1}
 \Vert\boldsymbol{u}\Vert_{[\boldsymbol{H}^{p'}_{0}(\mathrm{div},\Omega)]'}\,\leq\,\frac{C(\Omega,p)}{\vert\lambda\vert}\,\Vert\boldsymbol{f}\Vert_{[\boldsymbol{H}^{p'}_{0}(\mathrm{div},\Omega)]'}
 \end{equation}
 for some constant $C(\Omega,p)>0$ independent of $\lambda$ and $\boldsymbol{f}$.
 \end{theo} 
 \begin{proof}
\textbf{(i)} For the existence of solutions for Problem (\ref{wsp1}) we proceed in the same way as in \cite[Theorem 4.4]{Am3}, Theorem \ref{existencehilbert} and Theorem \ref{existencelp}.

\noindent\textbf{(ii)} To prove estimate (\ref{estwsp1}) we proceed as follows:
Consider the problem:
 \begin{equation}\label{adjwsp1}
 \left\{
\begin{array}{cccc}
\lambda \boldsymbol{v} - \Delta \boldsymbol{v}\,+\,\nabla\theta = \boldsymbol{F},&\mathrm{div}\,\boldsymbol{v}=0 &\qquad \mathrm{in}&\Omega, \\
\boldsymbol{v}\cdot \boldsymbol{n} = 0,& \boldsymbol{\mathrm{curl}}\,\boldsymbol{v}\times\boldsymbol{n}=\boldsymbol{0}&\qquad
\mathrm{on}& \Gamma,
\end{array}
\right.
\end{equation}
where $\boldsymbol{F}\in\boldsymbol{H}^{p'}_{0}(\mathrm{div},\Omega)$ and $\lambda\in\mathbb{C}^{\ast}$ such that $\mathrm{Re}\,\lambda\geq 0$. 
Thanks to Lemma \ref{wn1} there exists a unique up to an additive function $\theta\in W^{1,p'}(\Omega)/\mathbb{R}$ solution of
\begin{equation*}
\mathrm{div}(\nabla\theta-\boldsymbol{F})=0\qquad\textrm{in}\,\,\Omega\qquad(\nabla\theta-\boldsymbol{F})\cdot\boldsymbol{n}=0\qquad\textrm{on}\,\,\Gamma.
\end{equation*}
Moreover the function $\theta$ satisfies the estimate
\begin{equation*}
\Vert\nabla\theta\Vert_{\boldsymbol{L}^{p'}(\Omega)}\,\leq\,C(\Omega,p')\,\Vert\boldsymbol{F}\Vert_{\boldsymbol{L}^{p'}(\Omega)}.
\end{equation*}
As a result, thanks to Theorem \ref{existencelp} and Theorem \ref{estlpnavier}, Problem (\ref{adjwsp1}) has a unique solution $(\boldsymbol{v},\theta)\in\boldsymbol{W}^{1,p'}(\Omega)\times W^{1,p'}(\Omega)/\mathbb{R}$ that satisfies the estimate
\begin{equation*}
\Vert\boldsymbol{v}\Vert_{\boldsymbol{L}^{p'}(\Omega)}\,\leq\,\frac{C(\Omega,p')}{\vert\lambda\vert}\,\Vert\boldsymbol{F}\Vert_{\boldsymbol{L}^{p'}(\Omega)}.
\end{equation*} 
Thus 
\begin{equation*}
\Vert\boldsymbol{v}\Vert_{\boldsymbol{H}^{p'}_{0}(\mathrm{div},\Omega)}\,\leq\,\frac{C(\Omega,p')}{\vert\lambda\vert}\,\Vert\boldsymbol{F}\Vert_{\boldsymbol{H}^{p'}_{0}(\mathrm{div},\Omega)}.
\end{equation*}
Now let $(\boldsymbol{u},\pi)\in\boldsymbol{W}^{1,p}(\Omega)\times\boldsymbol{L}^{p}(\Omega)/\mathbb{R}$ be the solution of Problem (\ref{wsp1}), then by using (\ref{greenfrhdiv}) we have:
\begin{eqnarray*}
\Vert\boldsymbol{u}\Vert_{[\boldsymbol{H}^{p'}_{0}(\mathrm{div},\Omega)]'}&=&\sup _{\boldsymbol{F}\in\boldsymbol{H}^{p'}_{0}(\mathrm{div},\Omega), \boldsymbol{F}\neq 0}\frac{\vert\langle\boldsymbol{u}\,,\,\boldsymbol{F}\rangle_{\Omega}\vert}{\Vert\boldsymbol{F}\Vert_{\boldsymbol{H}^{p'}_{0}(\mathrm{div},\Omega)}}\\
&=&\sup _{\boldsymbol{F}\in\boldsymbol{H}^{p'}_{0}(\mathrm{div},\Omega), \boldsymbol{F}\neq 0}\frac{\vert\langle\boldsymbol{u}\,,\,\lambda\,\boldsymbol{v}-\Delta\boldsymbol{v}-\nabla\theta\rangle_{\Omega}\vert}{\Vert\boldsymbol{F}\Vert_{\boldsymbol{H}^{p'}_{0}(\mathrm{div},\Omega)}}\\
&=&\sup _{\boldsymbol{F}\in\boldsymbol{H}^{p'}_{0}(\mathrm{div},\Omega), \boldsymbol{F}\neq 0}\frac{\vert\langle\lambda\,\boldsymbol{u}-\Delta\boldsymbol{u}-\nabla\pi\,,\,\boldsymbol{v}\rangle_{\Omega}\vert}{\Vert\boldsymbol{F}\Vert_{\boldsymbol{H}^{p'}_{0}(\mathrm{div},\Omega)}}\\
&=&\sup _{\boldsymbol{F}\in\boldsymbol{H}^{p'}_{0}(\mathrm{div},\Omega), \boldsymbol{F}\neq 0}\frac{\vert\langle\boldsymbol{f}\,,\,\boldsymbol{v}\rangle_{\Omega}\vert}{\Vert\boldsymbol{F}\Vert_{\boldsymbol{H}^{p'}_{0}(\mathrm{div},\Omega)}}\\
&\leq & \frac{C(\Omega ,p')}{\vert\lambda\vert}\,\Vert\boldsymbol{f}\Vert_{[\boldsymbol{H}^{p'}_{0}(\mathrm{div},\Omega)]'},
\end{eqnarray*} which is estimate (\ref{estwsp1}).
 \end{proof}

 As consequence of Theorem \ref{weakexistenceth} we have the following corollary
 \begin{coro}\label{existenceweaklaplacian}
 Let $\lambda\in\mathbb{C}^{\ast}$ such that $\mathrm{Re}\,\lambda\geq 0$ and let $\boldsymbol{f}\in[\boldsymbol{H}^{p'}_{0}(\mathrm{div},\Omega)]'$ such that $\mathrm{div}\,\boldsymbol{f}=0$ in $\Omega$ and $\boldsymbol{f}\cdot\boldsymbol{n}=0$ on $\Gamma$. The Problem (\ref{resolventexistence})
 has a unique solution $\boldsymbol{u}\in\boldsymbol{W}^{1,p}(\Omega)$ satisfying the estimate (\ref{estwsp1}).
 \end{coro}
  Next, using Proposition \ref{pr2}, one gets the analyticity of the semi-group generated by the operator $B_{p}$:
 \begin{theo}
 The operator $-B_{p}$ generates a bounded analytic semi-group on the space $[\boldsymbol{H}^{p'}_{0}(\mathrm{div},\Omega)]'_{\sigma,\tau}$.
 \end{theo}
\subsection{Analyticity on $[\boldsymbol{T}^{p'}(\Omega)]'_{\sigma,\tau}$}
\label{semigroupCp}
In this subsection we  prove the analyticity of the semi-group generated by the Stokes operator on $[\boldsymbol{T}^{p'}(\Omega)]'_{\sigma,\tau}$. Using this property we will show the existence of very weak solutions to the Problem (\ref{lens}) with the Navier-type boundary condition ($\ref{nbc}$). The method and arguments in this Section are very similar to those  of the previous one.

 The following theorem gives the very weak solution to Problem (\ref{wsp1}).
 \begin{theo}\label{veryweakexistth}
 Let $\lambda\in\mathbb{C}^{\ast}$ such that $\mathrm{Re}\,\lambda\geq 0$ and let $\boldsymbol{f}\in(\boldsymbol{T}^{p'}(\Omega))'$ then the Problem (\ref{wsp1}) has a unique solution $(\boldsymbol{u},\pi)\in\boldsymbol{L}^{p}(\Omega)\times W^{-1,p}(\Omega)/\mathbb{R}$. Moreover we have the estimate
 \begin{equation}\label{estveryweak}
 \Vert\boldsymbol{u}\Vert_{\boldsymbol{L}^{p}(\Omega)}\,\leq\,\frac{C(\Omega,p)}{\vert\lambda\vert}\,\Vert\boldsymbol{f}\Vert_{(\boldsymbol{T}^{p'}(\Omega))'},
 \end{equation}
 for some constant $C(\Omega,p)>0$ independent of $\lambda$ and $\boldsymbol{f}$.
 \end{theo}
 \begin{proof}
 \textbf{(i)} Thanks to the Green formula (\ref{greenfrtp2}) and to \cite[Theorem 4.15]{Am3} we can easily verify that Problem  (\ref{wsp1}) is equivalent to the problem: Find $\boldsymbol{u}\in\boldsymbol{L}^{p}(\Omega)$ such that for all $\boldsymbol{\varphi}\in\mathbf{D}(A_{p'})$ (given by (\ref{c2dpa})) and for all $\mathit{q}\in W^{1,p'}(\Omega)$ 
 \begin{equation}\label{vwvp}
\begin{array}{r@{~}c@{~}l}
 \lambda\,\int_{\Omega}\boldsymbol{u}\cdot\overline{\boldsymbol{\varphi}}\,\textrm{d}\,x\,-\,\int_{\Omega}\boldsymbol{u}\cdot\Delta\overline{\boldsymbol{\varphi}}\,\textrm{d}\,x &=& \langle\boldsymbol{f},\,\boldsymbol{\varphi}\rangle_{(\boldsymbol{T}^{p'}(\Omega))'\times\boldsymbol{T}^{p'}(\Omega)}\\
 \int_{\Omega} \boldsymbol{u}\cdot\nabla\overline{\mathit{q}}\,\textrm{d}\,x&=&0.
 \end{array}
\end{equation}
  Notice that we recuperate the pressure using the De-Rham argument: if $\boldsymbol{F}\in\boldsymbol{W}^{-2,p}(\Omega)$ verifying $\langle \boldsymbol{F}\,,\,\boldsymbol{v}\rangle_{\boldsymbol{\mathcal{D}'}(\Omega)\times\boldsymbol{\mathcal{D}(\Omega)}}=0$, for all $\boldsymbol{v}\in\boldsymbol{\mathcal{D}}_{\sigma}(\Omega)$ then there exists $\chi\in W^{-1,p}(\Omega)$ such that $\boldsymbol{F}=\nabla\,\chi$.
  
\textbf{(ii)} Let us now solve (\ref{vwvp}). 
As in the proof of Theorem \ref{weakexistenceth}, we know that for all
$\boldsymbol{F}\in\boldsymbol{L}^{p'}(\Omega)$ the
problem:
\begin{equation}\label{adpb}
\left\{
\begin{array}{cccc}
\lambda \boldsymbol{\varphi} - \Delta \boldsymbol{\varphi}-\nabla\theta = \boldsymbol{F},&\mathrm{div}\,\boldsymbol{\varphi} = 0& \qquad \textrm{in}& \Omega, \\
\boldsymbol{\varphi}\cdot \boldsymbol{n} = 0,& \boldsymbol{\mathrm{curl}}\,\boldsymbol{\varphi}\times\boldsymbol{n}=\boldsymbol{0}&\qquad
\textrm{on}&\Gamma,
\end{array}
\right.
\end{equation}
has a unique solution $(\boldsymbol{\varphi},\theta)\in\mathbf{D}(A_{p'})\times W^{1,p'}(\Omega)/\mathbb{R}$
that satisfies the estimate
\begin{equation*}
\|\boldsymbol{\varphi}\|_{\boldsymbol{L}^{p'}(\Omega)}\leq\,\frac{C(\Omega,p')}{|\lambda|}\,\|\boldsymbol{F}\|_{\boldsymbol{L}^{p'}(\Omega)}.
\end{equation*}
Now the following linear mapping:
\begin{eqnarray*}
L&:&\boldsymbol{L}^{p'}(\Omega)\,\longmapsto\,\mathbb{C}\\
&&\boldsymbol{F}\,\longmapsto\,\langle\boldsymbol{f}\,,\,\boldsymbol{\varphi}\rangle_{(\boldsymbol{T}^{p'}(\Omega))'\times\boldsymbol{T}^{p'}(\Omega)},
\end{eqnarray*}
where $\boldsymbol{\varphi}$ is the unique solution of Problem
(\ref{adpb}), satisfies
\begin{equation*}
|L(\boldsymbol{F})|\,\leq\,\|\boldsymbol{f}\|_{(\boldsymbol{T}^{p'}(\Omega))'}\|\boldsymbol{\varphi}\|_{\boldsymbol{L}^{p'}(\Omega)}\,\leq\,\frac{C(\Omega,p')}{|\lambda|}\,\|\boldsymbol{f}\|_{(\boldsymbol{T}^{p'}(\Omega))'}\|\boldsymbol{F}\|_{\boldsymbol{L}^{p'}(\Omega)}.
\end{equation*}
Then there exists a unique
$\boldsymbol{u}\in\boldsymbol{L}^{p}(\Omega)$ such that
\begin{equation*}
L(\boldsymbol{F})\,=\,\int_{\Omega}\boldsymbol{u}\cdot\overline{\boldsymbol{F}}\,\textrm{d}\,x\,=\,\langle\boldsymbol{f}\,,\,\boldsymbol{\varphi}\rangle_{(\boldsymbol{T}^{p'}(\Omega))'\times\boldsymbol{T}^{p'}(\Omega)}
\end{equation*}
and satisfying the estimate (\ref{estveryweak}).
On other worlds $\boldsymbol{u}$ is the unique solution of Problem
(\ref{vwvp}). 
 \end{proof}

As a consequence of Theorem \ref{veryweakexistth} we deduce the existence and uniqueness of very weak solutions to Problem \eqref{resolventexistence}.
\begin{coro}
Let $\lambda\in\mathbb{C}^{\ast} $ such that $\mathrm{Re}\,\lambda\geq 0$ and let $\boldsymbol{f}\in(\boldsymbol{T}^{p'}(\Omega))'$ such that $\mathrm{div}\,\boldsymbol{f}=0$ in $\Omega$ and $\boldsymbol{f}\cdot\boldsymbol{n}=0$ on $\Gamma$. The Problem (\ref{resolventexistence}) has a unique solution $\boldsymbol{u}\in\boldsymbol{L}^{p}(\Omega)$ that satisfies the estimate (\ref{estveryweak}).
\end{coro}
 As described above, using Proposition \ref{pr2} with $w=0$, we have the analyticity of the semi-group generated by the Stokes operator on $[\boldsymbol{T}^{p'}(\Omega)]'_{\sigma,\tau}$:
 \begin{theo}
 The operator $-C_{p}$ is a densely defined operator and it generates a bounded analytic semi-group on $[\boldsymbol{T}^{p'}(\Omega)]'_{\sigma,\tau}$.
 \end{theo}
\section{ Stokes operator with flux boundary  conditions}
\label{Stokes operator with flux boundary  conditions}
As we have already mentioned, the Stokes operator with Navier-type boundary conditions in a non simply connected domain has a non trivial finite dimensional kernel $\boldsymbol{K} _{ \tau  }(\Omega )$. It is then natural to study the Stokes problem on the orthogonal of that kernel. To this end we first consider the Stokes operator on that space.  It turns out  that, under the assumption of Condition H for the domain $\Omega $, for a function $\boldsymbol{u}\in L^p _{ \sigma , \tau  }(\Omega )$, to be in the orthogonal of $\boldsymbol{K} _{ \tau  }(\Omega )$ is equivalent to the condition \eqref{condition2} (cf.  \cite{Am4}, see also \cite{Albaba1}). It is then equivalent for our purpose to consider the Stokes problem, with Navier-type boundary conditions, with the supplementary flux condition \eqref{condition2}.

We  then begin this section  considering $A'_p$, the Stokes operator, with Navier-type boundary conditions, and with flux condition \eqref{condition2}. Its resolvent set is given by the solutions of problem \eqref{*BIS} that we  may recall here:
\begin{equation}\label{resolA'}
\left\{
\begin{array}{r@{~}c@{~}l}
\lambda \boldsymbol{u} - \Delta \boldsymbol{u}\,=\, \boldsymbol{f}, &&\mathrm{div}\,\boldsymbol{u} = 0 \,\,\, \qquad\qquad \mathrm{in} \,\,\, \Omega, \\
\boldsymbol{u}\cdot \boldsymbol{n} = 0, && \boldsymbol{\mathrm{curl}}\,\boldsymbol{u}\times\boldsymbol{n}=\boldsymbol{0}\qquad
\mathrm{on}\,\,\, \Gamma,\\
\langle\boldsymbol{u}\cdot\boldsymbol{n}, 1\rangle_{\Sigma_{j}}\,=\,0,&& 1\leq
j\leq J.
\end{array}
\right.
\end{equation} 

The  addition of the  extra boundary condition on the  cuts $\Sigma_{j}$, $1\leq j \leq J$ makes  the Stokes operator invertible on $L^p _{ \sigma , \tau  }(\Omega )$  with bounded and compact inverse.

 Consider then the space
 \begin{equation}\label{Xp}
\boldsymbol{X}_{p}\,=\,\big\{\boldsymbol{f}\in\boldsymbol{L}^{p}_{\sigma,\tau}(\Omega);\,\,\,\int_{\Omega}\boldsymbol{f}\cdot\overline{\boldsymbol{v}}\,\textrm{d}\,x=0,\,\,\forall\,\,\boldsymbol{v}\in\boldsymbol{K}_{\tau}(\Omega)\big\},
 \end{equation}
 (not to confuse with the space $\boldsymbol{X}^{p}(\Omega)$ defined in the subsection 2.1).
It worth noting that, 
\begin{equation}\label{realtionorthogo}
\forall\,1<p<\infty,\quad\boldsymbol{L}^{p}_{\sigma,\tau}(\Omega)=\boldsymbol{K}_{\tau}(\Omega)\oplus\boldsymbol{X}_{p}\quad\mathrm{and}\quad(\boldsymbol{X}_{p})'=\boldsymbol{X}_{p'}.
\end{equation}

Next, we define the operator
$A'_{p}\,:\,\mathbf{D}(A'_{p})\subset\boldsymbol{X}_{p}\longmapsto\boldsymbol{X}_{p}$
by:
\begin{equation}\label{A'}
\mathbf{D}(A'_{p})=\big\{\boldsymbol{u}\in\mathbf{D}(A_{p});\,\,\langle\boldsymbol{u}\cdot\boldsymbol{n}\,,\,1\rangle_{\Sigma_{j}}=0,\,\,1\leq
j\leq J\big\}
\end{equation}
and $A'_{p}\boldsymbol{u}=A_{p}\boldsymbol{u}$, for all $\boldsymbol{u}\in\mathbf{D}(A'_{p})$.
In other words, the operator $A'_{p}$ is the restriction of the Stokes
operator to the space $\boldsymbol{X}_{p}$. It is clear that when $\Omega$ is simply connected
the Stokes operator $A_{p}$ coincides with the operator $A'_{p}$.
\begin{rmk}\label{cdflux}
\rm{Let $\boldsymbol{u}\in\boldsymbol{L}^{p}_{\sigma,\tau}(\Omega)$, we note that the condition 
\begin{equation}\label{compcond}
\forall\,\,\boldsymbol{v}\in\boldsymbol{K}_{\tau}(\Omega),\qquad \int_{\Omega}\boldsymbol{u}\cdot\overline{\boldsymbol{v}}\,\textrm{d}\,x=0,
\end{equation}
 is equivalent to the condition (see \cite[Lemma 3.2, Corollary 3.4]{Am4}):
\begin{equation}\label{concomp}
\langle\boldsymbol{u}\cdot\boldsymbol{n}\,,\,1\rangle_{\Sigma_{j}}=0,\,\,1\leq
j\leq J.
\end{equation}
}
\end{rmk}

We prove in the following proposition the density of the domain of the operator $A'_{p}$.
\begin{prop}\label{dd2}
The operator $A'_{p}$ is a well defined operator of dense domain.
\end{prop}
\begin{proof} Thanks to
Remark \ref{cdflux} it is clear that
$\mathbf{D}_{p}(A')\subset\boldsymbol{X}_{p}$. Moreover, using Lemma \ref{fg1} we can easily
verify that for all $\boldsymbol{v}\in\boldsymbol{K}_{\tau}(\Omega)$, $\int_{\Omega}\Delta\boldsymbol{u}\cdot\overline{\boldsymbol{v}}\,\textrm{d}\,x\,=\,0$.
As a result $A'\boldsymbol{u}\in\boldsymbol{X}_{p}$
 and $A'$ is a well defined
operator.

Now, for the density, let $\boldsymbol{w}\in\boldsymbol{L}^{p}_{\sigma,\tau}(\Omega)$ such that $\langle\boldsymbol{w}\cdot\boldsymbol{n}\,,\,1\rangle_{\Sigma_{j}}=0$ for all $1\leqslant\,j\,\leqslant J$. We know that there exists a sequence $(\boldsymbol{w}_{k})_{k}$ in $\boldsymbol{\mathcal{D}}_{\sigma}(\Omega)$ such that $\boldsymbol{w}_{k}\longrightarrow\boldsymbol{w}$ in $\boldsymbol{L}^{p}(\Omega)$. As a consequence for all $1\leqslant j \leqslant J$, $\langle\boldsymbol{w}_{k}\cdot\boldsymbol{n}\,,\,1\rangle_{\Sigma_{j}}\longrightarrow\langle\boldsymbol{w}\cdot\boldsymbol{n}\,,\,1\rangle_{\Sigma_{j}}=0$, as $k\rightarrow+\infty.$\\
Now for all $k\in\mathbb{N}$, setting $\widetilde{\boldsymbol{w}}_{k}=\boldsymbol{w}_{k}-\sum_{j=1}^{J}\langle\boldsymbol{w}_{k}\cdot\boldsymbol{n}\,,\,1\rangle_{\Sigma_{j}}\,\widetilde{\boldsymbol{\mathrm{grad}}}\,q_{j}^{\tau}$.
We can easily verify that $(\widetilde{\boldsymbol{w}}_{k})_{k}$ is in $\mathbf{D}_{p}(A')$ and converges to $\boldsymbol{w}$ in $\boldsymbol{L}^{p}(\Omega)$. 
\end{proof}
Next we study the resolvent of the operator $A'_{p}$ and, to this end, consider the problem \eqref{resolA'}
for $\lambda \in \mathbb{C}$. The following results holds:
\begin{theo}\label{exislpA'}
Let $\lambda\in\mathbb{C}$ such that $\mathrm{Re}\,\lambda\geq 0$ and
$\boldsymbol{f}\in\boldsymbol{X}_{p}$. The problem (\ref{resolA'}) has a unique solution $\boldsymbol{u}\in\boldsymbol{W}^{1,p}(\Omega)$ that satisfies the estimates (\ref{estimlpf})-(\ref{curlestlp}).

\noindent In addition,  the solution $\boldsymbol{u}$ belongs to $\boldsymbol{W}^{2,p}(\Omega)$ and satisfies the estimate
\begin{equation}
\Vert\boldsymbol{u}\Vert_{\boldsymbol{W}^{2,p}(\Omega)}\leq\,C(\Omega,p)\,\Vert\boldsymbol{f}\Vert_{\boldsymbol{L}^{p}(\Omega)},
\end{equation}
where $C(\Omega,p)$ is independent of $\lambda$ and $\boldsymbol{f}$.
\end{theo}
This result is proved  for $\lambda =0$ in \cite{Am3} (cf. Proposition 4.3).  On the other, for $\lambda \in \mathbb{C}^*$, $\mathrm{Re}\,\lambda\geq 0$ a similar Theorem has been proved for the problem  \eqref{resolventexistence} in \cite{Albaba} (cf. Theorem  4.8 and Theorem  4.11). Since the proof of Theorem \ref{exislpA'} is very similar  it will be skipped. \\

The following theorem follows:
\begin{theo}\label{analsemi3}
The operator $-A'_{p}$ generates a bounded analytic semi-group on $\boldsymbol{X}_{p}$ for all $1<p<\infty$.
\end{theo}
\begin{rmk}
\rm{
Let $(S(t))_{t\geqslant0}$ be the semi-group generated by $-A'_{p}$ on $\boldsymbol{X}_{p}$. We notice that $S(t)=T(t)_{|\boldsymbol{X}_{p}}$ where $(T(t))_{t\geq
0}$ is the analytic semi-group generated by the operator $-A_{p}$ on
$\boldsymbol{L}^{p}_{\sigma,\tau}(\Omega)$.}
\end{rmk} 
\begin{rmk}
\rm{ Thanks to Proposition \ref{eva} we conclude that the space $\boldsymbol{X}_{2}$ has a Hilbertian basis formed from the
eigenfunctions of the operator $A'_{p}$. Moreover $
\sigma(A_{p})=\sigma(A'_{p})\cup\{0\}$ 
and $\boldsymbol{X}_{2}\,=\,\bigoplus_{k=1}^{+\infty}\boldsymbol{\mathrm{Ker}}(\lambda_{k}\,I\,-\,A_{2})$.
}
\end{rmk}
In a similar way, we now define and give some properties of the Stokes operators with flux boundary conditions defined on the subspaces of  $[\boldsymbol{H}^{p'}_{0}(\mathrm{div},\Omega)]'_{\sigma,\tau}$ and $[\boldsymbol{T}^{p'}(\Omega)]'_{\sigma,\tau}$. Since the proof of these properties are completely similar to those for the operator $A'_p$ we do not write any detail. \\

\textbf{(i)} Consider the space
\begin{equation}\label{Yp}
\boldsymbol{Y}_{p}\,=\,\left\lbrace\boldsymbol{f}\in[\boldsymbol{H}^{p'}_{0}(\mathrm{div},\Omega)]'_{\sigma,\tau};\,\,\forall\,\boldsymbol{v}\in\boldsymbol{K}_{\tau}(\Omega),\,\,\,\langle\boldsymbol{f},\,\boldsymbol{v}\rangle_{\Omega}\,=\,0  \right\rbrace, 
\end{equation} 
where $\langle .\,,\,.\rangle_{\Omega}=\langle .\,,\,.\rangle_{[\boldsymbol{H}^{p'}_{0}(\mathrm{div},\Omega)]'\times\boldsymbol{H}^{p'}_{0}(\mathrm{div},\Omega)}$.\\
We define the operator $B'_{p}\,:\,\mathbf{D}(B'_{p})\subset\boldsymbol{Y}_{p}\longmapsto\boldsymbol{Y}_{p}$
by:
\begin{equation}\label{B'p}
\mathbf{D}(B'_{p})=\big\{\boldsymbol{u}\in\mathbf{D}(B_{p});\,\,\langle\boldsymbol{u}\cdot\boldsymbol{n}\,,\,1\rangle_{\Sigma_{j}}=0,\,\,1\leq
j\leq J\big\}
\end{equation}
and $B'_{p}\boldsymbol{u}=B_{p}\boldsymbol{u}$, for all $\boldsymbol{u}\in\mathbf{D}(B'_{p})$. We recall that $\mathbf{D}(B_{p})$ is given by \eqref{Dpb}.
Observe that, the operator $B'_{p}$ is the restriction of the Stokes
operator to the space $\boldsymbol{Y}_{p}$. It is clear that when $\Omega$ is simply connected
the Stokes operator $B_{p}$ coincides with the operator $B'_{p}$.

We easily verify that $\boldsymbol{f}\in\boldsymbol{Y}_{p}$ and for all $\lambda\in\mathbb{C}^{\ast}$ such that $\mathrm{Re}\lambda\geq0$ the Problem \eqref{resolA'} has a unique solution $\boldsymbol{u}\in\boldsymbol{W}^{1,p}(\Omega)$ satisfying the estimate \eqref{estwsp1}. In other words, the operator $B'_{p}$ is a well defined densely defined operator and $-B'_{p}$ generates a bounded analytic semi-group on $\boldsymbol{Y}_{p}$.\\

\textbf{\textbf{(ii)}} Consider the space
\begin{equation}\label{Zp}
\boldsymbol{Z}_{p}\,=\,\left\lbrace\boldsymbol{f}\in[\boldsymbol{T}^{p'}(\Omega)]'_{\sigma,\tau};\,\,\forall\,\boldsymbol{v}\in\boldsymbol{K}_{\tau}(\Omega),\,\,\,\langle\boldsymbol{f},\,\boldsymbol{v}\rangle_{\Omega}\,=\,0  \right\rbrace, 
\end{equation} 
where $\langle .\,,\,.\rangle_{\Omega}=\langle .\,,\,.\rangle_{[\boldsymbol{T}^{p'}(\Omega)]'\times\boldsymbol{T}^{p'}(\Omega)}$.\\
We define the operator $C'_{p}\,:\,\mathbf{D}(C'_{p})\subset\boldsymbol{Z}_{p}\longmapsto\boldsymbol{Z}_{p}$
by:
\begin{equation}\label{C'p}
\mathbf{D}(C'_{p})=\big\{\boldsymbol{u}\in\mathbf{D}(C_{p});\,\,\langle\boldsymbol{u}\cdot\boldsymbol{n}\,,\,1\rangle_{\Sigma_{j}}=0,\,\,1\leq
j\leq J\big\}
\end{equation}
and $C'_{p}\boldsymbol{u}=C_{p}\boldsymbol{u}$, for all $\boldsymbol{u}\in\mathbf{D}(C'_{p})$. We recall that $\mathbf{D}(C_{p})$ is given by \eqref{Dpc}.
Notice that, the operator $C'_{p}$ is the restriction of the Stokes
operator to the space $\boldsymbol{Z}_{p}$. Similarly, when $\Omega$ is simply connected
the Stokes operator $C_{p}$ coincides with the operator $C'_{p}$.

We  verify that $\boldsymbol{f}\in\boldsymbol{Z}_{p}$ and for all $\lambda\in\mathbb{C}^{\ast}$ such that $\mathrm{Re}\lambda\geq0$ the Problem \eqref{resolA'} has a unique solution $\boldsymbol{u}\in\boldsymbol{L}^{p}(\Omega)$ satisfying the estimate \eqref{estveryweak}. In other words, the operator $C'_{p}$ is a well defined densely defined operator and $-C'_{p}$ generates a bounded analytic semi-group on $\boldsymbol{Z}_{p}$.

\section{Complex and fractional powers of the Stokes operator}
\label{powers}
In this section we are interested in the  study of the complex and the fractional powers of the Stokes operators $A_{p}$ and $A'_{p}$ on $\boldsymbol{L}^{p}_{\sigma,\tau}(\Omega)$ and $\boldsymbol{X}_{p}$ respectively. Since theses operators generates  bounded analytic semi-groups in their corresponding Banach spaces (see Theorems \ref{analsemi2} and \ref{analsemi3}), they are in particular non-negative operators. It then follows from the results in \cite{Ko} and in \cite{Tri} that their powers $A^{\alpha}_{p}$ and $(A'_{p})^{\alpha}$, $\alpha\in\mathbb{C}$, are well, densely defined and  closed linear operators on $\boldsymbol{L}^{p}_{\sigma,\tau}(\Omega)$ and $\boldsymbol{X}_{p}$ with domain $\mathbf{D}(A^{\alpha}_{p})$ and $\mathbf{D}((A'_{p})^{\alpha})$ respectively.

The purpose of this section is to prove some properties and estimates for these operators $A^{\alpha}_{p}$ and $(A'_{p})^{\alpha}$. Since it will be needed, we  also obtain in this section a result on the purely imaginary powers of  $(I+A_p)$ that easily follows  from previous results in \cite{Geissert}.

The fractional powers of the Stokes operator with Dirichlet boundary conditions  on a  bounded domains are studied in detail in  \cite{GiGa2}. In that case the Stokes operator is bijective with bounded inverse. This is  still true for the Stokes operator with Navier-type and flux boundary conditions $A'_{p}$ but not true for the Stokes operator with Navier-type boundary conditions $A_p$.
 
\subsection{Pure imaginary powers.}
\label{Imaginary}
In this section we prove that the pure imaginary powers of the operators $(I+L)$ and  $L'$, for $L=A_p, B_p, C_p$,  are bounded.The proofs are based on Lemma A2 in \cite{GiGa4} and some results in \cite{Geissert}  about the operator $\Delta _M$ defined as follows : 
$$\Delta_{M}:\mathbf{D}(\Delta_{M})\subset\boldsymbol{L}^{p}(\Omega)\longmapsto\boldsymbol{L}^{p}(\Omega),\,$$ where 
\begin{equation}\label{DomainHodgeLaplacian}
\mathbf{D}(\Delta_{M})=\big\{
\boldsymbol{u}\in\boldsymbol{W}^{2,p}(\Omega);\,
\boldsymbol{u}\cdot\boldsymbol{n}\,=\,0,\,\,\,\boldsymbol{\mathrm{curl}}\,\boldsymbol{u}\times\boldsymbol{n}\,=\,\boldsymbol{0}\,\,\textrm{on}\,\,\Gamma\big\}
\end{equation}
and 
\begin{equation}\label{HodgeLaplacian}
\forall\,\boldsymbol{u}\in\mathbf{D}(\Delta_{M}),\quad\Delta_{M}\boldsymbol{u}=\Delta\boldsymbol{u}\,\,\,\mathrm{in}\,\,\Omega.
\end{equation}
As it was noticed in \cite{Geissert}, (see also \cite{Albaba}):
\begin{equation}\label{DlapDAP}
\mathbf{D}(\Delta_{M})\cap\boldsymbol{L}^{p}_{\sigma,\tau}(\Omega)=\mathbf{D}(A_{p}),
\end{equation}
\begin{equation}\label{restrictionDeltaM}
\forall \, \boldsymbol{u}\in\mathbf{D}(A_{p}),\quad A_{p}\boldsymbol{u}=-\Delta_{M}\,\boldsymbol{u}\quad\mathrm{in}\,\,\Omega
\end{equation}
and
\begin{equation}\label{relationDeltaMAp}
R(\lambda,\,\Delta_{M})(\boldsymbol{L}^{p}_{\sigma,\tau}(\Omega))\subset\boldsymbol{L}^{p}_{\sigma,\tau}(\Omega).
\end{equation}

Similarly we have:
\begin{equation}\label{DeltaMDA'p}
\mathbf{D}(\Delta_{M})\cap\boldsymbol{X}_{p}=\mathbf{D}(A'_{p}),
\end{equation}
\begin{equation}\label{restrictionDeltaMA'p}
\forall \, \boldsymbol{u}\in\mathbf{D}(A'_{p}),\quad A_{p}\boldsymbol{u}=-\Delta_{M}\,\boldsymbol{u}\quad\mathrm{in}\,\,\Omega
\end{equation}
and
\begin{equation}\label{relationDeltaMA'p}
R(\lambda,\,\Delta_{M})(\boldsymbol{X}_{p})\subset\boldsymbol{X}_{p}.
\end{equation}

Our first result in this section is the following:

\begin{theo}\label{pureimg1+lap}
There exists an angle $0<\theta_{0}<\pi/2$ and a constant $M>0$ such that for all $s\in\mathbb{R}$ we have
\begin{equation}\label{estimpur1+lap}
\Vert(I+A_{p})^{i\,s}\Vert_{\mathcal{L}(\boldsymbol{L}^{p}_{\sigma,\tau}(\Omega))}\,\leq\,M\,e^{\vert s\vert\,\theta_{0}}.
\end{equation}
\end{theo}
\begin{proof} 
Using Theorem 3.1 and Remark 3.2 in  \cite{Geissert} with $\lambda=1,\,$ we deduce that $(I-\Delta_{M})$ has a bounded $\mathcal{H}^{\infty}$- calculus on $\boldsymbol{L}^{p}(\Omega)$. Then, there exist an angle $0<\theta_{0}<\pi/2$ and a constant $M>0$ such that for all $s\in\mathbb{R}$
\begin{equation*}
\Vert(I-\Delta_{M})^{i\,s}\Vert_{\mathcal{L}(\boldsymbol{L}^{p}(\Omega))}\,\leq\,M\,e^{\vert s\vert\,\theta_{0}}.
\end{equation*}
(For the definition of $\mathcal{H}^{\infty}$- calculus of an operator in a Banach space and its relation with the pure imaginary powers of this operator see \cite[Section 2]{Hinfiniecalculus} for instance). Using now \eqref{DlapDAP}-\eqref{relationDeltaMAp}, the estimate \eqref{estimpur1+lap} follows.
\end{proof}

\begin{rmk}
\rm{ The results in \cite{Geissert} are proved  under the hypothesis that the domain $\Omega$ is bounded with a uniform $C^{3}$-  boundary. On the other hand, it is well known in elliptic theory (cf. Grsivard \cite{Gris}) that the same regularity results hold if  the $C^k$ regularity is replaced by the regularity $C^{k-1, 1}$. Notice that this is precisely our hypothesis with $k=3$.
}
\end{rmk}

In the following Proposition we prove that the pure imaginary powers of the operators $(I+B_{p})$ and $(I+C_{p})$ are bounded on $[\boldsymbol{H}^{p'}_{0}(\mathrm{div},\Omega)]'_{\sigma,\tau}$ (given by \eqref{hpdivsigmatau}) and on $[\boldsymbol{T}^{p'}(\Omega)]'_{\sigma,\tau}$ (given by \eqref{tp'sigmatau}) respectively.  We recall that the operators $B_{p}$ and $C_{p}$ given by  \eqref{b2} and \eqref{C2} respectively are the extensions of the Stokes operator to the spaces $[\boldsymbol{H}^{p'}_{0}(\mathrm{div},\Omega)]'_{\sigma,\tau}$ and  $[\boldsymbol{T}^{p'}(\Omega)]'_{\sigma,\tau}$ respectively.

\begin{prop}\label{PureimgI+Bp+Cp}
There exists $0<\theta_{0}<\pi/2$ and a constant $C>0$ such that for all $s\in\mathbb{R}$
\begin{equation}\label{pureimphdiv}
\Vert(I+B_{p})^{i\,s}\Vert_{\mathcal{L}([\boldsymbol{H}^{p'}_{0}(\mathrm{div},\Omega)]'_{\sigma,\tau})}\,\leq\,C\,e^{\vert s\vert\,\theta_{0}} 
\end{equation}
and
\begin{equation}\label{pureimptp}
\Vert(I+C_{p})^{i\,s}\Vert_{\mathcal{L}([\boldsymbol{T}^{p'}(\Omega)]'_{\sigma,\tau})}\,\leq\,C\,e^{\vert s\vert\,\theta_{0}} .
\end{equation}
\end{prop}
\begin{proof}
We will prove estimate \eqref{pureimphdiv}, estimate \eqref{pureimptp} follows in the same way .  Consider the operator $B_{p}$ defined in \eqref{b2} and let $\boldsymbol{f}\in\boldsymbol{L}^{p}_{\sigma,\tau}(\Omega)$. Notice that
\begin{eqnarray*}
\Vert(I+B_{p})^{i\,s}\boldsymbol{f}\Vert_{[\boldsymbol{H}^{p'}_{0}(\mathrm{div},\Omega)]'}=\Vert(I+A_{p})^{i\,s}\boldsymbol{f}\Vert_{[\boldsymbol{H}^{p'}_{0}(\mathrm{div},\Omega)]'}&\leq & \Vert(I+A_{p})^{i\,s}\boldsymbol{f}\Vert_{\boldsymbol{L}^{p}(\Omega)}\\
&\leq & C\,e^{\vert s\vert\,\theta_{0}}\Vert\boldsymbol{f}\Vert_{\boldsymbol{L}^{p}(\Omega)}.
\end{eqnarray*}
This means that for all $s\in\mathbb{R},\,$  the operator $(I+B_{p})^{i\,s}$ is bounded from $\boldsymbol{L}^{p}_{\sigma,\tau}(\Omega)$ into $[\boldsymbol{H}^{p'}_{0}(\mathrm{div},\Omega)]'_{\sigma,\tau}$. Next, observe that 
$$\boldsymbol{\mathcal{D}}_{\sigma}(\Omega)\subset\boldsymbol{L}^{p}_{\sigma,\tau}(\Omega)\subset[\boldsymbol{H}^{p'}_{0}(\mathrm{div},\Omega)]'_{\sigma,\tau}.$$
As a result, using the density of $\boldsymbol{\mathcal{D}}_{\sigma}(\Omega)$ in $[\boldsymbol{H}^{p'}_{0}(\mathrm{div},\Omega)]'_{\sigma,\tau}$ (see Proposition \ref{denshpdiv}) and the Hahn-Banach theorem we can extend $(I+B_{p})^{i\,s}$ to a bounded linear operator on $[\boldsymbol{H}^{p'}_{0}(\mathrm{div},\Omega)]'_{\sigma,\tau}$ and we deduce deduce estimate \eqref{pureimphdiv}.
\end{proof}

We consider now the Stokes operators with flux  condition $A'_{p}$, $B'_{p}$, $C'_p$ on $\boldsymbol{X_{p}}$, $\boldsymbol{Y_{p}}$, $\boldsymbol{Z_{p}}$ respectively. Using that the operator these operators are densely defined and invertible with bounded inverse, we prove that their pure imaginary powers are bounded in  $\boldsymbol{X_{p}}$, $\boldsymbol{Y_{p}}$, $\boldsymbol{Z_{p}}$ respectively. To this end, we first show the following auxiliary Proposition for whose proof we use  again the results of \cite{Geissert}. 

\begin{prop}\label{Lapimpowerproposition}
Suppose that $\Omega$ is strictly star shaped with respect to one of its points and let $1<p<\infty$.  There exist an angle $\theta_{0}$ and a constant $M>0$ such that, for all $s\in  \mathbb{R}$:
\begin{eqnarray}
&&\Vert(A'_{p})^{i\,s}\Vert_{\mathcal{L}(\boldsymbol{X}_{p})}+\Vert(B'_{p})^{i\,s}\Vert_{\mathcal{L}(\boldsymbol{Y}_{p})}+\Vert(C'_{p})^{i\,s}\Vert_{\mathcal{L}(\boldsymbol{Z}_{p})}\,\leq\,M\,e^{\vert s\vert\,\theta_{0}}.\label{estimpurlapproposition}\\
&&\Big\Vert\Big(\lambda\,I-\Delta_{M}\Big)^{i\,s}\Big\Vert_{\mathcal{L}(\boldsymbol{L}^{p}(\Omega))}\,\leq\,M\,e^{\vert s\vert\,\theta_{0}},
\,\,\,\forall \lambda >0. \label{estnuDeltaM*}
\end{eqnarray}
\end{prop}
\begin{rmk}
\rm{
It follows from the results in  \cite{Geissert} that the imaginary powers of $(\lambda\,I-\Delta_{M})$ are bounded in $\mathcal{L}(\boldsymbol{L}^{p}(\Omega))$. We   explicitly write the estimate (\ref{estnuDeltaM*}) in the statement of Proposition \ref{Lapimpowerproposition} in order to emphasize that the constants $M$ and $\theta_0$ are independent of $\lambda >0$.}
\end{rmk}
\begin{proof} Estimate  \eqref{estimpurlapproposition} is proved  by showing that it holds separately for each of the three terms in its left hand side. Since the proof is the same for the three terms, we only write the details for  $\Vert(A'_{p})^{i\,s}\Vert_{\mathcal{L}(\boldsymbol{X}_{p})}$.

The first step of the proof is to show the existence of  constants $C>0$ and $\theta_0\in (0, \pi /2)$ such that:
\begin{equation}\label{estnuDeltaMBIS}
\Big\Vert\Big(\frac{1}{\mu^{2}}\,I-\Delta_{M}\Big)^{i\,s}\Big\Vert_{\mathcal{L}(\boldsymbol{L}^{p}(\Omega))}\,\leq\,C\,e^{\vert s\vert\,\theta_{0}}
\end{equation}
and
\begin{equation}\label{estnu*}
\Big\Vert\Big(\frac{1}{\mu^{2}}\,I+A'_{p}\Big)^{i\,s}\Big\Vert_{\mathcal{L}(\boldsymbol{X}_{p})}\,\leq\,C\,e^{\vert s\vert\,\theta_{0}},
\end{equation}
for all  $\mu>0$. \\
Since $\Omega$ is strictly star shaped with respect to one of its points, then after translation in $\mathbb{R}^{3}$, we can suppose that this point is $0$. This amounts to say that 
\begin{equation*}
\forall\,\mu>1,\qquad\mu\,\overline{\Omega}\subset\Omega,\,\,\,\forall\,0\leq\mu<1\qquad\textrm{and}\,\,\,\,\overline{\Omega}\subset\mu\,\Omega.
\end{equation*} 
Here we take $\mu>1$ and we set $\Omega_{\mu}=\mu\,\Omega$. 

\noindent The proof is based on the scaling transformation 
\begin{equation}\label{SMU}
\forall\,x\in\Omega_{\mu},\qquad(S_{\mu}\boldsymbol{f})(x)=\boldsymbol{f}(x/\mu),\qquad\boldsymbol{f}\in\boldsymbol{L}^{p}(\Omega).
\end{equation}
As in the proof of Theorem A1 in \cite{GiGa4} we can easily verify that 
\begin{equation*}
-\mu^{2}\Delta_{M}\,=\,S_{\mu}(-\Delta_{M})\,S_{\mu}^{-1},\qquad I-\mu^{2}\Delta_{M}\,=\,S_{\mu}(I-\Delta_{M})\,S^{-1}_{\mu}.
\end{equation*}
We recall that the operator $\Delta_{M}$ is defined by \eqref{DomainHodgeLaplacian}-\eqref{HodgeLaplacian}.

\noindent Similarly we can also verify that
\begin{equation*}
\mu^{2}\,A'_{p}\,=\,S_{\mu}\,A'_{p}\,S_{\mu}^{-1},\qquad I+\mu^{2}A'_{p}\,=\,S_{\mu}(I+A'_{p})\,S^{-1}_{\mu}.
\end{equation*}
As a result for all $z\in\mathbb{C}$ using \eqref{forintimpur} we have,
\begin{equation*}
(I-\mu^{2}\Delta_{M})^{z}\,=\,S_{\mu}
(I-\Delta_{M})^{z}S^{-1}_{\mu} \quad  \mathrm{and} \quad (I+\mu^{2}\,A'_{p})^{z}\,=\,S_{\mu}
(I+A'_{p})^{z}S^{-1}_{\mu}.
\end{equation*}
Thus for all $z\in\mathbb{C}$ we have
$$\Vert(I-\mu^{2}\Delta_{M})^{z}\Vert_{\mathcal{L}(\boldsymbol{L}^{p}(\Omega))}=\Vert S_{\mu}
(I-\Delta_{M})^{z}S^{-1}_{\mu}\Vert_{\mathcal{L}(\boldsymbol{L}^{p}(\Omega))}\leq\Vert(I-\Delta_{M})^{z}\Vert_{\mathcal{L}(\boldsymbol{L}^{p}(\Omega))}$$
and
$$\Vert(I+\mu^{2}\,A'_{p})^{z}\Vert_{\mathcal{L}(\boldsymbol{X}_{p})}=\Vert S_{\mu}
(I+A'_{p})^{z}S^{-1}_{\mu}\Vert_{\mathcal{L}(\boldsymbol{X}_{p})}\leq\Vert(I+A'_{p})^{z}\Vert_{\mathcal{L}(\boldsymbol{X}_{p})}.$$
Using Theorem 3.1 and Remark 3.2 in  \cite{Geissert}, respectively to Theorem \ref{pureimg1+lap}, we deduce that there exist $0<\theta_{1},\,\theta_{2} <\pi/2$ and constants $M_{1},\,M_{2}>0$ such that : 
\begin{equation}\label{complexpowerDeltaM}
\forall\,s\in\mathbb{R},\quad\Vert(I-\mu^{2}\Delta_{M})^{i\,s}\Vert_{\mathcal{L}(\boldsymbol{L}^{p}(\Omega))}\,\leq\,M_{1}\,e^{\vert s\vert\,\theta_{1}},
\end{equation}
and
\begin{equation}\label{complexpower4}
\forall\,s\in\mathbb{R},\quad\Vert(I+\mu^{2}\,A'_{p})^{i\,s}\Vert_{\mathcal{L}(\boldsymbol{X}_{p})}\,\leq\,M_{2}\,e^{\vert s\vert\,\theta_{2}},
\end{equation}
where the constants $M_{1}$ in \eqref{complexpowerDeltaM} and $M_{2}$ in \eqref{complexpower4} are independents of $\mu$. Since
\begin{equation*}
\Big(\frac{1}{\mu^{2}}\,I-\Delta_{M}\Big)^{i\,s}\,=\,\frac{1}{\mu^{2\,i\,s}}\,(I-\mu^{2}\Delta_{M})^{i\,s}\quad\mathrm{and}\quad\Big(\frac{1}{\mu^{2}}\,I+A'_{p}\Big)^{i\,s}\,=\,\frac{1}{\mu^{2\,i\,s}}\,(I+\mu^{2}A'_{p})^{i\,s}
\end{equation*}
\eqref{estnuDeltaMBIS} and  \eqref{estnu*} follow.\\
Of course,  \eqref{estnuDeltaM*} follows from  \eqref{estnuDeltaMBIS}.  On the other hand, since by  Proposition 4.3 in \cite{Am3} and Proposition \ref{dd2} above,  the range and the domains  of $A'_p$, $B'_p$ and $C'_p$ are dense in $\boldsymbol{X}_{p}$, $\boldsymbol{Y}_{p}$, $\boldsymbol{Z}_{p}$ respectively, we may apply Lemma A2 of \cite{GiGa4} and obtain that, for all $\boldsymbol{f}\in\mathbf{D}(A'_{p})$
\begin{equation}\label{estnu**}
\Vert(A'_{p})^{i\,s}\boldsymbol{f}\Vert_{\boldsymbol{L}^{p}(\Omega)}\,=\,\lim _{\mu\rightarrow+\infty}\Big\Vert\Big(\frac{1}{\mu^{2}}\,I+A'_{p}\Big)^{i\,s}\boldsymbol{f}\Big\Vert_{\boldsymbol{L}^{p}(\Omega)}.
\end{equation}
As a result we deduce from  \eqref{estnu*} and \eqref{estnu**} that \eqref{estimpurlapproposition} holds for all $\boldsymbol{f}\in\mathbf{D}(A'_{p})$. By the density of $\mathbf{D}(A'_{p})$ in $\boldsymbol{X}_{p}$ (see Proposition \ref{dd2})  it then follows that  \eqref{estimpurlapproposition} holds for all  $\boldsymbol{f}$ in $\boldsymbol{X}_{p}$. 
\end{proof}
\begin{rmk}
\rm{\textbf{(i)} Notice that  estimate  \eqref{estnu*} is also true if we replace  $A'_{p}$  by $A_p$. However, since  the range of $A_{p}$ is not dense in  $\boldsymbol{L}^{p}_{\sigma,\tau}(\Omega)$ it is not possible to apply Lemma A2 of \cite{GiGa4} to pass to the limit as $\mu \to \infty$.\\

\noindent\textbf{(ii)} When $\Omega$ is simply connected the operator $A_{p}$ coincides with the operator $A'_{p}$ and it is invertible with bounded inverse in $\boldsymbol{L}^{p}_{\sigma,\tau}(\Omega)$. As a result we recover the result of \cite{Geissert} and we deduce the boundedness of the pure imaginary powers of $A_{p}$ in $\boldsymbol{L}^{p}_{\sigma,\tau}(\Omega)$.
}
\end{rmk}

For a general domain $\Omega\,$ of Class $C^{2,1}$, not necessarily strictly star shaped with respect to one of its points, we use that (see \cite{Ber} for instance), a bounded Lipschitz-Continuous open set is the union of a finite number of star-shaped, Lipschitz-continuous open sets. The idea is then to apply the argument above to each of these sets in order to derive the desired result on the entire domain. However, the divergence-free condition of a function $\boldsymbol{f}\in\boldsymbol{L}^{p}_{\sigma,\tau}(\Omega)$ is not preserved under the cut-off procedure and this process is non-trivial. This is done in the following Theorem.
\begin{theo}\label{Lapimpower}
There exist an angle $0<\theta_{0}<\pi/2$  and a constant $M>0$ such that for all $s\in\mathbb{R}$ we have
\begin{eqnarray}\label{estimpurlap}
\Vert(A'_{p})^{i\,s}\Vert_{\mathcal{L}(\boldsymbol{X}_{p})}+\Vert(B'_{p})^{i\,s}\Vert_{\mathcal{L}(\boldsymbol{Y}_{p})}
+\Vert(C'_{p})^{i\,s}\Vert_{\mathcal{L}(\boldsymbol{Z}_{p})}\,\leq\,M\,e^{\vert s\vert\,\theta_{0}}
\end{eqnarray}
\end{theo}

\begin{proof}
As in the proof of Proposition \ref{Lapimpowerproposition},  we  prove Theorem  \ref{Lapimpower} by showing  that estimate \eqref{estimpurlap} holds separately for each term in the left hand side. Since the proof is the same for the three terms, we only write the details for  $\Vert(A'_{p})^{i\,s}\Vert_{\mathcal{L}(\boldsymbol{X}_{p})}$.

Let $(\Theta_{j})_{j\in J}$ be an open covering of $\Omega$ by a finite number of star-shaped open sets and let us consider a partition of unity $(\varphi_{j})_{j\in J}$ subordinated to the covering $(\Omega_{j})_{j\in J}$ where for all $j\in J,\,$ $\Omega_{j}=\Theta_{j}\cap\Omega$. This means that
\begin{equation*}
\forall j\in J,\qquad \mathrm{Supp}\varphi_{j}\subset\Omega_{j}
\end{equation*}
and
\begin{equation*}
\sum_{j\in J}\varphi_{j}=1,\qquad\varphi_{j}\in\mathcal{D}(\Omega_{j}).
\end{equation*}
Let $\boldsymbol{f}\in\boldsymbol{X}_{p}$, then $\boldsymbol{f}$ can be written as 
\begin{equation*}
\boldsymbol{f}=\sum_{j\in J}\boldsymbol{f}_{j}, \qquad\forall j\in J,\quad\boldsymbol{f}_{j}=\varphi_{j}\boldsymbol{f}.
\end{equation*}
Notice that for all $j\in J,\,$ $\boldsymbol{f}_{j}$ is not necessarily a divergence free function.

Let $\mu>0\,$ and let $\,s\in\mathbb{R}.\,$  From \eqref{DeltaMDA'p}-\eqref{relationDeltaMA'p} we know that
\begin{equation*}
\Big(\frac{1}{\mu^{2}}\,I+A'_{p}\Big)^{i s}\boldsymbol{f}\,=\,\Big(\frac{1}{\mu^{2}}\,I-\Delta_{M}\Big)^{i s}\boldsymbol{f}\,=\,\sum_{j\in J}\Big(\frac{1}{\mu^{2}}\,I-\Delta_{M}\Big)^{i s}\boldsymbol{f}_{j}.
\end{equation*}
As a result, one has
\begin{eqnarray}
\Big\Vert\Big(\frac{1}{\mu^{2}}\,I+A'_{p}\Big)^{i s}\boldsymbol{f}\Big\Vert_{\boldsymbol{L}^{p}(\Omega)}&\leq&\sum_{j\in J}\Big\Vert\Big(\frac{1}{\mu^{2}}\,I-\Delta_{M}\Big)^{i s}\boldsymbol{f}_{j}\Big\Vert_{\boldsymbol{L}^{p}(\Omega)}\nonumber\\
&=&\sum_{j\in J}\Big\Vert\Big(\frac{1}{\mu^{2}}\,I-\Delta_{M}\Big)^{i s}\boldsymbol{f}_{j}\Big\Vert_{\boldsymbol{L}^{p}(\Omega_{j})}\nonumber
\end{eqnarray}
Since for all $j\in J,\,$ the domain $\Omega_{j}\,$ is strictly star shaped with respect to one of its points, then using \eqref{estnuDeltaM*} we have
\begin{eqnarray*}
\Big\Vert\Big(\frac{1}{\mu^{2}}\,I+A'_{p}\Big)^{i s}\boldsymbol{f}\Big\Vert_{\boldsymbol{L}^{p}(\Omega)}&\leq&\,e^{\vert s\vert\,\theta_{0}}\sum_{j\in J}\,C_{j}\,\Vert\boldsymbol{f}_{j}\Vert_{\boldsymbol{L}^{p}(\Omega_{j})}\nonumber\\
&\leq& C(\Omega,p)\,e^{\vert s\vert\,\theta_{0}}\Vert\boldsymbol{f}\Vert_{\boldsymbol{L}^{p}(\Omega)}\label{estimI+Apzdomainenonétoilé}
\end{eqnarray*}
with a constant $C(\Omega,p)$ independent of $\mu$ and $\boldsymbol{f}$. As a result one has
\begin{equation*}\label{estimI+Apnonetoile2}
\Big\Vert\Big(\frac{1}{\mu^{2}}\,I+A'_{p}\Big)^{i s}\Big\Vert_{\mathcal{L}(\boldsymbol{X}_{p})}\,\leq\,C(\Omega,p)\,e^{\vert s\vert\,\theta_{0}}.
\end{equation*}

Thus as in the proof of Theorem \ref{pureimg1+lap}, using \cite[Lemma A2]{GiGa4} we deduce that for all $\boldsymbol{f}\in\mathbf{D}(A'_{p})$
\begin{equation}\label{estnu**nonetoile}
\Vert(A'_{p})^{i\,s}\boldsymbol{f}\Vert_{\boldsymbol{L}^{p}(\Omega)}\,=\,\lim _{\mu\rightarrow+\infty}\Big\Vert\Big(\frac{1}{\mu^{2}}\,I+A'_{p}\Big)^{i\,s}\boldsymbol{f}\Big\Vert_{\boldsymbol{L}^{p}(\Omega)}.
\end{equation}
This means that \eqref{estimpurlap} hold for all $\boldsymbol{f}\in\mathbf{D}(A'_{p})$.
Using the density of $\mathbf{D}(A'_{p})$ in $\boldsymbol{X}_{p}$ we deduce our result in $\boldsymbol{X}_{p}$. 
\end{proof}

\begin{rmk}
\rm{We can also prove that there exists $0<\theta_{0}<\pi/2$ and a constant $C>0$ such that for all $s\in\mathbb{R}$
\begin{equation*}\label{pureimpB'p}
\Vert(B'_{p})^{i\,s}\Vert_{\mathcal{L}(\boldsymbol{Y}_{p})}\,\leq\,C(\Omega,p)\,e^{\vert s\vert\,\theta_{0}} 
\end{equation*}
and
\begin{equation*}\label{pureimpC'p}
\Vert(C'_{p})^{i\,s}\Vert_{\mathcal{L}(\boldsymbol{Z}_{p})}\,\leq\,C(\Omega,p)\,e^{\vert s\vert\,\theta_{0}} .
\end{equation*}
We recall that the operator $B'_{p}$ and $C'_{p}$ given by  \eqref{B'p} and \eqref{C'p} respectively are the extensions of the Stokes operator to the spaces $\boldsymbol{Y}_{p}$ (defined by \eqref{Yp}) and  $\boldsymbol{Z}_{p}$ (defined by \eqref{Zp}) respectively and they are invertible with bounded inverses.
}
\end{rmk}

\subsection{Domains of fractional powers.}
\label{Domains of fractional powers of the Stokes operator}
For all $\alpha\in\mathbb{R}$, the map $\boldsymbol{v}\longmapsto\Vert (A'_{p})^{\alpha}\,\boldsymbol{v}\Vert_{\boldsymbol{L}^{p}(\Omega)}$ is a norm on $\mathbf{D}((A'_{p})^{\alpha})$. This is due to the fact that  (cf. \cite[Theorem 1.15.2, part (e)]{Tri}), the operator $A'_{p}$ has a bounded inverse and thus for all $\alpha\in\mathbb{C}^{\ast}$, the operator $(A'_{p})^{\alpha}$ is an isomorphism from $\mathbf{D}((A'_{p})^{\alpha})$ to $\boldsymbol{X}_{p}$.

Consider the space 
\begin{equation}\label{vptflux}
\boldsymbol{V}^{p}_{\sigma,\tau}(\Omega)=\{\boldsymbol{v}\in\boldsymbol{X}^{p}_{\sigma,\tau}(\Omega);\,\langle\boldsymbol{v}\cdot\boldsymbol{n}\,,\,1\rangle_{\Sigma_{j}}=0,\,\,1\leq
j\leq J\},
\end{equation}
with $\boldsymbol{X}^{p}_{\sigma,\tau}(\Omega)$ is defined by \eqref{vpt}. Thanks to the work of \cite{Am3, Am4} we know that, for all $\boldsymbol{v}\in\boldsymbol{V}^{p}_{\sigma,\tau}(\Omega)$ the norm of $\Vert\boldsymbol{v}\Vert_{\boldsymbol{W}^{1,p}(\Omega)}$ is equivalent to the norm $\Vert\boldsymbol{\mathrm{curl}}\,\boldsymbol{u}\Vert_{\boldsymbol{L}^{p}(\Omega)}$.
The following theorem characterizes the domain of $(A'_{p})^{\frac{1}{2}}$.
\begin{theo}\label{DA1/2}
For all $1<p<\infty$,
$\mathbf{D}((A'_{p})^{\frac{1}{2}})\,=\,\boldsymbol{V}^{p}_{\sigma,\tau}(\Omega)$ with equivalent norms. Furthermore, for every $\boldsymbol{u}\in\mathbf{D}((A'_{p})^{\frac{1}{2}})$, the norm $\Vert (A'_{p})^{\frac{1}{2}}\boldsymbol{u}\Vert_{\boldsymbol{L}^{p}(\Omega)}$ is a norm on $\mathbf{D}((A'_{p})^{\frac{1}{2}})$ which is equivalent to the norm $\Vert\boldsymbol{\mathrm{curl}}\,\boldsymbol{u}\Vert_{\boldsymbol{L}^{p}(\Omega)}$. In other words, there exists two constants $C_{1}$ and $C_{2}$ such that for all $\boldsymbol{u}\in\mathbf{D}((A'_{p})^{\frac{1}{2}})$ 
\begin{equation*}
\Vert (A'_{p})^{\frac{1}{2}}\boldsymbol{u}\Vert_{\boldsymbol{L}^{p}(\Omega)}\leq C_{1}\Vert\boldsymbol{\mathrm{curl}}\,\boldsymbol{u}\Vert_{\boldsymbol{L}^{p}(\Omega)}\leq C_{2}\Vert (A'_{p})^{\frac{1}{2}}\boldsymbol{u}\Vert_{\boldsymbol{L}^{p}(\Omega)}.
\end{equation*}
\end{theo}
\begin{proof}
Thanks to Theorem \ref{Lapimpower} we know that that the pure imaginary powers of $A'_{p}$ are bounded on $\boldsymbol{X}_{p} $ and satisfy estimate (\ref{estimpurlap}). As a result thanks to Theorem \ref{domfracpower} we have
\begin{equation*}
\mathbf{D}((A'_{p})^{\frac{1}{2}})\,=\,\left[\mathbf{D}(A'_{p});\,\boldsymbol{X}_{p} \right]_{1/2}.
\end{equation*}
Consider now a function $\boldsymbol{u}\in\mathbf{D}(A'_{p})$, set $\boldsymbol{z}=\boldsymbol{\mathrm{curl}}\,\boldsymbol{u}$ and $\boldsymbol{U}=(\boldsymbol{u},\boldsymbol{z})$. It is clear that $\boldsymbol{z}\in\boldsymbol{H}^{p}_{0}(\boldsymbol{\mathrm{curl}},\Omega)$ and using Lemma \ref{con1} we deduce that $\boldsymbol{z}\in\boldsymbol{X}^{p}_{N}(\Omega)\hookrightarrow\boldsymbol{W}^{1,p}(\Omega)$ and $\boldsymbol{U}\in\boldsymbol{X}_{p}\times\boldsymbol{W}^{1,p}(\Omega)$. On the other hand if $\boldsymbol{u}\in\boldsymbol{X}_{p}$, thanks to \cite{Am3, Am4}, we know that $\boldsymbol{U}\in\boldsymbol{X}_{p}\times[\boldsymbol{H}^{p}_{0}(\boldsymbol{\mathrm{curl}},\Omega)]'\hookrightarrow\boldsymbol{X}_{p}\times\boldsymbol{W}^{-1,p}(\Omega)$. Next let $\boldsymbol{u}\in\mathbf{D}((A'_{p})^{\frac{1}{2}})$ then $\boldsymbol{U}\in\boldsymbol{X}_{p}\times[\boldsymbol{W}^{1,p}(
 \Omega),
 \boldsymbol{W}^{-1,p}(\Omega)]_{1/2}=\boldsymbol{X}_{p}\times\boldsymbol{L}^{p}(\Omega)$. Thus using Lemma \ref{con1} we deduce that $\boldsymbol{u}\in\boldsymbol{V}^{p}_{\sigma,\tau}(\Omega)$. This amount to say that
\begin{equation}
\mathbf{D}((A'_{p})^{\frac{1}{2}})\hookrightarrow\boldsymbol{V}^{p}_{\sigma,\tau}(\Omega).
\end{equation}
It remains to prove the second inclusion. First we recall that $(\boldsymbol{X}_{p})'=\boldsymbol{X}_{p'}$ and  the adjoint operator $((A'_{p})^{\frac{1}{2}})^{*}$ is equal to $(A'_{p'})^{\frac{1}{2}}$. Observe that thanks to \cite[Theorem 1.15.2, part (e)]{Tri}, since $A'_{p}$ has a bounded inverse, then for all $1<p<\infty$, $(A'_{p})^{\frac{1}{2}}$ is an isomorphism from $\mathbf{D}((A'_{p})^{\frac{1}{2}})$ to $\boldsymbol{X}_{p}$. This means that for all $\boldsymbol{F}\in\boldsymbol{X}_{p'}$ there exists a unique $\boldsymbol{v}\in\mathbf{D}((A'_{p'})^{\frac{1}{2}})$ solution of 
\begin{equation}\label{ap'1/2}
(A'_{p'})^{\frac{1}{2}}\boldsymbol{v}\,=\,\boldsymbol{F}.
\end{equation}
As a result for all $\boldsymbol{u}\in\mathbf{D}(A'_{p})$ we have
\begin{eqnarray}
\Vert (A'_{p})^{\frac{1}{2}}\boldsymbol{u}\Vert_{\boldsymbol{X}_{p}}&=&\sup _{\boldsymbol{F}\in\boldsymbol{X}_{p'},\,\boldsymbol{F}\neq\boldsymbol{0}}\frac{\Big\vert\langle (A'_{p})^{\frac{1}{2}}\boldsymbol{u}\,,\,\boldsymbol{F}\rangle_{\boldsymbol{X}_{p}\times\boldsymbol{X}_{p'}}\Big\vert}{\Vert\boldsymbol{F}\Vert_{\boldsymbol{L}^{p'}(\Omega)}}\nonumber\\
&=&\sup _{\boldsymbol{F}\in\boldsymbol{X}_{p'},\,\boldsymbol{F}\neq\boldsymbol{0}}\frac{\Big\vert\langle (A'_{p})^{\frac{1}{2}}\boldsymbol{u}\,,\, (A'_{p'})^{\frac{1}{2}}\boldsymbol{v}\rangle_{\boldsymbol{X}_{p}\times\boldsymbol{X}_{p'}}\Big\vert}{\Vert\boldsymbol{F}\Vert_{\boldsymbol{L}^{p'}(\Omega)}},\nonumber
\end{eqnarray}
where $\boldsymbol{v}$ is the unique solution of (\ref{ap'1/2}) and $\boldsymbol{X}_{p'}$ is a closed subspace of $\boldsymbol{L}^{p'}(\Omega)$ equipped with the norm of $\boldsymbol{L}^{p'}(\Omega)$.

As a result, 
\begin{eqnarray}
\Vert (A'_{p})^{\frac{1}{2}}\boldsymbol{u}\Vert_{\boldsymbol{X}_{p}}&=&\sup _{\boldsymbol{v}\in\mathbf{D}(A_{p'}^{1/2}),\,\boldsymbol{v}\neq\boldsymbol{0}}\frac{\Big\vert\langle A'_{p}\boldsymbol{u}\,,\,\boldsymbol{v}\rangle_{\boldsymbol{X}_{p}\times\boldsymbol{X}_{p'}}\Big\vert}{\Vert (A'_{p'})^{\frac{1}{2}}\boldsymbol{v}\Vert_{\boldsymbol{L}^{p'}(\Omega)}}\nonumber\\
&=&\sup _{\boldsymbol{v}\in\mathbf{D}(A_{p'}^{1/2}),\,\boldsymbol{v}\neq\boldsymbol{0}}\frac{\Big\vert\int_{\Omega}\boldsymbol{\mathrm{curl}}\,\boldsymbol{u}\cdot\boldsymbol{\mathrm{curl}}\,\overline{\boldsymbol{v}}\,\textrm{d}\,x\Big\vert}{\Vert (A'_{p'})^{\frac{1}{2}}\boldsymbol{v}\Vert_{\boldsymbol{L}^{p'}(\Omega)}}\nonumber\\
&\leq&C(\Omega,p)\,\Vert\boldsymbol{u}\Vert_{\boldsymbol{W}^{1,p}(\Omega)}\label{dpaw1p}.
\end{eqnarray}
Now since $\mathbf{D}(A'_{p})$ is dense in $\boldsymbol{V}^{p}_{\sigma,\tau}(\Omega)$ one gets inequality (\ref{dpaw1p}) for all $\boldsymbol{u}\in\boldsymbol{V}^{p}_{\sigma,\tau}(\Omega)$ and then
\begin{equation*}
\boldsymbol{V}^{p}_{\sigma,\tau}(\Omega)\hookrightarrow\mathbf{D}((A'_{p})^{\frac{1}{2}})
\end{equation*}
and the result is prove.
\end{proof}

The following proposition shows an embedding of Sobolev type for the domains of fractional powers of the Stokes operator with flux boundary conditions. This embedding give us the $\boldsymbol{L}^{p}-\boldsymbol{L}^{q}$ estimates for the corresponding homogeneous problem.
\begin{prop}\label{Soboembdalpha}
For all $1<p<\infty$ and for all $0<\alpha\leq 1$ we define $\beta=\max(\alpha,1-\alpha)$ then
\begin{equation}\label{sbf}
\mathbf{D}((A'_{p})^{\alpha})\hookrightarrow\boldsymbol{L}^{q}(\Omega)
\end{equation}
for all $q$ such that:\\
\textbf{(i)} For $1<p<\frac{3}{2\beta}$, $q\in\left[p,\,\frac{3p}{3-2\beta p}\right].$

\noindent\textbf{(ii)} For $p=\frac{3}{2\beta}$, $ q\in\left[1,\,+\infty \right[.$

\noindent\textbf{(iii)} For $p>\frac{3}{2\beta}$, $ q=+\infty.$

Moreover for such $q$, the following estimate holds
\begin{equation}\label{q}
\forall\boldsymbol{u}\in\mathbf{D}((A'_{p})^{\alpha}),\qquad\Vert\boldsymbol{u}\Vert_{\boldsymbol{L}^{q}(\Omega)}\,\leq\,C(\Omega,p)\,\Vert (A'_{p})^{\alpha}\boldsymbol{u}\Vert_{\boldsymbol{L}^{p}(\Omega)}.
\end{equation}
\end{prop}
\begin{proof}
As described in the proof of Theorem \ref{DA1/2} we know that $\mathbf{D}((A'_{p})^{\alpha})\,=\,\left[\mathbf{D}(A'_{p})\,;\,\boldsymbol{X}_{p} \right]_{\alpha}$.
Moreover, we know that, $\left[\mathbf{D}(A'_{p})\,;\,\boldsymbol{X}_{p}\right]_{\alpha}\hookrightarrow\left[\boldsymbol{W}^{2,p}(\Omega)\,;\,\boldsymbol{L}^{p}(\Omega) \right] _{\alpha}\,=\,\boldsymbol{W}^{2(1-\alpha),p}(\Omega).$
It is clear that for $0<\alpha<1/2$, we have $1-\alpha>\alpha$ and 
\begin{equation*}
\mathbf{D}((A'_{p})^{\alpha})\hookrightarrow\boldsymbol{W}^{2\alpha,p}(\Omega).
\end{equation*}
Similarly, for $1/2\leq\alpha\leq1$, we have $\alpha\geq 1-\alpha$ and
\begin{equation*}
\mathbf{D}((A'_{p})^{\alpha})\hookrightarrow\mathbf{D}((A'_{p})^{1-\alpha})\hookrightarrow\boldsymbol{W}^{2\alpha,p}(\Omega).
\end{equation*}
Thus one has, for all $0<\alpha\leq1$
\begin{equation*}
\mathbf{D}((A'_{p})^{\alpha})\hookrightarrow\mathbf{D}((A'_{p})^{1-\alpha})\hookrightarrow\boldsymbol{W}^{2\beta,p}(\Omega).
\end{equation*}
Now using the result of \cite[Theorem 7.57]{Adams} we deduce the Sobolev embedding \eqref{sbf} with $p$ and $q$ satisfying (i), (ii) and (iii). Finally, estimate \eqref{q} is a direct consequence of the Sobolev embedding \eqref{sbf}, since $\mathbf{D}((A'_{p})^{\alpha})$ is equipped with the graph norm of the operator $ (A'_{p})^{\alpha}$.
\end{proof}

The following Corollary extends Proposition \ref{Soboembdalpha} to any real $\alpha$ such that $0<\alpha<3/2p$. This result is similar to the result of  Borchers and Miyakawa \cite{Bor2} who proved the same result for the Stokes operator with Dirichlet boundary conditions in exterior domains for $1<p<3$.  
\begin{coro}\label{SoboembAS}
for all $1<p<\infty$ and for all $\alpha\in\mathbb{R}$ such that $0<\alpha<3/2p$ the following Sobolev embedding holds 
\begin{equation}\label{sbas}
\mathbf{D}((A'_{p})^{\alpha})\hookrightarrow\boldsymbol{L}^{q}(\Omega),\qquad\frac{1}{q}=\frac{1}{p}-\frac{2\alpha}{3}.
\end{equation}
Moreover for all $\boldsymbol{u}\in\mathbf{D}((A'_{p})^{\alpha})$ the following estimate holds
\begin{equation}\label{estas}
\Vert\boldsymbol{u}\Vert_{\boldsymbol{L}^{q}(\Omega)}\,\leq\,C(\Omega,p)\,\Vert (A'_{p})^{\alpha}\boldsymbol{u}\Vert_{\boldsymbol{L}^{p}(\Omega)}.
\end{equation}
\end{coro}
\begin{proof}
First observe that for $0<\alpha<\min(1,3/2p)$ the Sobolev embedding \eqref{sbas} is a consequence of Proposition \ref{Soboembdalpha} part (i). Next, for any real $\alpha$ such that $0<\alpha<3/2p$ we write $\alpha=k+\theta$, where $k$ is a non negative integer and $0<\theta<1$. 

Next we set
\begin{equation}
\frac{1}{q_{0}}=\frac{1}{p}-\frac{2\theta}{3}\qquad\textrm{and}\qquad\frac{1}{q_{j}}=\frac{1}{q_{0}}-\frac{2j}{3},\,\,\,\,\, j=0,1,....,k.
\end{equation}
It is clear that $\frac{1}{q_{j}}=\frac{1}{q_{j-1}}-\frac{2}{3}$ and that $q_{k}=q$.
Moreover, by assumptions on $p$ and $\alpha$ we have for $j=0,1,....,k$, $\theta+j<3/2p$. As a consequence of Proposition \ref{Soboembdalpha} part (i) it follows that
\begin{equation*}
\mathbf{D}((A'_{p})^{\theta})\hookrightarrow\boldsymbol{L}^{q_{0}}(\Omega)
\end{equation*}
and for all $1\leq j\leq k$
\begin{equation*}
\mathbf{D}(A'_{q_{j-1}})\hookrightarrow\boldsymbol{L}^{q_{j}}(\Omega).
\end{equation*}
 It thus follows that for all $\boldsymbol{u}\in\mathbf{D}((A'_{p})^{\infty})=\cap_{m\in\mathbb{N}}\mathbf{D}((A'_{p})^{m})$
\begin{equation}\label{estdsigma}
\Vert\boldsymbol{u}\Vert_{\boldsymbol{L}^{q}(\Omega)}\leq \Vert A'_{q_{k-1}}\boldsymbol{u}\Vert_{\boldsymbol{L}^{q_{k-1}}(\Omega)}\leq ... \leq \Vert (A'_{q_{0}})^{k}\boldsymbol{u}\Vert_{\boldsymbol{L}^{q_{0}}(\Omega)}\leq \Vert (A'_{p})^{\alpha}\boldsymbol{u}\Vert_{\boldsymbol{L}^{p}(\Omega)}.
\end{equation} 
By density of $\mathbf{D}((A'_{p})^{\infty})$ in $\mathbf{D}((A'_{p})^{\alpha})$ on gets the Sobolev  embeddings \eqref{sbas} and estimate \eqref{estdsigma}. Finally, estimate \eqref{estas} is a direct consequence of \eqref{sbas}.
\end{proof}

\begin{rmk}
\rm{
Using that $(I+A_p)$ is bijective with bounded inverse, the same arguments give similar results for  $\mathbf{D}(A_{p}^{\alpha})$.  More precisely, using  Lemma \ref{domfracpower}, we have in that case, for all $0<\alpha<1$
\begin{equation*}
\mathbf{D}(A_{p}^{\alpha})\,=\,\mathbf{D}((I\,+\,A_{p})^{\alpha})\,=\,\left[\mathbf{D}(I\,+\,A_{p});\boldsymbol{L}^{p}_{\sigma,\tau}(\Omega) \right]_{\alpha}\,=\, \left[\mathbf{D}(A_{p});\boldsymbol{L}^{p}_{\sigma,\tau}(\Omega) \right]_{\alpha},
\end{equation*}
where the space $\mathbf{D}(A_{p}^{\alpha})$ is endowed with $\Vert(I\,+\,A_{p})^{\alpha}\boldsymbol{\cdot}\Vert_{\boldsymbol{L}^{p}(\Omega)}$. Arguing as in the proof of Theorem \ref{DA1/2}  and Corollary \ref{SoboembAS}, we obtain\\
\noindent \textbf{(i)} For all $1<p<\infty$,
$\mathbf{D}(A_{p}^{1/2})\,=\,\boldsymbol{X}^{p}_{\sigma,\tau}(\Omega)$ with equivalent norms, (see \eqref{vpt} for the definition of $\boldsymbol{X}^{p}_{\sigma,\tau}(\Omega)$). \\
\noindent \textbf{(ii)} For all $1<p<\infty$ and for all $\alpha\in\mathbb{R}$ such that $0<\alpha<3/2p$, we have
\begin{equation*}
\mathbf{D}(A^{\alpha}_{p})\hookrightarrow\boldsymbol{L}^{q}(\Omega),\qquad\frac{1}{q}=\frac{1}{p}-\frac{2\alpha}{3},
\end{equation*}
\begin{equation*}
\forall\,\boldsymbol{u}\in\mathbf{D}(A^{\alpha}_{p}),\quad\Vert\boldsymbol{u}\Vert_{\boldsymbol{L}^{q}(\Omega)}\,\leq\,C(\Omega,p)\,\Vert (I+A_{p})^{\alpha}\boldsymbol{u}\Vert_{\boldsymbol{L}^{p}(\Omega)}.
\end{equation*}
Moreover using  that 
\begin{equation*}
\forall\alpha\in\mathbb{C},\qquad\boldsymbol{\mathcal{D}}_{\sigma}(\Omega)\hookrightarrow\mathbf{D}(A^{\alpha}_{p})\hookrightarrow\boldsymbol{L}^{p}_{\sigma,\tau}(\Omega),
\end{equation*}
one has the density of $\boldsymbol{\mathcal{D}}_{\sigma}(\Omega)$ in $\mathbf{D}(A^{\alpha}_{p})$ for all $\alpha\in\mathbb{C}$. }
\end{rmk}

\section{The time dependent Stokes problem}
\label{time}
In this section we solve the time dependent Stokes Problem (\ref{lens}) with the boundary condition (\ref{nbc}) using the semi-group theory. As described above, due to the boundary conditions (\ref{nbc}) the Stokes operator coincides with the $-\Delta$ operator. 
\subsection{The homogeneous problem}
Consider the problem:
\begin{equation}\label{henp}
 \left\{
\begin{array}{cccc}
\frac{\partial\boldsymbol{u}}{\partial t} - \Delta \boldsymbol{u
}=\boldsymbol{0}, & \mathrm{div}\,\boldsymbol{u}=0&\textrm{in}& 
\Omega\times (0,T), \\
\boldsymbol{u}\cdot\boldsymbol{n}=0, & \boldsymbol{\mathrm{curl}}\,\boldsymbol{u}\times \boldsymbol{n} = \boldsymbol{0}&\textrm{on}&
\Gamma\times (0,T),\\
&\boldsymbol{u}(0)= \boldsymbol{u}_{0} & \textrm{in} &
\Omega.
\end{array}
\right.
\end{equation}
Usually in the Problem (\ref{henp}) where figures the constraint $\mathrm{div}\,\boldsymbol{u}=0$ in $\Omega$,  a gradient of pressure appears. However, thanks to our boundary conditions, the pressure is constant in our case. For this reason, the Problem (\ref{henp}) is equivalent to the homogeneous Stokes problem
\begin{equation}\label{hensp}
 \left\{
\begin{array}{cccc}
\frac{\partial\boldsymbol{u}}{\partial t} - \Delta \boldsymbol{u
}+\nabla\pi=\boldsymbol{0}, & \mathrm{div}\,\boldsymbol{u}=0&\textrm{in}& 
\Omega\times (0,T), \\
\boldsymbol{u}\cdot\boldsymbol{n}=0, & \boldsymbol{\mathrm{curl}}\,\boldsymbol{u}\times \boldsymbol{n} = \boldsymbol{0}&\textrm{on}&
\Gamma\times (0,T),\\
&\boldsymbol{u}(0)= \boldsymbol{u}_{0} & \textrm{in} &
\Omega.
\end{array}
\right.
\end{equation}
We start with the following result for initial data in $L^p _{ \sigma , \tau  }(\Omega )$ that follows easily from the classiacl semi group theory for the operator $A_p$ on the space $L^p _{ \sigma , \tau  }(\Omega )$.
\begin{theo}\label{exishenp}
Let $\boldsymbol{u}_{0}\in\boldsymbol{L}^{p}_{\sigma,\tau}(\Omega)$,
then Problem (\ref{henp}) has a unique solution $\boldsymbol{u}(t)$ satisfying
\begin{equation}\label{Reghenp1}
\boldsymbol{u}\in
C([0,\,+\infty[,\,\boldsymbol{L}^{p}_{\sigma,\tau}(\Omega))\cap
C(]0,\,+\infty[,\,\mathbf{D}(A_{p}))\cap
C^{1}(]0,\,+\infty[,\,\boldsymbol{L}^{p}_{\sigma,\tau}(\Omega)),
\end{equation}
\begin{equation}\label{Reghenp2}
\boldsymbol{u}\in C^{k}(]0,\,+\infty[,\,\mathbf{D}(A^{\ell}_{p})),\qquad
\forall\,k\in\mathbb{N},\,\,\forall\,\ell\in\mathbb{N^{\ast}}.
\end{equation}
Moreover we have the estimates
\begin{eqnarray}\label{esthenp1}
\|\boldsymbol{u}(t)\|_{\boldsymbol{L}^{p}(\Omega)}\leq\,C_{1}(\Omega,p)\,\|\boldsymbol{u}_{0}\|_{\boldsymbol{L}^{p}(\Omega)}\\
\label{esthenp2}
\Big\|\frac{\partial\boldsymbol{u}(t)}{\partial t}\Big\|_{\boldsymbol{L}^{p}(\Omega)}\leq\frac{C_{2}(\Omega,p)}{t}\,\|\boldsymbol{u}_{0}\|_{\boldsymbol{L}^{p}(\Omega)}.\\
\label{esthenp4}
\Vert\boldsymbol{\mathrm{curl}}\,\boldsymbol{u}\Vert_{\boldsymbol{L}^{p}(\Omega)}\,\leq\,\frac{C_{3}(\Omega,p)}{\sqrt{t}}\,\Vert\boldsymbol{u}_{0}\Vert_{\boldsymbol{L}^{p}(\Omega)}
\end{eqnarray}
and
\begin{equation}\label{esthenp3}
\Vert\boldsymbol{u}(t)\Vert_{\boldsymbol{W}^{2,p}(\Omega)}\,\leq\,C_{4}(\Omega,p)\,(1+\frac{1}{t})\,\Vert\boldsymbol{u}_{0}\Vert_{\boldsymbol{L}^{p}(\Omega)}.
\end{equation}
\end{theo}
\begin{proof}
Since the operator $-A_{p}$ generates a bounded analytic semi-group $(T(t))_{t\geq 0}$ on $\boldsymbol{L}^{p}_{\sigma,\tau}(\Omega)$, the Problem (\ref{henp}) has a unique solution $\boldsymbol{u}(t)=T(t)\,\boldsymbol{u}_{0}$ . Thanks to \cite[Chapter 2, Proposition 4.3]{En} we know that $\Vert T(t)\Vert_{\mathcal{L}(\boldsymbol{L}^{p}_{\sigma,\tau}(\Omega))}\leq C_{1}(\Omega,p)$, where $C_{1}(\Omega,p)=M_{1}\,\kappa_{1}(\Omega,p)$ for some constant $M_{1}>0$. We recall that $\kappa_{1}(\Omega,p)$ is the constant in (\ref{estimlpf}). As a result one has estimate (\ref{esthenp1}). We also know thanks to \cite[Chapter 2, Theorem 4.6]{En} that this solution belongs to $\mathbf{D}(A_{p})$ thus one has (\ref{Reghenp1}). Now using the fact that $T(t)\,\boldsymbol{u}_{0}\in\mathbf{D}(A^{\infty}_{p})$ and the same argument of \cite[Chapitre 7, Theorem 7.5, Theorem 7.7]{Br} one gets the regularity (\ref{Reghenp2}). 
We recall that $\mathbf{D}(A^{\infty}_{p})=\cap_{n\in\mathbb{N}}\mathbf{D}(A^{n}_{p})$.

 Moreover, thanks to \cite[Chapter 2, Theorem 4.6, page 101]{En} we know that 
\begin{equation*}
\Vert A_{p}T(t)\Vert_{\boldsymbol{L}^{p}(\Omega)}\leq\,\frac{C_{2}(\Omega,p)}{t},
\end{equation*}
where $C_{2}(\Omega,p)=M_{2}\,\kappa_{1}(\Omega,p)$
for some constant $M_{2}>0$, which gives us estimate (\ref{esthenp2}). 

\noindent Next, to prove estimate (\ref{esthenp4}) we proceed in the same way as in the proof of the estimate (\ref{curlestlp}) (see \cite[Theorem 4.11]{Albaba} for the proof).

Since the norm of $\boldsymbol{W}^{2,p}(\Omega)$ is equivalent to the graph norm of the stokes operator $A_{p}$ one has estimate (\ref{esthenp3}). 
\end{proof}
Estimates (\ref{esthenp1}) and  (\ref{esthenp4}) allow to deduce the following Corollary:
\begin{coro}[Weak Solutions for the Stokes Problem]\label{corolrlpw1p}
Let $\boldsymbol{u}_{0}\in\boldsymbol{L}^{p}_{\sigma,\tau}(\Omega)$ and $\boldsymbol{u}$ be the unique solution of Problem (\ref{henp}) given by Theorem \ref{exishenp}. Then  $\boldsymbol{u}$ satisfies
\begin{equation}\label{lrw1p}
\forall\,1\leq q<2,\qquad\boldsymbol{u}\in L^{q}(0,T;\,\boldsymbol{W}^{1,p}(\Omega))\,\,\,\,\mathrm{and}\qquad\frac{\partial\boldsymbol{u}}{\partial t}\in L^{q}(0,T;\,[\boldsymbol{H}^{p'}_{0}(\mathrm{div},\Omega)]'),
\end{equation}
for all $T>0$.
\end{coro}
\begin{proof}
Let $\boldsymbol{u}(t)$ be the unique solution of Problem (\ref{henp}). By hypothesis we know that $\boldsymbol{u}$ satisfies the estimates (\ref{esthenp1})-(\ref{esthenp3}). Now thanks to Lemma \ref{con1} we know that $$\Vert\boldsymbol{u}(t)\Vert_{\boldsymbol{W}^{1,p}(\Omega)}\simeq\Vert\boldsymbol{u}(t)\Vert_{\boldsymbol{L}^{p}(\Omega)}+\Vert\boldsymbol{\mathrm{curl}}\,\boldsymbol{u}(t)\Vert_{\boldsymbol{L}^{p}(\Omega)}.$$ Thus one deduces directly that $\boldsymbol{u}\in L^{q}(0,T;\,\boldsymbol{W}^{1,p}(\Omega))$ for all $1\leq q<2$ and for all $0<T<\infty
$.\\ 
Next, let us prove that $\frac{\partial\boldsymbol{u}}{\partial t}\in L^{q}(0,T;\,[\boldsymbol{H}^{p'}_{0}(\mathrm{div},\Omega)]'),$ 
 set $$\widetilde{\boldsymbol{u}}(t)=\boldsymbol{u}(t)-\sum_{j=1}^{J}\langle\boldsymbol{u}(t)\cdot\boldsymbol{n}\,,\,1\rangle_{\Sigma_{j}}\widetilde{\boldsymbol{\mathrm{grad}}}\,q_{j}^{\tau}.$$ It is clear that $\boldsymbol{u}(t)\,=\,\widetilde{\boldsymbol{u}}(t)\,+\,\sum_{j=1}^{J}\langle\boldsymbol{u}(t)\cdot\boldsymbol{n}\,,\,1\rangle_{\Sigma_{j}}\widetilde{\boldsymbol{\mathrm{grad}}}\,q_{j}^{\tau}.$  Moreover thanks to \cite[Proposition 4.3]{Am3} we know that 
\begin{equation*}
\Vert\Delta\boldsymbol{u}\Vert_{[\boldsymbol{H}^{p'}_{0}(\mathrm{div},\Omega)]'}\,=\,\Vert\Delta\widetilde{\boldsymbol{u}}\Vert_{[\boldsymbol{H}^{p'}_{0}(\mathrm{div},\Omega)]'}\,\simeq\,\Vert\widetilde{\boldsymbol{u}}\Vert_{\boldsymbol{W}^{1,p}(\Omega)}\,\leq\, \Vert\boldsymbol{u}\Vert_{\boldsymbol{W}^{1,p}(\Omega)}.
\end{equation*}
The last inequality comes from the fact (see \cite[Lemma 3.2]{Am4}) 
\begin{equation*}
\vert\langle\boldsymbol{u}\cdot\boldsymbol{n}\,,\,1\rangle_{\Sigma_{j}}\vert\,\leq\,C(\Omega,p)\,\Vert\boldsymbol{u}\Vert_{\boldsymbol{L}^{p}(\Omega)}.
\end{equation*}
Thus 
 $\frac{\partial\boldsymbol{u}}{\partial t}=\Delta\boldsymbol{u}\in L^{q}(0,T;\,[\boldsymbol{H}^{p'}_{0}(\mathrm{div},\Omega)]')$ and the result is proved.
\end{proof}

We observe the following remark:
\begin{rmk}\label{rmkhenp}
\rm{
\textbf{(i)} In the Hilbertian case ($\boldsymbol{u}_{0}\in\boldsymbol{L}^{2}_{\sigma,\tau}(\Omega)$), the properties (\ref{esthenp1})-(\ref{esthenp3}) are immediate. We will prove estimate (\ref{esthenp4}). Observe that, thanks to Propositon \ref{eva} and Remark \ref{remarkeva},  on $\boldsymbol{L}^{2}_{\sigma,\tau}(\Omega)$  we can express $\boldsymbol{u}(t)$
explicitly in the form
\begin{equation}
\boldsymbol{u}(t)\,=\,\sum_{j=1}^{J}\alpha_{j}\,\widetilde{\boldsymbol{\mathrm{grad}}}\,q^{\tau}_{j}\,+\,\sum_{k=1}^{+\infty}\beta_{k}\,e^{-\lambda_{k}\,t}\,\boldsymbol{z}_{k},
\end{equation}
where
\begin{equation*}
\alpha_{j}\,=\,\int_{\Omega}\boldsymbol{u}_{0}\cdot\widetilde{\boldsymbol{\mathrm{grad}}}\,\overline{q^{\tau}_{j}}\,\textrm{d}\,x\qquad
\textrm{and}\qquad
\beta_{k}\,=\,\int_{\Omega}\boldsymbol{u}_{0}\cdot\overline{\boldsymbol{z}_{k}}\,\textrm{d}\,x.
\end{equation*}
As a result, using the fact that $A_{2}\boldsymbol{z}_{k}=\lambda_{k}\,\boldsymbol{z}_{k}$ and the fact that
\begin{equation*}
\int_{\Omega}\vert\boldsymbol{\mathrm{curl}}\,\boldsymbol{z}_{k}\vert^{2}\mathrm{d}\,x=\lambda_{k}\,\Vert\boldsymbol{z}_{k}\Vert^{2}_{\boldsymbol{L}^{2}(\Omega)}=\lambda_{k}
\end{equation*}
one has
\begin{equation*}
\Vert\boldsymbol{\mathrm{curl}}\,\boldsymbol{u}(t)\Vert^{2}_{\boldsymbol{L}^{2}(\Omega)}\,=\,\sum_{k=1}^{+\infty}\beta_{k}^{2}\,e^{-2\lambda_{k}\,t}\lambda_{k}.
\end{equation*}
Finally, since 
\begin{equation*}
\Vert\boldsymbol{u}_{0}\Vert^{2}_{\boldsymbol{L}^{2}(\Omega)}=\sum_{j=1}^{J}\alpha^{2}_{j}\,+\,\sum^{+\infty}_{k=1}\beta^{2}_{k}
\end{equation*} estimate (\ref{esthenp4}) follows directly.  Similarly 
 one gets directly
 estimates (\ref{esthenp1})-(\ref{esthenp3}). We recall that $(\boldsymbol{z}_{k})_{k}$ are eigenvectors for the Stokes operator 
associated to the eigenvalues $(\lambda_{k})_{k}$ and they form with $(\widetilde{\boldsymbol{\mathrm{grad}}}\,q^{\tau}_{j})_{1\leq j\leq J}$ an orthonormal basis for $\boldsymbol{L}^{2}_{\sigma,\tau}(\Omega)$ .

\noindent\textbf{(ii)} For $p=2,$ the solution $\boldsymbol{u}$ satisfies (see \cite[Theorem 6.4]{Am5}) \begin{equation}\label{lrw12}
\boldsymbol{u}\in L^{2}(0,T;\,\boldsymbol{H}^{1}(\Omega))\,\,\,\,\textrm{and}\qquad\frac{\partial\boldsymbol{u}}{\partial t}\in L^{2}(0,T;\,[\boldsymbol{H}^{2}_{0}(\mathrm{div},\Omega)]')
\end{equation} 
and
\begin{equation*}
\frac{1}{2}\,\frac{\mathrm{d}}{\mathrm{d\,t}}\,\Vert\boldsymbol{u}(t)\Vert^{2}_{\boldsymbol{L}^{2}(\Omega)}\,+\,\int_{\Omega}\vert\boldsymbol{\mathrm{curl}}\,\boldsymbol{u}(t)\vert^{2}\,\mathrm{d}\,x\,=\,0.
\end{equation*}
In other words for $p=2$, Corollary \ref{corolrlpw1p} still holds true for $q=2$ included.

}
\end{rmk}

We consider now the case where the initial data $\boldsymbol{u}_{0}\in\boldsymbol{X}_{p}$, (see \ref{Xp} for the definition of $\boldsymbol{X}_{p}$).
\begin{theo}\label{exishenpA'}
 Suppose that $\boldsymbol{u}_{0}\in\boldsymbol{X}_{p}$ and let $\boldsymbol{u}$ be the unique solution to Problem
 (\ref{henp}). Then $\boldsymbol{u}$ satisfies the following:
\begin{equation}\label{Reghenp1A'}
\boldsymbol{u}\in
C([0,\,+\infty[,\,\boldsymbol{X}_{p})\cap
C(]0,\,+\infty[,\,\mathbf{D}(A'_{p}))\cap
C^{1}(]0,\,+\infty[,\,\boldsymbol{X}_{p}),
\end{equation}
\begin{equation}\label{Reghenp2A'}
\boldsymbol{u}\in C^{k}(]0,\,+\infty[,\,\mathbf{D}((A'_{p})^{\ell})),\qquad
\forall\,k,\,\ell\in\mathbb{N}.
\end{equation}
Moreover, for all $q\in [p, \infty)$, for all  integers $m\ge 0$,  $n\ge 0$ and for all $\mu \in (0, \lambda _1)$ there exists a constant $M>0$ such that  the solution $\boldsymbol{u}$ satisfies, for all $t>0$:
\begin{equation}\label{estlplqutxp}
 \Vert\boldsymbol{u}(t)\Vert_{\boldsymbol{L}^{q}(\Omega)}\,\leq\,M\,e^{-\mu \,t}\,t^{-3/2(1/p-1/q)}\Vert \boldsymbol{u}_{0}\Vert_{\boldsymbol{L}^{p}(\Omega)},
 \end{equation}
 \begin{equation}\label{estlplqcurlutxp}
  \Vert\boldsymbol{\mathrm{curl}}\,\boldsymbol{u}(t)\Vert_{\boldsymbol{L}^{q}(\Omega)}\,\leq\,M\,e^{-\mu \,t}\,t^{-3/2(1/p-1/q)-1/2}\Vert \boldsymbol{u}_{0}\Vert_{\boldsymbol{L}^{p}(\Omega)}
 \end{equation}
 and
 \begin{equation}\label{estlplqlaputxp}
 \Big\Vert\frac{\partial^{m}}{\partial t^{m}}\Delta^{n}\boldsymbol{u}(t)\Big\Vert_{\boldsymbol{L}^{q}(\Omega)}\,\leq\,M\,e^{-\mu \,t}\,t^{-(m+n)-3/2(1/p-1/q)}\Vert \boldsymbol{u}_{0}\Vert_{\boldsymbol{L}^{p}(\Omega)},
 \end{equation}  
where $\lambda_{1}$ is the first non zero eigenvalue of the Stokes operator defined above.
\end{theo}
\begin{proof}
Applying the semi-group theory to the operator $A'_{p}$, one gets the existence and uniqueness of a solution to the homogeneous Stokes Problem \eqref{henp} given by $\boldsymbol{v}(t)=T(t)_{\vert\boldsymbol{X}_{p}}\boldsymbol{u}_{0}$ and satisfying (\ref{Reghenp1A'})-(\ref{Reghenp2A'}). We recall that $(T(t)_{\vert\boldsymbol{X}_{p}})_{t\geq 0}$ is the semi-group generated by the Stokes operator with flux boundary conditions on $\boldsymbol{X}_{p}$. Moreover, since $\boldsymbol{X}_{p}\subset\boldsymbol{L}^{p}_{\sigma,\tau}(\Omega)$,  by the uniqueness of solution $\boldsymbol{u}$ in Theorem \ref{exishenp}, we deduce that $\boldsymbol{v}(t)=\boldsymbol{u}(t)=T(t)\boldsymbol{u}_{0}$, the unique solution to Problem \eqref{henp}.  Let us prove estimates \eqref{estlplqutxp}--\eqref{estlplqlaputxp}.  To this end observe first that, by Theorem \ref{exislpA'} and to \cite{Am3}, we have:
\begin{equation*}
S(-A'_{p})=\sup \{\mathrm{Re}\,\lambda\in\sigma(-A'_{p})\}=-\lambda_{1}<0.
\end{equation*}
As a result, thanks to \cite[Chapitre 4, Theorem 4.3, page 118]{Pa}, there is a constant $M>0$ such that for all $0<\mu<\lambda_{1},\,$ $\Vert T(t)_{\vert\boldsymbol{X}_{p}}\Vert_{\mathcal{L}(\boldsymbol{X}_{p})}\leq\,M\,\kappa_{1}(\Omega,p)\,e^{-\mu\,t}$. 

The estimates  \eqref{estlplqutxp}--\eqref{estlplqlaputxp}  follow for the cases where  $q=p$ and  $m=1, n=0$ or $m=0, n=0, 1, 2$ using the classical semi-group theory.

Suppose that $p\neq q$, the proof is similar to the proof of \cite[Corollary 4.6]{Bor2}. Let $s\in\mathbb{R}$ such that $\frac{3}{2}(\frac{1}{p}-\frac{1}{q})<s<\frac{3}{2p}$ and set $\frac{1}{p_{0}}=\frac{1}{p}-\frac{2s}{3}$. It is clear that $p<q<p_{0}$. Let $\boldsymbol{u}(t)$ be the unique solution of Problem (\ref{henp}). Since for all $t>0$, $\boldsymbol{u}(t)\in\mathbf{D}((A'_{p})^{\infty})$, then thanks to Corollary \ref{SoboembAS}, $\boldsymbol{u}(t)\in\mathbf{D}((A'_{p})^{s})\hookrightarrow\boldsymbol{L}^{p_{0}}(\Omega)$. Now set $\alpha=\frac{1/p-1/q}{1/p-1/p_{0}}\in\left] 0,1\right[ $, we can easily verify that $\frac{1}{q}=\frac{\alpha}{p_{0}}+\frac{1-\alpha}{p}$. Thus $\boldsymbol{u}(t)\in\boldsymbol{L}^{q}(\Omega)$ and
 \begin{eqnarray}
 \Vert\boldsymbol{u}(t)\Vert_{\boldsymbol{L}^{q}(\Omega)}&\leq&C \Vert\boldsymbol{u}(t)\Vert^{\alpha}_{\boldsymbol{L}^{p_{0}}(\Omega)}
 \Vert\boldsymbol{u}(t)\Vert^{1-\alpha}_{\boldsymbol{L}^{p}(\Omega)}\nonumber\\
 &\leq&C \Vert (A'_{p})^{s}T(t)\boldsymbol{u}_{0}\Vert^{\alpha}_{\boldsymbol{L}^{p}(\Omega)} \Vert T(t)\boldsymbol{u}_{0}\Vert^{1-\alpha}_{\boldsymbol{L}^{p}(\Omega)}\nonumber\\
 &\leq& C\,e^{-\mu t}t^{-\alpha s} \Vert\boldsymbol{u}_{0}\Vert_{\boldsymbol{L}^{p}(\Omega)}.\label{expodec}\\
 &=&C\,e^{-\mu t}t^{-3/2(1/p-1/q)} \Vert\boldsymbol{u}_{0}\Vert_{\boldsymbol{L}^{p}(\Omega)}.
\end{eqnarray}
Estimate \eqref{expodec} follows from the fact that, (cf. \cite[Chapter 2, Theorem 6.13, page 76]{Pa}),
\begin{equation}\label{semigroupproperty1}
\Vert (A'_{p})^\alpha \,T(t)_{\vert\boldsymbol{X}_{p}}\Vert_{\mathcal{L}(\boldsymbol{X}_{p})}\,\leq\,M\,\kappa_{1}(\Omega,p)\,\frac{e^{-\mu\,t}}{t^\alpha }.
\end{equation}
 
  Next, let $\boldsymbol{u}_{0}\in\boldsymbol{X}_{p}\cap\boldsymbol{X}_{q}$ then $\boldsymbol{\mathrm{curl}}\,\boldsymbol{u}(t)\in\boldsymbol{L}^{q}(\Omega)$ and
  \begin{eqnarray*}
   \Vert\boldsymbol{\mathrm{curl}}\,\boldsymbol{u}(t)\Vert_{\boldsymbol{L}^{q}(\Omega)}\,\leq\,C\,\Vert (A'_{q})^{\frac{1}{2}}\boldsymbol{u}(t)\Vert_{\boldsymbol{L}^{q}(\Omega)}&=&\Vert(A'_{q})^{\frac{1}{2}}\,T(t/2)T(t/2)\boldsymbol{u}_{0}\Vert_{\boldsymbol{L}^{q}(\Omega)}\\
   &\leq&C\,e^{-\mu t}t^{-1/2}\,\Vert T(t/2)\boldsymbol{u}_{0}\Vert_{\boldsymbol{L}^{q}(\Omega)}\\
   &\leq&C\,e^{-\mu t}t^{-1/2}\,t^{-3/2(1/p-1/q)}\Vert \boldsymbol{u}_{0}\Vert_{\boldsymbol{L}^{p}(\Omega)}.
  \end{eqnarray*}
  Now let $\boldsymbol{u}_{0}\in\boldsymbol{X}_{p}$, using the density of $\boldsymbol{X}_{p}\cap\boldsymbol{X}_{q}$ in $\boldsymbol{X}_{p}$ we know that there exists a sequence $(\boldsymbol{u}_{0_{m}})_{m\geq0}$ in  $\boldsymbol{X}_{p}\cap\boldsymbol{X}_{q}$ that converges to $\boldsymbol{u}_{0}$ in $\boldsymbol{X}_{p}$. For all $m\in\mathbb{N}$ we set $\boldsymbol{u}_{m}(t)=T(t)\boldsymbol{u}_{0_{m}}$, as a result the sequences $(\boldsymbol{u}_{m}(t))_{m\geq0}$ and $(\boldsymbol{\mathrm{curl}}\,\boldsymbol{u}_{m}(t))_{m\geq0}$ converges to $\boldsymbol{u}(t)$ and $\boldsymbol{\mathrm{curl}}\,\boldsymbol{u}(t)$ respectively in $\boldsymbol{L}^{p}(\Omega)$, where $u(t)=T(t)\boldsymbol{u}_{0}$. On the other hand, for all $m,n\in\mathbb{N}$ one has
 \begin{equation*}
  \Vert\boldsymbol{\mathrm{curl}}(\boldsymbol{u}_{n}(t)-\boldsymbol{u}_{m}(t))\Vert_{\boldsymbol{L}^{q}(\Omega)}\leq C\,e^{-\mu t}t^{-1/2}\,t^{-3/2(1/p-1/q)}\Vert \boldsymbol{u}_{0_{n}}-\boldsymbol{u}_{0_{m}}\Vert_{\boldsymbol{L}^{p}(\Omega)}.
 \end{equation*}
 Thus $(\boldsymbol{\mathrm{curl}}\,\boldsymbol{u}_{m}(t))_{m\geq0}$ is a Cauchy sequence in $\boldsymbol{L}^{q}(\Omega)$ and converges to $\boldsymbol{\mathrm{curl}}\,\boldsymbol{u}(t)$
 in $\boldsymbol{L}^{q}(\Omega)$. This means that $\boldsymbol{\mathrm{curl}}\,\boldsymbol{u}(t)\in\boldsymbol{L}^{q}(\Omega)$ and by passing to the limit as $m\rightarrow\infty$ one gets estimate \eqref{estlplqcurlut}.
   
 Finally, using \eqref{Reghenp1A'}-\eqref{Reghenp2A'}, we have for all $m,n\in\mathbb{N},\,$ $\frac{\partial^{m}}{\partial t^{m}}\Delta^{n}\boldsymbol{u}\in C^{\infty}((0,\infty),\,\mathbf{D}(A'_{p}))$. Thus  $\frac{\partial^{m}}{\partial t^{m}}\Delta^{n}\boldsymbol{u}(t)$ belongs to $\boldsymbol{L}^{q}(\Omega)$ and
  \begin{equation*}
  \Big\Vert\frac{\partial^{m}}{\partial t^{m}}\Delta^{n}\boldsymbol{u}(t)\Big\Vert_{\boldsymbol{L}^{q}(\Omega)}=\Vert (A'_{p})^{(m+n)}\,T(t)\boldsymbol{u}_{0}\Vert_{\boldsymbol{L}^{q}(\Omega)}\leq C\,e^{-\mu t}t^{-(m+n)-3/2(1/p-1/q)}\Vert \boldsymbol{u}_{0}\Vert_{\boldsymbol{L}^{p}(\Omega)}.
 \end{equation*} 
\end{proof}

 Using now the results of Theorem \ref{exishenpA'} we will  extend estimates (\ref{esthenp1})-(\ref{esthenp3}) and obtain the following  $L^{p}-L^{q}$ estimates
 \begin{theo}\label{theoestlplq}
 Let $1<p\leq q<\infty$ and $\boldsymbol{u}_{0}\in\boldsymbol{L}^{p}_{\sigma,\tau}(\Omega)$. 
 The unique solution $\boldsymbol{u}$ to  Problem (\ref{henp}) given by Theorem \ref{exishenp} belongs to $\boldsymbol{L}^{q}(\Omega)$ and  satisfies, for all $t>0$:
 \begin{equation}\label{estlplqut}
 \Vert\boldsymbol{u}(t)-\boldsymbol{w}_{0}\Vert_{\boldsymbol{L}^{q}(\Omega)}\,\leq\,C\,e^{-\mu \,t}\,t^{-3/2(1/p-1/q)}\Vert \widetilde{\boldsymbol{u}}_{0}\Vert_{\boldsymbol{L}^{p}(\Omega)},
 \end{equation}
 with $\boldsymbol{w}_{0}$ and $\widetilde{\boldsymbol{u}}_{0}$ are given by \eqref{W0introduction} and \eqref{widetildeu0} respectively. Moreover, the following estimates hold 
 \begin{equation}\label{estlplqcurlut}
  \Vert\boldsymbol{\mathrm{curl}}\,\boldsymbol{u}(t)\Vert_{\boldsymbol{L}^{q}(\Omega)}\,\leq\,C\,e^{-\mu t}t^{-3/2(1/p-1/q)-1/2}\Vert \widetilde{\boldsymbol{u}}_{0}\Vert_{\boldsymbol{L}^{p}(\Omega)},
 \end{equation}
 \begin{equation}\label{estlplqlaput}
\forall\,m,n\in\mathbb{N},\,\,m+n>0,\quad\Big\Vert\frac{\partial^{m}}{\partial t^{m}}\Delta^{n}\boldsymbol{u}(t)\Big\Vert_{\boldsymbol{L}^{q}(\Omega)}\,\leq\,C\,e^{-\mu t}t^{-(m+n)-3/2(1/p-1/q)}\Vert \widetilde{\boldsymbol{u}}_{0}\Vert_{\boldsymbol{L}^{p}(\Omega)}.
 \end{equation}
 \end{theo}
 \begin{proof} By definition, $\boldsymbol{u}_{0}\,=\,\boldsymbol{w}_{0}\,+\,\widetilde{\boldsymbol{u}}_{0},$ with $\boldsymbol{w}_{0}\in\boldsymbol{K}_{\tau}(\Omega)$ and $\widetilde{\boldsymbol{u}}_{0}\in\boldsymbol{X}_{p}$. It follows that the unique solution to Problem \eqref{henp} given by Theorem \ref{exishenp} can be written in the form 
\begin{equation}\label{solutionexplicitesum}
\boldsymbol{u}(t)= \boldsymbol{w}_{0}\,+\,T(t)\widetilde{\boldsymbol{u}}_{0},
\end{equation}
where $T(t)\widetilde{\boldsymbol{u}}_{0}$ satisfies \eqref{Reghenp1A'}-\eqref{estlplqlaputxp}.
 
 The case $p=q\,$ follows directly from Theorem \ref{exishenp}, so let us suppose that $p\neq q$. The estimate \eqref{estlplqut} follows from \eqref{solutionexplicitesum} and \eqref{estlplqutxp}.
 
 Estimate \eqref{estlplqcurlut} follows from \eqref{estlplqcurlutxp} using that
 $\boldsymbol{\mathrm{curl}}\,\boldsymbol{u}(t)=\boldsymbol{\mathrm{curl}}\,\boldsymbol{w}_{0}+\boldsymbol{\mathrm{curl}}\,(T(t)\widetilde{\boldsymbol{u}}_{0})=\boldsymbol{\mathrm{curl}}\,(T(t)\widetilde{\boldsymbol{u}}_{0})$. 
 
Finally, for all $m,n\in\mathbb{N},\,$ such that $m+n>0$ we have $$\frac{\partial^{m}}{\partial t^{m}}\Delta^{n}\boldsymbol{u}(t)\,=\,A^{m+n}_{p}\boldsymbol{u}(t)\,=\,A^{m+n}_{p}\boldsymbol{w}_{0}\,+\,A^{m+n}_{p}(T(t)\widetilde{\boldsymbol{u}}_{0})\,=\,(A'_{p})^{m+n}\,(T(t)\widetilde{\boldsymbol{u}}_{0}).$$
As a result, using Theorem \ref{exishenpA'} one has estimate \eqref{estlplqlaput}.
 \end{proof}

\begin{proof}[Proof of Theorem \ref{theo3}]
 Theorem \ref{theo3} immediately follows from  Theorem \ref{exishenp} and Theorem \ref{theoestlplq}.
 \end{proof}

Our next result shows that  under stronger regularity assumptions on the initial data $\boldsymbol{u}_{0}$, the solution given by Theorem \ref{exishenp}  is actually a strong solution. 
\begin{prop}[Strong Solutions for the homogeneous Stokes Problem]\label{StrongSolution}
Assume that  $\boldsymbol{u}_{0}\in\boldsymbol{X}^{p}_{\sigma,\tau}(\Omega)$ (given by \eqref{vpt}). 
The unique solution $\boldsymbol{u}(t)$ of Problem (\ref{henp}) given by Theorem \ref{exishenp}  satisfies:
\begin{equation}\label{lrlpw2p}
\forall\,q\in [1, 2),\, \forall\, T>0,\,\;\boldsymbol{u}\in L^{q}(0,T;\,\boldsymbol{W}^{2,p}(\Omega))\,\mathrm{and}\,\,\frac{\partial\boldsymbol{u}}{\partial t}\in L^{q}(0,T;\,\boldsymbol{L}^{p}_{\sigma,\tau}(\Omega)),
\end{equation}
\begin{equation}\label{estlplqutvptau}
\forall\, q\ge p,\, \forall\, t>0,\, \Vert\boldsymbol{u}(t)-\boldsymbol{w}_{0}\Vert_{\boldsymbol{L}^{q}(\Omega)}\leq\Vert C e^{-\mu t}t^{-3/2(1/p-1/q)-1/2}\Vert \widetilde{\boldsymbol{u}}_{0}\Vert_{\boldsymbol{W}^{1,p}(\Omega)},
\end{equation}
 with $\boldsymbol{w}_{0}$ and $\widetilde{\boldsymbol{u}}_{0}$ are given by \eqref{W0introduction} and \eqref{widetildeu0} respectively. Moreover we have,
\begin{equation}\label{dutw1p}
\forall\, t>0,\quad \Big\Vert\frac{\partial\boldsymbol{u}}{\partial t}(t)\Big\Vert_{\boldsymbol{L}^{p}(\Omega)}\,\leq\,\frac{C(\Omega,p)}{\sqrt{t}}\,\Vert\boldsymbol{u}_{0}\Vert_{\boldsymbol{W}^{1,p}(\Omega)}
\end{equation}
and
\begin{equation}\label{w2pw1p}
\forall t>0,\,\, \Vert\boldsymbol{u}(t)\Vert_{\boldsymbol{W}^{2,p}(\Omega)}\,\leq\,C(\Omega,p)\Big(1\,+\,\frac{1}{\sqrt{t}}\Big)\,\Vert\boldsymbol{u}_{0}\Vert_{\boldsymbol{W}^{1,p}(\Omega)}.
\end{equation}
\end{prop}
\begin{proof}
Let $\boldsymbol{u}_{0}\in\boldsymbol{X}^{p}_{\sigma,\tau}(\Omega)$ and let $\boldsymbol{u}(t)$ be the unique solution of Problem (\ref{henp}). Set $\boldsymbol{z}\,=\,\boldsymbol{\mathrm{curl}}\,\boldsymbol{u}(t)$.
It is clear that $\boldsymbol{z}(t)$ is a solution of the problem
\begin{equation*}
\left\{
\begin{array}{cccc}
\frac{\partial\boldsymbol{z}}{\partial t} - \Delta \boldsymbol{z}= \mathbf{curl}\,\boldsymbol{f},& \mathrm{div}\,\boldsymbol{z}=0& \textrm{in}&
\Omega\times (0,T), \\
&\boldsymbol{z}\times \boldsymbol{n} = \boldsymbol{0},& \textrm{on} & \Gamma\times (0,T), \\
&\boldsymbol{z}(0)= \mathbf{curl}\,\boldsymbol{u}_{0} &\textrm{in} &\Omega.
\end{array}
\right.
\end{equation*}
Thus, thanks to Remark \ref{normboundcond} and proceeding in a similar way as in the proof of Theorem \ref{exishenp} we have 
\begin{equation*}
\Vert\boldsymbol{\mathrm{curl}}\,\boldsymbol{z}\Vert_{\boldsymbol{L}^{p}(\Omega)}\,\leq\,\frac{C(\Omega,p)}{\sqrt{t}}\,\Vert\boldsymbol{z}_{0}\Vert_{\boldsymbol{L}^{p}(\Omega)}.
\end{equation*}
This means that 
\begin{equation*}
\Big\Vert\frac{\partial\boldsymbol{u}}{\partial t}\Big\Vert_{\boldsymbol{L}^{p}(\Omega)}\,=\,\Vert\Delta\boldsymbol{u}\Vert_{\boldsymbol{L^{p}}(\Omega)}\,\leq\,\frac{C(\Omega,p)}{\sqrt{t}}\,\Vert\boldsymbol{u}_{0}\Vert_{\boldsymbol{W}^{1,p}(\Omega)}.
\end{equation*}
Finally using the fact that (see Remark \ref{rmkequivnorm}) $\Vert\boldsymbol{u}(t)\Vert_{\boldsymbol{W}^{2,p}(\Omega)}\simeq\Vert\boldsymbol{u}(t)\Vert_{\boldsymbol{L}^{p}(\Omega)}\,+\,\Vert\Delta\boldsymbol{u}(t)\Vert_{\boldsymbol{L}^{p}(\Omega)}$, (\ref{lrlpw2p}), \eqref{dutw1p} and (\ref{w2pw1p}) follow directly.

In order to prove \eqref{estlplqutvptau} we first notice that, since $\boldsymbol{u}_0\in\boldsymbol{X}^{p}_{\sigma,\tau}(\Omega)\subset\boldsymbol{L}_{\sigma, \tau}^p(\Omega)$ the solution $\boldsymbol{u}$ satisfies the $L^{p}-L^{q}$ estimates \eqref{estlplqut}--\eqref{estlplqlaput}. As described in the proof of Theorem \ref{theoestlplq}, the solution  $\boldsymbol{u}$ can be written in the form \eqref{solutionexplicitesum}. Thus 
\begin{equation}\label{estimlplqstrong1}
\Vert  \boldsymbol{u}(t)\Vert _{\boldsymbol{L}^q(\Omega ) }\,\leq\,\Vert  \boldsymbol{w}_{0}\Vert _{\boldsymbol{L}^q(\Omega ) }\,+\,\Vert  T(t)\widetilde{\boldsymbol{u}}_{0}\Vert _{\boldsymbol{L}^q(\Omega ) }.
\end{equation}
 By  definition of  $\widetilde{\boldsymbol{u}}_{0}$ and Theorem \ref{exishenpA'}, we know that $\widetilde{\boldsymbol{u}}_{0}\in\boldsymbol{V}^{p}_{\sigma,\tau}(\Omega)$ and $T(t)\widetilde{\boldsymbol{u}}_{0}\in \mathbf{D}(A'_{q})$. Furthermore by estimate \eqref{estlplqutxp} we have:
\begin{equation}\label{LpLqumu}
 \Vert T(t)\widetilde{\boldsymbol{u}}_{0}\Vert_{\boldsymbol{L}^{q}(\Omega)}\,\leq\,C\,e^{-\mu t}\,t^{-3/2(1/p-1/q)}\Vert \widetilde{\boldsymbol{u}}_{0}\Vert_{\boldsymbol{L}^{p}(\Omega)}.
 \end{equation}
In the other hand, since $A'_p=A'_q\,$ on $\mathbf{D}(A'_{q})\cap\mathbf{D}(A'_{p}),$ then for all $t>0$, we have $A'_p\,T(t)\widetilde{\boldsymbol{u}}_{0}\,=\,A'_q\,T(t)\widetilde{\boldsymbol{u}}_{0}$. Using that $\Vert T(t)\widetilde{\boldsymbol{u}}_{0}\Vert_{\boldsymbol{W}^{2, q}(\Omega)}$ is equivalent to $\Vert A'_q\,T(t)\widetilde{\boldsymbol{u}}_{0}\Vert _{ \boldsymbol{L}^{q}(\Omega ) }$ we have:
\begin{eqnarray*}
\Vert T(t)\widetilde{\boldsymbol{u}}_{0}\Vert _{ \boldsymbol{L}^q(\Omega ) }&\le & C\, \Vert A'_q\,T(t)\widetilde{\boldsymbol{u}}_{0}\Vert _{ \boldsymbol{L}^{q}(\Omega ) }\\
&=&C\,\Vert (A'_{q})^{\frac{1}{2}} (A'_{q})^{\frac{1}{2}}\, T(t/2) T(t/2)\widetilde{\boldsymbol{u}}_{0}\Vert _{ \boldsymbol{L}^{q}(\Omega ) }
\end{eqnarray*}
Using now the fact that $T(t/2)\widetilde{\boldsymbol{u}}_0\in \mathbf{D}(A'_{q})\subset \mathbf{D}((A'_{q})^{\frac{1}{2}})$, we may commute $T(t/2)$ and $(A'_{q})^{\frac{1}{2}}$. We deduce then
\begin{eqnarray}
\Vert T(t)\widetilde{\boldsymbol{u}}_{0}\Vert _{ \boldsymbol{L}^q(\Omega ) }&\leq &
C\,\Vert (A'_{q})^{\frac{1}{2}}\, T(t/2) (A'_{q})^{\frac{1}{2}}\;T(t/2)\widetilde{\boldsymbol{u}}_{0}\Vert _{ \boldsymbol{L}^{q}(\Omega ) }\nonumber\\
&\le& C\,\Vert (A'_{q})^{\frac{1}{2}}T(t/2)\Vert_{\mathcal{L}(\boldsymbol{X}_{q})} \Vert (A'_{q})^{\frac{1}{2}}T(t/2) \widetilde{\boldsymbol{u}}_{0}\Vert _{ \boldsymbol{L}^{q}(\Omega ) }\label{estlplqstrong2}
\end{eqnarray}
Next using estimate \eqref{semigroupproperty1} we have
\begin{eqnarray*}
\Vert (A'_{q})^{\frac{1}{2}}\,T(t/2)\Vert_{\mathcal{L}(\boldsymbol{X}_{q})} \leq C\,e^{-\mu t} t^{-1/2}.
\end{eqnarray*}
On the other hand since $\widetilde{\boldsymbol{u}}_{0}\in\boldsymbol{V}^{p}_{\sigma,\tau}(\Omega)=\mathbf{D}((A'_{p})^{\frac{1}{2}})$  we may also commute $T(t/2)$ and $(A'_{p})^{\frac{1}{2}}$ in \eqref{estlplqstrong2}. We thus have 
\begin{equation*}
\Vert T(t)\widetilde{\boldsymbol{u}}_{0}\Vert _{ \boldsymbol{L}^q(\Omega ) }\,\leq\,C\,e^{-\mu t} t^{-1/2}\Vert T(t/2) (A'_{q})^{\frac{1}{2}}\widetilde{\boldsymbol{u}}_{0}\Vert _{ \boldsymbol{L}^{q}(\Omega ) }.
\end{equation*}
 Using now estimate \eqref{LpLqumu} we obtain 
 \begin{equation*}
\Vert T(t)\widetilde{\boldsymbol{u}}_{0}\Vert _{ \boldsymbol{L}^q(\Omega ) }\,\leq\,C\,e^{-\mu t} t^{-1/2}t^{-\frac{3}{2}(\frac{1}{p}-\frac{1}{q})}\Vert (A'_{p})^{\frac{1}{2}}\,\widetilde{\boldsymbol{u}}_{0}\Vert _{ \boldsymbol{L}^{p}(\Omega ) }.
\end{equation*}
 Using the fact that $\Vert (A'_{p})^{\frac{1}{2}}\,\widetilde{\boldsymbol{u}}_{0}\Vert _{ \boldsymbol{L}^{p}(\Omega ) }$ is equivalent to the norm $\Vert\widetilde{\boldsymbol{u}}_{0}\Vert _{ \boldsymbol{W}^{1,p}(\Omega ) }$ we get 
  \begin{equation*}
\Vert T(t)\widetilde{\boldsymbol{u}}_{0}\Vert _{ \boldsymbol{L}^q(\Omega ) }\,\leq\,C\,e^{-\mu t} t^{-1/2}t^{-\frac{3}{2}(\frac{1}{p}-\frac{1}{q})}\Vert\widetilde{\boldsymbol{u}}_{0}\Vert _{ \boldsymbol{W}^{1,p}(\Omega ) }.
\end{equation*}
 As a result, using  estimate \eqref{estimlplqstrong1}, estimate \eqref{estlplqutvptau} follows directly.
 \end{proof} 
\begin{rmk} 
\rm{As in Remark \ref{rmkhenp} (ii), for $p=2$, the solution $\boldsymbol{u}$ satisfies (\ref{lrlpw2p}) also for $q=2$.}
\end{rmk}

We may also use the analyticity of the semigroups generated by the operators $B_p$ and $C_p$, proved in Section \ref{semigroupBp} and Section \ref{semigroupCp}. We then deduce the following result, as we did in  Theorem \ref{exishenp}.

\begin{theo}\label{semigroupBC}
  \rm{\textbf{(i)} For all $\boldsymbol{u}_{0}\in[\boldsymbol{H}^{p'}_{0}(\mathrm{div},\Omega)]'_{\sigma,\tau}$ the Problem (\ref{henp}) has a unique solution $\boldsymbol{u}$ satisfying
  \begin{equation}\label{Reghenphdiv1}
\boldsymbol{u}\in
C([0,\,+\infty[,\,[\boldsymbol{H}^{p'}_{0}(\mathrm{div},\Omega)]'_{\sigma,\tau})\cap
C(]0,\,+\infty[,\,\mathbf{D}(B_{p}))\cap
C^{1}(]0,\,+\infty[,\,[\boldsymbol{H}^{p'}_{0}(\mathrm{div},\Omega)]'_{\sigma,\tau}),
\end{equation}
\begin{equation}\label{Reghenphdiv2}
\boldsymbol{u}\in C^{k}(]0,\,+\infty[,\,\mathbf{D}(B^{\ell}_{p})),\qquad
\forall\,k\in\mathbb{N},\,\,\forall\,\ell\in\mathbb{N^{\ast}}.
\end{equation}
Moreover, for all $t>0$:
\begin{equation}\label{esthenphdiv1}
\|\boldsymbol{u}(t)\|_{[\boldsymbol{H}^{p'}_{0}(\mathrm{div},\Omega)]'}\leq\,C(\Omega,p)\,\|\boldsymbol{u}_{0}\|_{[\boldsymbol{H}^{p'}_{0}(\mathrm{div},\Omega)]'},
\end{equation}
\begin{equation}\label{esthenphdiv2}
\Big\|\frac{\partial\boldsymbol{u}(t)}{\partial t}\Big\|_{[\boldsymbol{H}^{p'}_{0}(\mathrm{div},\Omega)]'}\leq\frac{C(\Omega,p)}{t}\,\|\boldsymbol{u}_{0}\|_{[\boldsymbol{H}^{p'}_{0}(\mathrm{div},\Omega)]'}
\end{equation}
and
\begin{equation}\label{esthenphdiv3}
\|\boldsymbol{u}(t)\|_{\boldsymbol{W}^{1,p}(\Omega)}\leq\,C(\Omega,p)\,(1+\frac{1}{t})\|\boldsymbol{u}_{0}\|_{[\boldsymbol{H}^{p'}_{0}(\mathrm{div},\Omega)]'}.
\end{equation}

\noindent\textbf{(ii)} For every  $\boldsymbol{u}_{0}\in[\boldsymbol{T}^{p'}(\Omega)]'_{\sigma,\tau}$ the Problem (\ref{henp}) has a unique solution $\boldsymbol{u}$ satisfying
  \begin{equation}\label{Reghenptp1}
\boldsymbol{u}\in
C([0,\,+\infty[,\,[\boldsymbol{T}^{p'}(\Omega)]'_{\sigma,\tau})\cap
C(]0,\,+\infty[,\,\mathbf{D}(C_{p}))\cap
C^{1}(]0,\,+\infty[,\,[\boldsymbol{T}^{p'}(\Omega)]'_{\sigma,\tau}),
\end{equation}
\begin{equation}\label{Reghenptp2}
\boldsymbol{u}\in C^{k}(]0,\,+\infty[,\,\mathbf{D}(C^{\ell}_{p})),\qquad
\forall\,k\in\mathbb{N},\,\,\forall\,\ell\in\mathbb{N^{\ast}}.
\end{equation}
Moreover, for all $t>0$:
\begin{equation}\label{esthenptp1}
\|\boldsymbol{u}(t)\|_{[\boldsymbol{T}^{p'}(\Omega)]'}\leq\,C(\Omega,p)\,\|\boldsymbol{u}_{0}\|_{[\boldsymbol{T}^{p'}(\Omega)]'},
\end{equation}
\begin{equation}\label{esthenptp2}
\Big\|\frac{\partial\boldsymbol{u}(t)}{\partial t}\Big\|_{[\boldsymbol{T}^{p'}(\Omega)]'}\leq\frac{C(\Omega,p)}{t}\,\|\boldsymbol{u}_{0}\|_{[\boldsymbol{T}^{p'}(\Omega)]'}.
\end{equation}
and
\begin{equation}\label{esthenptp3}
\|\boldsymbol{u}(t)\|_{\boldsymbol{L}^{p}(\Omega)}\leq\,C(\Omega,p)\,(1\,+\,\frac{1}{t})\,\|\boldsymbol{u}_{0}\|_{[\boldsymbol{T}^{p'}(\Omega)]'}.
\end{equation} } 
\end{theo}
\begin{proof}
 The part (i) of Theorem \ref{semigroupBC} follows applying the semi-group theory to the operator $B_{p}$ (given by \eqref{b2}). The part (ii) follows applying the semi-group theory to the operator $C_{p}$ (given by \eqref{C2}).
\end{proof}  
In the same  way as we deduced Corollary \ref{corolrlpw1p}, we deduce the following Corollary  from  Theorem \ref{semigroupBC}.
\begin{coro}[Very weak solutions for the homogeneous Stokes Problem]\label{veryweak}
  Let $\boldsymbol{u}_{0}\in[\boldsymbol{H}^{p'}_{0}(\mathrm{div},\Omega)]'_{\sigma,\tau}$, $T<\infty$ and let $\boldsymbol{u}$ be the unique solution of Problem (\ref{henp}) given by Theorem \ref{semigroupBC}, (i). Then  $\boldsymbol{u}$ satisfies
\begin{equation}\label{lrl1p}
\forall\,q\in [1, 2),\qquad\boldsymbol{u}\in L^{q}(0,T;\,\boldsymbol{L}^{p}(\Omega))\,\,\,\,\mathrm{and}\qquad\frac{\partial\boldsymbol{u}}{\partial t}\in L^{q}(0,T;\,[\boldsymbol{T}^{p'}(\Omega)]'_{\sigma,\tau}).
\end{equation}
  \end{coro}
  \begin{proof}
  Using the semi-group theory we know that the solution $\boldsymbol{u}(t)\in\boldsymbol{W}^{1,p}(\Omega)$ for all $t>0$. As a result, using the interpolation inequality we have
  \begin{equation}\label{1}
  \Vert\boldsymbol{u}(t)\Vert_{\boldsymbol{L}^{p}(\Omega)}\,\leq\,C(\Omega,p)\, \Vert\boldsymbol{u}(t)\Vert^{1/2}_{\boldsymbol{W}^{1,p}(\Omega)} \Vert\boldsymbol{u}(t)\Vert^{1/2}_{\boldsymbol{W}^{-1,p}(\Omega)}.
  \end{equation}
  On the other hand, thanks to Corollary \ref{existenceweaklaplacian} we know that
  \begin{eqnarray}
   \Vert\boldsymbol{u}(t)\Vert_{\boldsymbol{W}^{1,p}(\Omega)}&\simeq&\Vert\boldsymbol{u}(t)\Vert_{[\boldsymbol{H}^{p'}_{0}(\mathrm{div},\Omega)]'}\,+\,\Vert\Delta\boldsymbol{u}(t)\Vert_{[\boldsymbol{H}^{p'}_{0}(\mathrm{div},\Omega)]'}\nonumber\\
   &\leq& \big(1+\frac{1}{t}\big)\,\Vert\boldsymbol{u}_{0}\Vert_{\boldsymbol{H}^{p'}_{0}(\mathrm{div},\Omega)]'}\label{2}.
 \end{eqnarray}  
 Moreover, thanks to the continuous embeddings $[\boldsymbol{H}^{p'}_{0}(\mathrm{div},\Omega)]'\hookrightarrow\boldsymbol{W}^{-1,p}(\Omega)$ and to the semi-group theory we have
 \begin{equation}\label{3}
 \Vert\boldsymbol{u}(t)\Vert_{\boldsymbol{W}^{-1,p}(\Omega)}\,\leq\,C(\Omega,p)\,\Vert\boldsymbol{u}(t)\Vert_{[\boldsymbol{H}^{p'}_{0}(\mathrm{div},\Omega)]'}\,\leq\,C(\Omega,p)\,\Vert\boldsymbol{u}_{0}\Vert_{[\boldsymbol{H}^{p'}_{0}(\mathrm{div},\Omega)]'}.
 \end{equation}
 As a result, putting together (\ref{1}), (\ref{2}) and \eqref{3} one gets
 \begin{equation*}
  \Vert\boldsymbol{u}(t)\Vert_{\boldsymbol{L}^{p}(\Omega)}\,\leq\,C(\Omega,p)\,\Big(1+\frac{1}{t}\Big)^{1/2}\,\Vert\boldsymbol{u}_{0}\Vert_{\boldsymbol{H}^{p'}_{0}(\mathrm{div},\Omega)]'}.
 \end{equation*}
 Thus, for every $T<\infty$ and for every $1\leq q<2$, $\boldsymbol{u}\in L^{q}(0,T;\,\boldsymbol{L}^{p}(\Omega))$.
 
 It remains to prove that $\frac{\partial\boldsymbol{u}}{\partial t}\in L^{q}(0,T;\,[\boldsymbol{T}^{p'}(\Omega)]'_{\sigma,\tau})$.  We proceed in a similar way as in the proof of Corollary \ref{corolrlpw1p}.
We set $$\widetilde{\boldsymbol{u}}(t)=\boldsymbol{u}(t)-\sum_{j=1}^{J}\langle\boldsymbol{u}(t)\cdot\boldsymbol{n}\,,\,1\rangle_{\Sigma_{j}}\widetilde{\boldsymbol{\mathrm{grad}}}\,q_{j}^{\tau}.$$ It is clear that $\boldsymbol{u}(t)\,=\,\widetilde{\boldsymbol{u}}(t)\,+\,\sum_{j=1}^{J}\langle\boldsymbol{u}(t)\cdot\boldsymbol{n}\,,\,1\rangle_{\Sigma_{j}}\widetilde{\boldsymbol{\mathrm{grad}}}\,q_{j}^{\tau}.$  Moreover thanks to \cite[Theorem 4.15]{Am3} we know that 
\begin{equation*}
\Vert\Delta\boldsymbol{u}\Vert_{[\boldsymbol{T}^{p'}(\Omega)]'}\,=\,\Vert\Delta\widetilde{\boldsymbol{u}}\Vert_{[\boldsymbol{T}^{p'}(\Omega)]'}\,\simeq\,\Vert\widetilde{\boldsymbol{u}}\Vert_{\boldsymbol{L}^{p}(\Omega)}\,\leq\, \Vert\boldsymbol{u}\Vert_{\boldsymbol{L}^{p}(\Omega)}.
\end{equation*}
The last inequality comes from the fact (see \cite[Lemma 3.2]{Am4}) 
\begin{equation*}
\vert\langle\boldsymbol{u}\cdot\boldsymbol{n}\,,\,1\rangle_{\Sigma_{j}}\vert\,\leq\,C(\Omega,p)\,\Vert\boldsymbol{u}\Vert_{\boldsymbol{L}^{p}(\Omega)}.
\end{equation*}
Thus 
 $\frac{\partial\boldsymbol{u}}{\partial t}=\Delta\boldsymbol{u}\in L^{q}(0,T;\,[\boldsymbol{T}^{p'}(\Omega)]')$ and the result is proved.
  \end{proof}
 
  \begin{rmk}[Optimal initial value]
\rm{
It is may be an important question to know what is the optimal (weakest possible) initial value to obtain a unique strong, weak or very weak solution to Problem \eqref{henp}.

\noindent\textbf{(i)} A unique solution $\boldsymbol{u}$ of Problem \eqref{henp} is said to be a strong solution if it satisfies 
$$1<p,q<\infty,\,\,\, T\leq\infty,\,\,\,\boldsymbol{u}\in L^{q}(0,T;\,\boldsymbol{W}^{2,p}(\Omega)),\,\,\,\textrm{and}\,\,\,\, \frac{\partial\boldsymbol{u}}{\partial t}\in L^{q}(0,T;\,\boldsymbol{L}^{p}(\Omega)).$$  

\noindent The assumption $\boldsymbol{u}_{0}\in\boldsymbol{X}^{p}_{\sigma,\tau}(\Omega)$ is not optimal and may be replaced by the properties 
\begin{equation}\label{opinitval}
\boldsymbol{u}_{0}\in\boldsymbol{L}^{p}_{\sigma,\tau}(\Omega),\qquad\int_{0}^{\infty}\Vert A_{p}T(t)\boldsymbol{u}_{0}\Vert^{q}_{\boldsymbol{L}^{p}(\Omega)}\textrm{d}\,t<\infty,
\end{equation}
where $1<p,q<\infty$ and $(T(t))_{\geq0}$ is the semi-group generated by the Stokes operator on $\boldsymbol{L}^{p}_{\sigma,\tau}(\Omega)$. 
With an initial value $\boldsymbol{u}_{0}$ satisfying \eqref{opinitval} the unique solution $\boldsymbol{u}$ of Problem \eqref{henp} satisfies \eqref{lrlpw2p} for all $1<p,q<\infty$
and for all $T\leq\infty$ (see Proposition \ref{StrongSolution}).

\noindent\textbf{(ii)} A unique solution $\boldsymbol{u}$ of Problem \eqref{henp} is said to be a weak solution if it satisfies
\begin{equation}\label{maxreginithdiv}
1<p,q<\infty,\,\,\, T\leq\infty,\,\,\,\boldsymbol{u}\in L^{q}(0,T;\,\boldsymbol{W}^{1,p}(\Omega)),\,\,\,\textrm{and}\,\,\,\, \frac{\partial\boldsymbol{u}}{\partial t}\in L^{q}(0,T;\,[\boldsymbol{H}^{p'}_{0}(\mathrm{div},\Omega)]').
\end{equation} 
The optimal choice of the initial value $\boldsymbol{u}_{0}$ to obtain a unique weak solution to Problem \eqref{henp} satisfying the maximal regularity \eqref{maxreginithdiv} is 
\begin{equation}\label{optinitvalhdiv}
\boldsymbol{u}_{0}\in[\boldsymbol{H}^{p'}_{0}(\mathrm{div},\Omega)]'_{\sigma,\tau},\qquad\int_{0}^{\infty}\Vert B_{p}T(t)\boldsymbol{u}_{0}\Vert^{q}_{[\boldsymbol{H}^{p'}_{0}(\mathrm{div},\Omega)]'}\textrm{d}\,t<\infty,
\end{equation}
where $1<p,q<\infty$ and $(T(t))_{t\geq0}$ is the semi-group generated by the Stokes operator on $[\boldsymbol{H}^{p'}_{0}(\mathrm{div},\Omega)]'_{\sigma,\tau}$. Observe that for the weak solution, the choice of an initial value  $\boldsymbol{u}_{0}$ satisfying \eqref{optinitvalhdiv} is better than $\boldsymbol{u}_{0}\in\boldsymbol{L}^{p}_{\sigma,\tau}(\Omega)$ (see Corollary \ref{corolrlpw1p}), since it allows us to obtain a unique solution satisfying \eqref{lrw1p} for all $1<p,q<\infty$ and for $T=\infty$ included.

\noindent\textbf{(iii)} A unique solution $\boldsymbol{u}$ of Problem \eqref{henp} is said to be a very weak solution if it satisfies 
\begin{equation}\label{maxreginittp}
1<p,q<\infty,\,\,\, T\leq\infty,\,\,\,\boldsymbol{u}\in L^{q}(0,T;\,\boldsymbol{L}^{p}(\Omega)),\,\,\,\textrm{and}\,\,\,\, \frac{\partial\boldsymbol{u}}{\partial t}\in L^{q}(0,T;\,[\boldsymbol{T}^{p'}(\Omega)]').
\end{equation} 
The optimal choice of the initial value $\boldsymbol{u}_{0}$ to obtain a unique very weak solution to Problem \eqref{henp} satisfying the maximal regularity \eqref{maxreginittp} is 
\begin{equation}\label{optinitvaltp}
\boldsymbol{u}_{0}\in[\boldsymbol{T}^{p'}(\Omega)]'_{\sigma,\tau},\qquad\int_{0}^{\infty}\Vert C_{p}T(t)\boldsymbol{u}_{0}\Vert^{q}_{[\boldsymbol{T}^{p'}(\Omega)]'}\textrm{d}\,t<\infty,
\end{equation}
where $1<p,q<\infty$ and $(T(t))_{t\geq0}$ is the semi-group generated by the Stokes operator on $[\boldsymbol{T}^{p'}(\Omega)]'_{\sigma,\tau}$. Notice that for the very weak solution, the choice of an initial value  $\boldsymbol{u}_{0}$ satisfying \eqref{optinitvaltp} is better than $\boldsymbol{u}_{0}\in[\boldsymbol{H}^{p'}_{0}(\mathrm{div}),\Omega)]'_{\sigma,\tau}$ (see Theorem \ref{veryweak}), since it allows us to obtain a unique solution satisfying \eqref{lrl1p} for all $1<p,q<\infty$ and for $T=\infty$ included.
} 
\end{rmk}

We present now the remaining results for the homogeneous Stokes system with flux conditions. As it was said in the Introduction, they  are very similar, although with some differences, to those for the  problem without flux condition that are described just above. As for the proofs, they are also very similar and actually simpler to those without flux condition, reason for which we will not give all of them in detail.

\begin{rmk}
\label{solflux}
\rm{By \eqref{Reghenp1A'}, the function $\boldsymbol{u}$ that is obtained in Theorem \ref{exishenpA'} solves Problem (\ref{henp}) and also satisfies  condition \eqref{condition2}. Then, for all $\pi \in \R$, $(\boldsymbol{u}, \pi )$ is a solution of the Stokes problem with flux  (\ref{lens}), (\ref{nbc}), \eqref{condition2}.}
\end{rmk}

\begin{rmk}
\rm{Notice that the decay rates The estimates  \eqref{estlplqutxp}--\eqref{estlplqlaputxp} for the solution $\boldsymbol{u}(t)$ are exponential, and not algebraic as in (\ref{esthenp1})-(\ref{esthenp3}) of Theorem \ref{exishenp}.}
\end{rmk}

{\begin{rmk}
\rm{
For $p=2$, the solution $\boldsymbol{u}$
can be written explicitly in the form
\begin{equation*}
\boldsymbol{u}(t)\,=\,\sum_{k=1}^{+\infty}\beta_{k}\,e^{-\lambda_{k}\,t}\,\boldsymbol{z}_{k},\qquad
\beta_{k}=\int_{\Omega}\boldsymbol{u}_{0}\cdot\overline{\boldsymbol{z}_{k}}\,\textrm{d}\,x
\end{equation*}
and the exponential decay with respect to time can be obtained directly. Moreover, contrary to the case $p\neq2$ one has  
\begin{equation}\label{7.9}
\|\boldsymbol{u}(t)\|_{\boldsymbol{L}^{2}(\Omega)}\leq\,e^{-\lambda_{1}t}\,\|\boldsymbol{u}_{0}\|_{\boldsymbol{L}^{2}(\Omega)}.
\end{equation}
It is clear that estimate (\ref{7.9}) yields a faster decay rate than (\ref{estlplqutxp}).
We recall that $\lambda_{1}$ is the first eigenvalue for the operator $A'_{p}$ and it is equal to $\frac{1}{C_{2}(\Omega)}$ where $C_{2}(\Omega)$ is the constant of the Poincaré-type inequality (\ref{pc1}).
}
\end{rmk}
In our next Theorem we consider initial data $\boldsymbol{u}_{0}$ belonging to $\boldsymbol{Y}_{p}$ and to $\boldsymbol{Z}_{p}$.
  \begin{theo}
  \label{brpimecprime}
  \rm{\textbf{(i)} For all $\boldsymbol{u}_{0}\in\boldsymbol{Y}_{p}$ the Problem
 (\ref{henp}) has a unique solution $\boldsymbol{u}$ satisfying
\begin{equation}\label{Reghenp1hdivB'}
\boldsymbol{u}\in
C([0,\,+\infty[,\,\boldsymbol{Y}_{p})\cap
C(]0,\,+\infty[,\,\mathbf{D}(B'_{p}))\cap
C^{1}(]0,\,+\infty[,\,\boldsymbol{Y}_{p}),
\end{equation}
\begin{equation}\label{Reghenp2hdivB'}
\boldsymbol{u}\in C^{k}(]0,\,+\infty[,\,\mathbf{D}(B'^{\ell}_{p})),\qquad
\forall\,k,\,\ell\in\mathbb{N}.
\end{equation}
Moreover there exists a constant $C(\Omega,p)$ and a constant $\mu>0$, such that, for all $t>0$:
\begin{equation}\label{esthenp1hdivB'}
\|\boldsymbol{u}(t)\|_{[\boldsymbol{H}^{p'}_{0}(\mathrm{div},\,\Omega)]'}\leq\,C(\Omega,p)\,e^{-\mu\,t}\,\|\boldsymbol{u}_{0}\|_{[\boldsymbol{H}^{p'}_{0}(\mathrm{div},\,\Omega)]'},
\end{equation}
\begin{equation}\label{esthenp2hdivB'}
\Big\Vert\frac{\partial\boldsymbol{u}(t)}{\partial t}\Big\Vert_{[\boldsymbol{H}^{p'}_{0}(\mathrm{div},\,\Omega)]'}\,\leq\,C(\Omega,p)\,\frac{e^{-\mu\,t}}{t}\,\|\boldsymbol{u}_{0}\|_{[\boldsymbol{H}^{p'}_{0}(\mathrm{div},\,\Omega)]'}
\end{equation}
and
\begin{equation}\label{esthenp3hdivB'}
\|\boldsymbol{u}(t)\|_{\boldsymbol{W}^{1,p}(\Omega)}\leq\,C(\Omega,p)\,\frac{e^{-\mu\,t}}{t}\,\|\boldsymbol{u}_{0}\|_{[\boldsymbol{H}^{p'}_{0}(\mathrm{div},\,\Omega)]'}.
\end{equation}

\noindent \textbf{(ii)} For all  $\boldsymbol{u}_{0}\in\boldsymbol{Z}_{p}$ the Problem
 (\ref{henp}) has a unique solution $\boldsymbol{u}$ satisfying
\begin{equation}\label{Reghenp1tpC'}
\boldsymbol{u}\in
C([0,\,+\infty[,\,\boldsymbol{Z}_{p})\cap
C(]0,\,+\infty[,\,\mathbf{D}(C'_{p}))\cap
C^{1}(]0,\,+\infty[,\,\boldsymbol{Z}_{p}),
\end{equation}
\begin{equation}\label{Reghenp2tpC'}
\boldsymbol{u}\in C^{k}(]0,\,+\infty[,\,\mathbf{D}(C'^{\ell}_{p})),\qquad
\forall\,k,\,\ell\in\mathbb{N}.
\end{equation}
Moreover there exists a constant $C(\Omega,p)$ and a constant $\mu>0$, such that, for all $t>0$:
\begin{equation}\label{esthenp1tpC'}
\|\boldsymbol{u}(t)\|_{[\boldsymbol{T}^{p'}(\Omega)]'}\leq\,C(\Omega,p)\,e^{-\mu\,t}\,\|\boldsymbol{u}_{0}\|_{[\boldsymbol{T}^{p'}(\Omega)]'},
\end{equation}
\begin{equation}\label{esthenp2tpC'}
\Big\Vert\frac{\partial\boldsymbol{u}(t)}{\partial t}\Big\Vert_{[\boldsymbol{T}^{p'}(\Omega)]'}\,\leq\,C(\Omega,p)\,\frac{e^{-\mu\,t}}{t}\,\|\boldsymbol{u}_{0}\|_{[\boldsymbol{T}^{p'}(\Omega)]'}
\end{equation}
and
\begin{equation}\label{esthenp3tpC'}
\|\boldsymbol{u}(t)\|_{\boldsymbol{L}^{p}(\Omega)}\leq\,C(\Omega,p)\,\frac{e^{-\mu\,t}}{t}\,\|\boldsymbol{u}_{0}\|_{[\boldsymbol{T}^{p'}(\Omega)]'}.
\end{equation}
}
\end{theo} 
\begin{proof}
The theorem follows by the classical semigroup theory applied to the analytic  semigroups  generated by the operators $B'_p$ and $C'_p$.
\end{proof}

\begin{rmk}
\label{solfluxBIS}
\rm{By \eqref{Reghenp1hdivB'} and \eqref{Reghenp2tpC'}, the functions $\boldsymbol{u}$ obtained in Theorem \ref{brpimecprime} solve Problem (\ref{henp}) and  satisfy  condition \eqref{condition2}. Then, for all $\pi \in \R$, $(\boldsymbol{u}, \pi )$ is a solution of the Stokes problem with flux  (\ref{lens}), (\ref{nbc}), \eqref{condition2}.}
\end{rmk}

\subsection{The inhomogeneous problem}
Given the  Cauchy-Problem:
\begin{equation}\label{abstcauchprob}
\left\{
\begin{array}{cc}
\frac{\partial u}{\partial t}\,+\,\mathcal{A}\,u(t)=\textit{f}(t)& 0\leq t\leq T\\
u(0)=0,&
\end{array}
\right.
\end{equation}
where $-\mathcal{A}$ is the infinitesimal generator of an analytic semi-group on a Banach space $X$  and $\textit{f}\in L^{p}(0,T;\,X)$,
the analyticity of $-\mathcal{A}$ is not enough in general  to ensure that  solutions to Problem (\ref{abstcauchprob}) satisfy
\begin{equation}\label{regulacp}
u\in W^{1,p}(0,T;\,X)\cap L^{p}(0,T;\,D(\mathcal{A})).
\end{equation}
Although it is enough when  $X$ is a Hilbert space, (see \cite{Ben, Sob} for instance), 
in general it is necessary to impose some further regularity condition on $\textit{f}$  such as H$\mathrm{\ddot{o}}$lder continuity, (see \cite{Pa} for instance). However, using the concept of $\zeta$-convexity and a perturbation argument,  the existence of a solution to Problem (\ref{abstcauchprob}) satisfying (\ref{regulacp}), when the pure imaginary powers of $\mathcal{A}$ satisfy estimate (\ref{estimpur}) is proved in \cite{DV, GiGa4}. Moreover,  \cite[Theorem 2.1]{GiGa4} extends  \cite[Theorem 3.2]{DV} in two directions: First, the operator $\mathcal{A}$ may not have bounded inverse and second, the maximal interval of time  $T$ may be infinite. In the case of a Hilbert space it was proved in  \cite{Ka1, Ka2}  that the pure imaginary powers of a maximal accretive operator are bounded and satisfy estimates of type (\ref{estimpur}).
 
For the sake of completeness we state  the following theorem that is proved in \cite{GiGa4} (cf. Theorem 2.1). 
\begin{theo}\label{existabscp}
Let $X$ be a $\zeta$-convex Banach space. Assume that $0<T\leq\infty$, $1<p<\infty$ and that $\mathcal{A}\in\mathcal{E}^{\theta}_{K}(X)$ for some $K\geq 1$, $0\leq\theta<\pi/2$ and $\mathcal{E}^{\theta}_{K}(X)$ as in Definition \ref{EthetaK}. Then for every $\textit{f}\in L^{p}(0,T;\,X)$ there exists a unique solution $\boldsymbol{u}$ of the  Cauchy-Problem (\ref{abstcauchprob}) satisfying the properties:
\begin{equation*}
u\in L^{p}(0,T_{0};\,D(\mathcal{A})),\,\,\,\,  T_{0}\leq T\,\,\,\textrm{if}\,\,\,T<\infty\,\,\,\, \mathrm{and }\,\,\,\,T_{0}<T\,\,\,\textrm{if}\,\,\,T=\infty,
\end{equation*}
\begin{equation*}
\frac{\partial u}{\partial t}\in L^{p}(0,T;\,X)
\end{equation*}
and
\begin{equation*}
\int_{0}^{T}\Big\Vert\frac{\partial u}{\partial t}\Big\Vert^{p}_{X}\,\mathrm{d}\,t\,+\,\int_{0}^{T}\Vert\mathcal{A}u(t)\Vert^{p}_{X}\,\mathrm{d}\,t\,\leq\,C\,\int_{0}^{T}\Vert\textit{f}(t)\Vert_{X}^{p}\,\mathrm{d}\,t
\end{equation*}
with $C=C(p,\theta,K,X)$ independent of $\textit{f}$ and $T$.
\end{theo}

Let us consider now the non homogeneous  Problem:
 \begin{equation}
 \label{inhensp}
 \left\{
\begin{array}{cccc}
\frac{\partial\boldsymbol{u}}{\partial t} - \Delta \boldsymbol{u
}=\boldsymbol{f},& 
\mathrm{div}\,\boldsymbol{u}= 0 &\mathrm{in}&\Omega\times (0,T), \\
\boldsymbol{u}\cdot\boldsymbol{n}=0,& 
\boldsymbol{\mathrm{curl}}\,\boldsymbol{u}\times \boldsymbol{n} = \boldsymbol{0} &\mathrm{on} & \Gamma\times (0,T), \\
&\boldsymbol{u}(0)=\boldsymbol{0} &\mathrm{in}&
\Omega,
\end{array}
\right.
\end{equation}
where $\boldsymbol{f}\in
 L^{q}(0,T;\,\boldsymbol{L}^{p}_{\sigma,\tau}(\Omega))$ and $1<p,q<\infty$. 
 
 In a first step, applying  Theorem \ref{existabscp} to  $(I+A_{p})$ we  obtain strong-solutions to the Stokes Problem (\ref{lens}), (\ref{nbc}) and  maximal regularity  $L^{p}-L^{q}$ estimates. This is done in the   following Theorem.

 \begin{theo}\label{existinhensp}Let $1<p,q<\infty$ and $0<T\leq\infty$. Then for every $\boldsymbol{f}\in L^{q}(0,T;\,\boldsymbol{L}^{p}_{\sigma,\tau}(\Omega))$ there exists a unique solution $\boldsymbol{u}$ of  \eqref{inhensp} satisfying
\begin{equation}\label{reglplqlap1}
\boldsymbol{u}\in L^{q}(0,T_{0};\,\boldsymbol{W}^{2,p}(\Omega)),\,\,\,\,  T_{0}\leq T\,\,\,\textrm{if}\,\,\,T<\infty\,\,\,\, \mathrm{and }\,\,\,\,T_{0}<T\,\,\,\textrm{if}\,\,\,T=\infty,
\end{equation}
\begin{equation}\label{reglplqlap2}
\frac{\partial\boldsymbol{u}}{\partial t}\in L^{q}(0,T;\,\boldsymbol{L}^{p}_{\sigma,\tau}(\Omega))
\end{equation}
and
\begin{equation}\label{estlplqlap}
\int_{0}^{T}\Big\Vert\frac{\partial\boldsymbol{u}}{\partial t}\Big\Vert^{q}_{\boldsymbol{L}^{p}(\Omega)}\,\mathrm{d}\,t\,+\,\int_{0}^{T}\Vert\Delta\boldsymbol{u}(t)\Vert^{q}_{\boldsymbol{L}^{p}(\Omega)}\,\mathrm{d}\,t\,\leq\,C(p,q,\Omega)\,\int_{0}^{T}\Vert\boldsymbol{f}(t)\Vert^{q}_{\boldsymbol{L}^{p}(\Omega)}\,\mathrm{d}\,t.
\end{equation}
\end{theo}
\begin{proof}
It is well known (see \cite{Pa} for instance), that since the operator $-A_{p}$ generates a bounded analytic semi-group in $\boldsymbol{L}^{p}_{\sigma,\tau}(\Omega)$ (see Theorem \ref{analsemi2}), then Problem \eqref{inhensp}  has a unique solution $\boldsymbol{u}\in C(0,T;\,\boldsymbol{L}^{p}_{\sigma,\tau}(\Omega))$. To prove the maximal $L^{p}-L^{q}$ regularity \eqref{reglplqlap1}-\eqref{estlplqlap} we proceed as follows.  Let $\mu>0$ and set $\boldsymbol{u}_{\mu}(t)=e^{-\frac{1}{\mu^{2}}t}\boldsymbol{u}(t)$. Notice that $\boldsymbol{u}_{\mu}(t)$ is a solution of the problem
\begin{equation}\label{inhenspv(t)}
 \left\{
\begin{array}{cccc}
\frac{\partial\boldsymbol{u}_{\mu}}{\partial t} + (\frac{1}{\mu^{2}}I\,+\,A_{p})\boldsymbol{u}_{\mu}(t)=e^{-\frac{1}{\mu^{2}}t}\boldsymbol{f},& 
\mathrm{div}\,\boldsymbol{u}_{\mu}(t)= 0 &\mathrm{in}&\Omega\times (0,T), \\
\boldsymbol{u}_{\mu}(t)\cdot\boldsymbol{n}=0,& 
\boldsymbol{\mathrm{curl}}\,\boldsymbol{u}_{\mu}(t)\times \boldsymbol{n} = \boldsymbol{0} &\mathrm{on} & \Gamma\times (0,T), \\
&\boldsymbol{u}_{\mu}(0)=\boldsymbol{u}(0)=\boldsymbol{0} &\mathrm{in}&
\Omega,
\end{array}
\right.
\end{equation}
where the function $e^{-\frac{1}{\mu^{2}}t}\boldsymbol{f}\in L^{q}(0,T;\,\boldsymbol{L}^{p}_{\sigma,\tau}(\Omega))$. Since the pure imaginary powers of the operator $(\frac{1}{\mu^{2}}I\,+\,A_{p})$ are bounded in $\boldsymbol{L}^{p}_{\sigma,\tau}(\Omega)$  (see Theorem \ref{pureimg1+lap}) and since for all $1<p<\infty$, $\boldsymbol{L}^{p}_{\sigma,\tau}(\Omega)$ is $\zeta$-convex  (see Proposition \ref{zetaconvexsubsp}), we can apply Theorem \ref{existabscp} to the operator $(\frac{1}{\mu^{2}}I\,+\,A_{p})$. Thus we deduce that the solution $\boldsymbol{u}_{\mu}(t)$ of the Problem \eqref{inhenspv(t)} satisfies the following maximal $L^{p}-L^{q}$ regularity
\begin{equation}\label{reglplqlap1vt}
\boldsymbol{u}_{\mu}\in L^{q}(0,T_{0};\,\mathbf{D}(A_{p}))\cap W^{1,q}(0,T;\,\boldsymbol{L}^{p}_{\sigma,\tau}(\Omega)),
\end{equation}
with $T_{0}\leq T\,$ if $T<\infty\,$ and $T_{0}<T\,$ if $T=\infty$. Moreover the following estimate holds
\begin{multline}\label{estlplqlapvt}
\int_{0}^{T}\Big\Vert\frac{\partial\boldsymbol{u}_{\mu}}{\partial t}\Big\Vert^{q}_{\boldsymbol{L}^{p}(\Omega)}\,\mathrm{d}\,t\,+\,\int_{0}^{T}\Big\Vert\Big(\frac{1}{\mu^{2}}I\,+\,A_{p}\Big)\boldsymbol{u}_{\mu}(t)\Big\Vert^{q}_{\boldsymbol{L}^{p}(\Omega)}\,\mathrm{d}\,t\,\leq\\
C(p,q,\Omega)\,\int_{0}^{T}\Vert e^{-\frac{1}{\mu^{2}}t}\boldsymbol{f}(t)\Vert^{q}_{\boldsymbol{L}^{p}(\Omega)}\,\mathrm{d}\,t\leq C(p,q,\Omega)\,\int_{0}^{T}\Vert\boldsymbol{f}(t)\Vert^{q}_{\boldsymbol{L}^{p}(\Omega)}\,\mathrm{d}\,t,
\end{multline}
where the constant $C(p,q,\Omega)$ is independent of $\mu$.

Now observe that the solution $\boldsymbol{u}$ of the Problem \eqref{inhensp} can be written as $\boldsymbol{u}(t)=e^{\frac{1}{\mu^{2}}}\boldsymbol{u}_{\mu}(t).\,$ As a result, since for all $t\in(0,T)$ and for all $1<p,q<\infty$ the function $\boldsymbol{u}_{\mu}$ satisfies \eqref{reglplqlap1vt}, then the solution $\boldsymbol{u}$ satisfies also 
$$\boldsymbol{u}\in L^{q}(0,T_{0};\,\mathbf{D}(A_{p}))\cap W^{1,q}(0,T;\,\boldsymbol{L}^{p}_{\sigma,\tau}(\Omega)),$$
with $T_{0}\leq T\,$ if $T<\infty\,$ and $T_{0}<T\,$ if $T=\infty$.  Moreover  as $\mu\rightarrow\infty$ we have
\begin{equation}\label{convergenceumuu}
\boldsymbol{u}_{\mu}\longrightarrow\boldsymbol{u}\quad\mathrm{in}\,\,\, L^{q}(0,T_{0};\,\mathbf{D}(A_{p}))\cap W^{1,q}(0,T_{0};\,\boldsymbol{L}^{p}_{\sigma,\tau}(\Omega)).
\end{equation}
 It remains to prove estimate \eqref{estlplqlap}. First we recall that for every function $\boldsymbol{v}\in\mathbf{D}(A_{p})$ the following three quantities $\Vert \boldsymbol{v}\Vert_{\boldsymbol{W}^{2,p}(\Omega)},\,$ $\Vert \boldsymbol{v}\Vert_{\mathbf{D}(A_{p})}\,$ and 
 $\Big\Vert\Big(\frac{1}{\mu^{2}}I\,+\,A_{p}\Big)\boldsymbol{v}\Big\Vert_{\boldsymbol{L}^{p}(\Omega)}$ are equivalent norms on $\mathbf{D}(A_{p})$.
 As a result, substituting in \eqref{estlplqlapvt} we have
 \begin{multline}\label{estlplqlapvt2}
\int_{0}^{T}\Big\Vert\frac{\partial\boldsymbol{u}_{\mu}}{\partial t}\Big\Vert^{q}_{\boldsymbol{L}^{p}(\Omega)}\,\mathrm{d}\,t\,+\,\int_{0}^{T}\Vert\boldsymbol{u}_{\mu}(t)\Vert^{q}_{\mathbf{D}(A_{p})}\,\mathrm{d}\,t\,\leq
C(p,q,\Omega)\,\int_{0}^{T}\Vert\boldsymbol{f}(t)\Vert^{q}_{\boldsymbol{L}^{p}(\Omega)}\,\mathrm{d}\,t,
\end{multline}
where the constant $C(p,q,\Omega)$ is independent of $\mu$. Now using \eqref{convergenceumuu}, the dominated convergence theorem and passing to the limit as $\mu\longrightarrow\infty$ in \eqref{estlplqlapvt2} we have 
\begin{equation*}
\int_{0}^{T}\Big\Vert\frac{\partial\boldsymbol{u}}{\partial t}\Big\Vert^{q}_{\boldsymbol{L}^{p}(\Omega)}\,\mathrm{d}\,t\,+\,\int_{0}^{T}\Vert\boldsymbol{u}(t)\Vert^{q}_{\mathbf{D}(A_{p})}\,\mathrm{d}\,t\,\leq\,
C(p,q,\Omega)\,\int_{0}^{T}\Vert\boldsymbol{f}(t)\Vert^{q}_{\boldsymbol{L}^{p}(\Omega)}\,\mathrm{d}\,t.
\end{equation*}
Finally using the fact that $\Vert \boldsymbol{u}\Vert_{\mathbf{D}(A_{p})}$ is equivalent to $\Vert \boldsymbol{u}\Vert_{\boldsymbol{L}^{p}(\Omega)}\,+\,\Vert A_{p}\boldsymbol{u}\Vert_{\boldsymbol{L}^{p}(\Omega)}$,  estimate \eqref{estlplqlap} follows directly.
\end{proof}

\medskip

\medskip

We prove now Theorem \ref{Exisinhnsplp}. It extends  the previous result  to the more genera case where  the external force $\boldsymbol{f}\in L^{q}(0,T;\,\boldsymbol{L}^{p}(\Omega))$ is not necessarily divergence free. The following theorem shows that the pressure can be decoupled from the problem using the weak Neumann Problem (\ref{wn.1}). 

\begin{proof} 

[of Theorem \ref{Exisinhnsplp}]
Let
$\boldsymbol{f}\in\boldsymbol{L}^{q}(0,T;\,\boldsymbol{L}^{p}(\Omega))$, thanks to Lemma \ref{wn1} we know that for almost all $0<t<T,$ the problem
\begin{equation*}
\mathrm{div}(\boldsymbol{\mathrm{grad}}\,\pi(t)-\boldsymbol{f}(t))=0,\qquad\textrm{in}\,\Omega,\qquad
(\boldsymbol{\mathrm{grad}}\,\pi(t)-\boldsymbol{f}(t))\cdot\boldsymbol{n}=0,\qquad\textrm{on}\,\Gamma,
\end{equation*}
has a unique solution $\pi(t)\in W^{1,p}(\Omega)/\mathbb{R}$ that
satisfies the estimate
\begin{equation*}
\textrm{for}\,\,\textrm{a.e.}\,\,t\in(0,T)\qquad\|\pi(t)\|_{W^{1,p}(\Omega)/\mathbb{R}}\,\leq\,C(\Omega)\|\boldsymbol{f}(t)\|_{\boldsymbol{L}^{p}(\Omega)}.
\end{equation*}
It is clear that $\pi\in L^{q}(0,T;\,W^{1,p}(\Omega)/\mathbb{R})$ and $(\boldsymbol{f}-\boldsymbol{\mathrm{grad}}\,\pi)\in
L^{q}(0,T;\,\boldsymbol{L}^{p}_{\sigma,\tau}(\Omega))$. As a result, thanks to Theorem \ref{existinhensp},  Problem (\ref{lens}) with (\ref{nbc}) has a unique solution $(\boldsymbol{u},\pi)$ satisfying (\ref{reglplqstokes1})-(\ref{estlplqstokes}).
\end{proof}
Using now the estimate \eqref{pureimphdiv}, the $\zeta$-convexity of $[\boldsymbol{H}^{p'}_{0}(\mathrm{div,\Omega})]'$ (see Proposition \ref{Hpdivtp'zetaconx}), the Lemma \ref{wn1} (part (ii) concerning the weak Neumann Problem) and proceeding in a similar way as in the proof of Theorem \ref{existinhensp} and Theorem \ref{Exisinhnsplp}, one gets the weak solutions for the inhomogeneous Stokes problem. We will skip the proof because it is similar to the proof of Theorems \ref{existinhensp}-\ref{Exisinhnsplp}.
\begin{theo}[Weak Solutions for the inhomogeneous Stokes Problem]\label{Existinhsphdiv}
Let $1<p,q<\infty$, 
$\boldsymbol{u}_{0}=0$ and let $\boldsymbol{f}\in L^{q}(0,T;\,[\boldsymbol{H}^{p'}_{0}(\mathrm{div},\Omega)]')$, $0<T\leq\infty$. The Problem  (\ref{lens}) with (\ref{nbc}) has a unique solution $(\boldsymbol{u},\pi)$ satisfying
\begin{equation*}
\boldsymbol{u}\in L^{q}(0,T_{0};\,\,\boldsymbol{W}^{1,p}(\Omega)),\,\,\,\,  T_{0}\leq T\,\,\,\textrm{if}\,\,\,T<\infty\,\,\,\, \mathrm{and }\,\,\,\,T_{0}<T\,\,\,\textrm{if}\,\,\,T=\infty,
\end{equation*}
\begin{equation*}
\pi\in L^{q}(0,T;\,\,L^{p}(\Omega)/\mathbb{R}),\qquad\frac{\partial\boldsymbol{u}}{\partial t}\in L^{q}(0,T;\,\in[\boldsymbol{H}^{p'}_{0}(\mathrm{div}\Omega)]'_{\sigma,T})
\end{equation*}
and
\begin{multline*}
\int_{0}^{T}\Big\Vert\frac{\partial\boldsymbol{u}}{\partial t}\Big\Vert^{q}_{[\boldsymbol{H}^{p'}_{0}(\mathrm{div}\Omega)]'}\,\mathrm{d}\,t\,+\,\int_{0}^{T}\Vert\Delta\boldsymbol{u}(t)\Vert^{q}_{[\boldsymbol{H}^{p'}_{0}(\mathrm{div}\Omega)]'}\,\mathrm{d}\,t\,+\,\int_{0}^{T}\Vert\pi(t)\Vert^{q}_{L^{p}(\Omega)/\mathbb{R}}\,\mathrm{d}\,t\\
\leq\,C(p,q,\Omega)\,\int_{0}^{T}\Vert\boldsymbol{f}(t)\Vert^{q}_{[\boldsymbol{H}^{p'}_{0}(\mathrm{div}\Omega)]'}\,\mathrm{d}\,t.
\end{multline*}
 \end{theo}
 
 \medskip
 
 Similarly using the estimate \eqref{pureimptp}, the $\zeta$-convexity of $[\boldsymbol{T}^{p'}(\Omega)]'$ (see Proposition \ref{Hpdivtp'zetaconx}) and the Lemma \ref{wn1} (part (iii)), one has the very weak solution for the inhomogeneous Stokes problem. The proof of the following theorem is similar to the proof of Theorems \ref{existinhensp}-\ref{Exisinhnsplp}.
\begin{theo}[Very weak solutions for the inhomogeneous Stokes Problem]\label{Existinhsptp} Let $T\in (0, \infty]$, $1<p,q<\infty$,  $\boldsymbol{u}_{0}=0$ and $\boldsymbol{f}\in L^{q}(0,T;\,[\boldsymbol{T}^{p'}(\Omega)]')$. Then the time dependent Stokes Problem (\ref{lens}) with the boundary condition (\ref{nbc}) has a unique solution $(\boldsymbol{u},\pi)$ satisfying
\begin{equation*}
\boldsymbol{u}\in L^{q}(0,T_{0};\,\boldsymbol{L}^{p}(\Omega)),\,\,\,\,  T_{0}\leq T\,\,\,\textrm{if}\,\,\,T<\infty\,\,\,\, \mathrm{and }\,\,\,\,T_{0}<T\,\,\,\textrm{if}\,\,\,T=\infty,
\end{equation*}  
\begin{equation*}
\pi\in L^{q}(0,T;\,\,W^{-1,p}(\Omega)/\mathbb{R}),\qquad\frac{\partial\boldsymbol{u}}{\partial t}\in L^{q}(0,T;\,\in[\boldsymbol{T}^{p'}(\Omega)]'_{\sigma,\tau})
\end{equation*}
and
\begin{multline*}
\int_{0}^{T}\Big\Vert\frac{\partial\boldsymbol{u}}{\partial t}\Big\Vert^{q}_{[\boldsymbol{T}^{p'}(\Omega)]'}\,\mathrm{d}\,t\,+\,\int_{0}^{T}\Vert\Delta\boldsymbol{u}(t)\Vert^{q}_{[\boldsymbol{T}^{p'}(\Omega)]'}\,\mathrm{d}\,t\,+\,\int_{0}^{T}\Vert\pi(t)\Vert^{q}_{W^{-1,p}(\Omega)/\mathbb{R}}\,\mathrm{d}\,t\\
\leq\,C(p,q,\Omega)\,\int_{0}^{T}\Vert\boldsymbol{f}(t)\Vert^{q}_{[\boldsymbol{T}^{p'}(\Omega)]'}\,\mathrm{d}\,t.
\end{multline*}
 \end{theo}

We treat now the Stokes problem with flux condition (\ref{lens}), (\ref{nbc}), \eqref{condition2}. 

Since the system  (\ref{inhensp})--(\ref{condition2})  is equivalent to the Stokes Problem with flux condition (\ref{lens}), (\ref{nbc}), \eqref{condition2}, we deduce in that way the existence and maximal regularity of strong, weak and very weak solution for the Stokes Problem with flux condition (\ref{lens}), (\ref{nbc}), \eqref{condition2}. 
\begin{theo}[Strong Solutions for the inhomogeneous Stokes Problem with flux]
\label{Theorem10}
Let $T\in (0, \infty]$, $1<p,q<\infty$.  For all $\boldsymbol{f}\in L^{q}(0,T;\,\boldsymbol{X}_{p})$, there exists a unique solution $\boldsymbol{u}$ of (\ref{inhensp}) such  that 
\begin{equation}\label{reglplqstokes1}
\boldsymbol{u}\in L^{q}(0,T_{0};\,\mathbf{D}(A'_{p})),\,\,\,\,  T_{0}\leq T\,\,\,\textrm{if}\,\,\,T<\infty\,\,\,\, \mathrm{and }\,\,\,\,T_{0}<T\,\,\,\textrm{if}\,\,\,T=\infty,
\end{equation}
\begin{equation}\label{reglplqstokes2}
\frac{\partial\boldsymbol{u}}{\partial t}\in L^{q}(0,T;\,\boldsymbol{X}_{p})
\end{equation}
\begin{equation}\label{estlplqstokes}
\int_{0}^{T}\Big\Vert\frac{\partial\boldsymbol{u}}{\partial t}\Big\Vert^{q}_{\boldsymbol{X}_{p}}\,\mathrm{d}\,t\,+\,\int_{0}^{T}\Vert\Delta\boldsymbol{u}(t)\Vert^{q}_{\boldsymbol{X}_{p}}\,\mathrm{d}\,t
\leq C(p,q,\Omega)\,\int_{0}^{T}\Vert\boldsymbol{f}(t)\Vert^{q}_{\boldsymbol{L}^{p}(\Omega)}\,\mathrm{d}\,t.
\end{equation}
and such that $(\boldsymbol{u}, \pi )$ is a solution of the  inhomogeneous Stokes Problem (\ref{lens}), (\ref{nbc}) , (\ref{condition2})  for all $\pi \in \R$.
\end{theo}
\begin{proof}
The space $\boldsymbol{X}_{p}$ is $\zeta$-convex, and  by  Theorem \ref{Lapimpower} the pure imaginary powers of the operators $A'_{p}$ are bounded in $\boldsymbol{X}_{p}$.  It is then possible to apply  Theorem \ref{existabscp} to the operator $A'_{p}$ itself  in $\boldsymbol{X}_{p}$ and Theorem  \ref{Theorem10} follows.\end{proof}
\begin{rmk}
\label{rmk10}
\rm{The spaces $\boldsymbol{Y}_p$ and $\boldsymbol{Z}_p$  are  also $\zeta$-convex, and  by  Theorem \ref{Lapimpower} the pure imaginary powers of the operators $B'_{p}$ and  $C'_{p}$ are bounded in $\boldsymbol{Y}_p$ and $\boldsymbol{Z}_p$ respectively.  It is then possible to apply Theorem \ref{existabscp} to the operator $B'_{p}$ and  $C'_{p}$ in $\boldsymbol{Y}_p$ and $\boldsymbol{Z}_p$.   We obtain in this way the existence, uniqueness and maxial regularity of weak and very weak solutions for the Stokes problem with flux condition  (\ref{lens}), (\ref{nbc}) , (\ref{condition2}). The corresponding Theorems are very similar to Theorem  \ref{Theorem10} and we do not write their statements in detail.}
\end{rmk}

\noindent
\textbf{Acknowledgements} The work of M. E. has been supported by DGES Grant  MTM2011-29306-C02-00 and Basque Government Grant IT641-13.
The authors wish to thank the referees for their helpful remarks. They are particularly grateful for their comments and suggestions on Section \ref{Imaginary} and Proposition \ref{densitédurang}.

 \bigskip
\noindent
\parbox[t]{.48\textwidth}{
Hind Al Baba\\
Laboratoire de Mathématiques\\ 
et de leurs applications
Pau,\\
UMR, CNRS 5142,
Batiment IPRA,\\ 
Université de Pau et des pays de L'Adour,\\ Avenue de L'université,\\ Bureau 012, BP 1155,\\
 64013 Pau cedex, France\\
hind.albaba@univ-pau.fr } \hfill
\parbox[t]{.48\textwidth}{
Ch\'erif Amrouche\\
Laboratoire de Mathématiques\\ 
et de leurs applications
Pau,\\
UMR, CNRS 5142,
Batiment IPRA,\\ 
Université de Pau et des pays de L'Adour,\\ Avenue de L'université,\\ Bureau 225, BP 1155,\\
 64013 Pau cedex, France\\
cherif.amrouche@univ-pau.fr
}

\bigskip
\noindent
\parbox[t]{.48\textwidth}{
Miguel Escobedo\\
Departamento de Matemáticas\\
Facultad de Ciencias y Tecnología\\
Universidad del País Vasco\\
Barrio Sarriena s/n,\\ 48940 Lejona (Vizcaya), Spain\\
miguel.escobedo@ehu.es
}
\end{document}